\documentclass[a4paper,11pt,reqno,noindent]{amsart}
\usepackage[centertags]{amsmath}
\usepackage{amsfonts,amssymb,amsthm,dsfont,cases,amscd,esint,enumerate}
\usepackage[T1]{fontenc}
\usepackage[english]{babel}
\usepackage[applemac]{inputenc}
\usepackage{newlfont}
\usepackage{color}
\usepackage[body={15cm,21.5cm},centering]{geometry} 
\usepackage{fancyhdr}
\usepackage{enumerate}

\theoremstyle{plain}
\newtheorem{theor10}{Theorem}
\newenvironment{theor1}
  {\pushQED{\qed}\begin{theor10}}
  {\popQED\end{theor10}}
\newtheorem{conj10}[theor10]{Conjecture}
\newenvironment{conj1}
  {\pushQED{\qed}\begin{conj10}}
  {\popQED\end{conj10}}
\newtheorem{theor0}{Theorem}[section]
\newenvironment{theor}
  {\pushQED{\qed}\begin{theor0}}
  {\popQED\end{theor0}}
\newtheorem{lem0}[theor0]{Lemma}
\newenvironment{lem}
  {\pushQED{\qed}\begin{lem0}}
  {\popQED\end{lem0}}
\newtheorem{prop0}[theor0]{Proposition}
\newenvironment{prop}
  {\pushQED{\qed}\begin{prop0}}
  {\popQED\end{prop0}}
\newtheorem{cor0}[theor0]{Corollary}
\newenvironment{cor}
  {\pushQED{\qed}\begin{cor0}}
  {\popQED\end{cor0}}
\newtheorem{defin0}[theor0]{Definition}
\newenvironment{defin}
  {\pushQED{\qed}\begin{defin0}}
  {\popQED\end{defin0}}
\theoremstyle{definition}
\newtheorem{rems0}[theor0]{Remarks}
\newenvironment{rems}
  {\pushQED{\qed}\begin{rems0}}
  {\popQED\end{rems0}}
\newtheorem{rem0}[theor0]{Remark}
\newenvironment{rem}
  {\pushQED{\qed}\begin{rem0}}
  {\popQED\end{rem0}}

\numberwithin{equation}{section}

\newcommand{\N}{\mathbb N}
\newcommand{\R}{\mathbb R}
\newcommand{\Z}{\mathbb Z}
\newcommand{\T}{\mathbb T}
\newcommand{\e}{\varepsilon}
\newcommand{\Log}{|\!\log\e|}
\newcommand{\st}{{\operatorname{st}}}
\newcommand{\Lc}{\mathcal{L}}
\newcommand{\calB}{\mathcal{B}}
\newcommand{\Dm}{\mathbb{D}}
\newcommand{\calC}{\mathcal{C}}
\newcommand{\calK}{\mathcal{K}}

\newcommand{\Oc}{\mathcal U}
\newcommand{\Ac}{\mathcal A}
\newcommand{\Fc}{\mathcal F}

\newcommand{\Sc}{\mathcal S}
\newcommand{\Rc}{\mathcal R}
\newcommand{\calR}{\mathcal R}

\newcommand{\calH}{\mathcal H}
\newcommand{\Md}{\mathbb M}
\newcommand{\loc}{{\operatorname{loc}}}
\newcommand{\per}{{\operatorname{per}}}
\newcommand{\Id}{\operatorname{Id}}
\newcommand{\E}{\mathbb{E}}
\newcommand{\ee}{e}
\newcommand{\Aa}{\boldsymbol a}
\newcommand{\Bb}{\boldsymbol b}
\newcommand{\Cc}{\boldsymbol c}
\newcommand{\Ld}{\operatorname{L}}
\newcommand{\step}[1]{\noindent \textit{Step} #1.}

\newcommand{\Pm}{\mathbb{P}}
\newcommand{\expec}[1]{\mathbb{E}\left[ #1 \right]}
\newcommand{\expecm}[1]{\mathbb{E}\big[ #1 \big]}
\newcommand{\expecmm}[1]{\mathbb{E}\,[ #1 ]}
\newcommand{\expecM}[1]{\mathbb{E}\bigg[ #1 \bigg]}
\newcommand{\var}[1]{\mathrm{Var}\left[#1\right]}

\newcommand{\cov}[2]{\operatorname{Cov}\left[{#1};{#2}\right]}
\newcommand{\covm}[2]{\operatorname{Cov}\big[{#1};{#2}\big]}
\newcommand{\expeC}[2]{\mathbb{E}\left[\left. #1 \,\right\|\,#2\right]}
\newcommand{\expeCm}[2]{\mathbb{E}\big[ #1 \,\big\|\,#2\big]}

\usepackage[colorlinks,linkcolor=black,citecolor=black,urlcolor=black]{hyperref}

\title[Non-perturbative approach to the Bourgain--Spencer conjecture]
{Non-perturbative approach to the Bourgain--Spencer conjecture in stochastic homogenization}

\author[M. Duerinckx]{Mitia Duerinckx}
\address[Mitia Duerinckx]{Universit\'e Libre de Bruxelles, Département de Mathématique, 1050 Brussels, Belgium}
\email{mitia.duerinckx@ulb.be}

\begin{document}
\selectlanguage{english}

\begin{abstract}
In the context of stochastic homogenization, the Bourgain--Spencer conjecture states that the ensemble-averaged solution of a divergence-form linear elliptic equation with random coefficients admits an intrinsic description in terms of higher-order homogenized equations with an accuracy four times better than the almost sure solution itself. While previous rigorous results were restricted to a perturbative regime with small ellipticity ratio, we make the very first progress in a non-perturbative setting, establishing half of the conjectured optimal accuracy. The validity of the full conjecture remains an open question and might in fact fail in general. Our approach involves the construction of a new corrector theory in stochastic homogenization: while only a bounded number of correctors can be constructed as stationary $\Ld^2$ random fields, we show that twice as many stationary correctors can be defined in a Schwartz-like distributional sense on the probability space. We focus on the Gaussian setting for the coefficient field, and the proof relies heavily on Malliavin calculus.

\bigskip\noindent
{\sc MSC-class:}
35J15, 35B27, 60H25, 60H07, 60H30, 46F29.
\end{abstract}

\maketitle
\setcounter{tocdepth}{1}
\tableofcontents

\section{Introduction}\label{sec:intro}

Let $\Aa$ be a stationary and ergodic random coefficient field on $\R^d$ with values in $\R^{d\times d}$, with the following boundedness and ellipticity properties, for some $\lambda>0$,
\begin{equation}\label{eq:elliptic}
|\Aa(x)\ee|\le|\ee|,\qquad\ee\cdot\Aa(x)\ee\ge\lambda|\ee|^2,\qquad\text{for all $x,\ee\in\mathbb{R}^d$},
\end{equation}
and denote by $(\Omega,\Pm)$ the underlying probability space. In the sequel, we further assume that $\Aa$ satisfies some strong mixing condition, and we focus for simplicity on the Gaussian setting, see~Section~\ref{sec:as}, in which case we can take advantage of the powerful tools of Malliavin calculus.
Given a deterministic vector field $f\in C^\infty_c(\R^d)^d$, we consider the random family $\{\nabla u_{\e,f}\}_{\e>0}$ of unique almost sure gradient solutions in $\Ld^2(\R^d)^d$ of the rescaled Poisson problems
\begin{align}\label{eq:first-def-ups}
-\nabla\cdot \Aa(\tfrac\cdot\e)\nabla u_{\e,f}\,=\, \nabla\cdot f,\qquad\text{in $\R^d$},
\end{align}
where the rescaled coefficient field $\Aa(\tfrac\cdot\e)$ varies on the microscopic scale $\e>0$.
While stochastic homogenization theory has been focussing on the intrinsic description of spatial oscillations and random fluctuations of the random solution $\nabla u_{\e,f}$ in the macroscopic limit~$\e\downarrow0$,
the present work is rather concerned with the effective description of the ensemble-averaged solution $\expec{\nabla u_{\e,f}}$.

Before describing the Bourgain--Spencer conjecture and our results, in order to contextualize them properly, we start with a review of the state of the art in higher-order stochastic homogenization.
First, the qualitative theory~\cite{PapaVara,Kozlov-79} states the almost sure weak convergence of fields \mbox{$\nabla u_{\e,f}\rightharpoonup\nabla\bar u^1_f$} and fluxes $\Aa(\tfrac\cdot\e)\nabla u_{\e,f}\rightharpoonup\bar\Aa^1\nabla\bar u^1_f$ in $\Ld^2(\R^d)^d$,
where $\nabla\bar u^1_f$ is the unique gradient solution in $\Ld^2(\R^d)^d$ of the (deterministic) homogenized equation
\begin{equation}\label{eq:first-def-ubar}
-\nabla\cdot\bar\Aa^1\nabla\bar u^1_f=\nabla\cdot f,\qquad\text{in $\R^d$},
\end{equation}
where the effective coefficients $\bar\Aa^1\in\R^{d\times d}$ are given in each direction $\ee_i$, $1\le i\le d$, by
\[\bar\Aa^1\ee_i=\expec{\Aa(\nabla\varphi_i^1+\ee_i)},\]
in terms of the so-called corrector gradient $\nabla\varphi_i^1$ in the direction $\ee_i$, which is defined as the unique almost sure gradient solution in $\Ld^2_\loc(\R^d)^d$ of
\begin{equation}\label{eq:cor1}
-\nabla\cdot\Aa(\nabla\varphi_i^1+\ee_i)=0,\qquad\text{in $\R^d$,}
\end{equation}
such that $\nabla\varphi_i^1$ is a stationary field with vanishing expectation and with bounded second moments.
This convergence result for fields and fluxes is understood as follows: the heterogeneous field-flux constitutive law $\nabla w_\e\mapsto\Aa(\tfrac\cdot\e)\nabla w_\e$ is replaced on large scales by the effective law $\nabla\bar w\mapsto\bar\Aa^1\nabla\bar w$.

As the coefficient field $\Aa(\tfrac\cdot\e)$ is oscillating on scale $O(\e)$, fields and fluxes display oscillations on that scale too, which is why their convergence is only weak in~$\Ld^2(\R^d)^d$. Homogenization theory has aimed at describing these oscillations intrinsically by means of two-scale expansions.
The key ingredient is the corrector $\varphi^1=\{\varphi^1_i\}_{1\le i\le d}$, which is defined in~\eqref{eq:cor1} up to an additive random constant and is almost surely sublinear at infinity.
This corrector is interpreted as correcting Euclidean coordinates $\{x\mapsto x_i\}_{1\le i\le d}$ into $\Aa$-harmonic ones \mbox{$\{x\mapsto x_i+\varphi_i^1(x)\}_{1\le i\le  d}$}. The following so-called two-scale expansion
\begin{equation}\label{eq:2sc}
\Fc_\e^1[\bar u^1_f]\,:=\,\bar u^1_f+\e\varphi_i^1(\tfrac\cdot\e)\nabla_i\bar u^1_f
\end{equation}
is then viewed as an intrinsic Taylor expansion of the limiting profile $\bar u^1_f$ in $\Aa$-harmonic coordinates.\footnote{Throughout, we use Einstein's summation convention on repeated indices.} The standard corrector result~\cite{PapaVara,Kozlov-79} ensures that this expansion indeed captures leading-order oscillations in the sense that almost surely
\[\nabla u_{\e,f}-\nabla \Fc_\e^1[\bar u^1_f]\,\to\,0,\qquad\text{strongly in $\Ld^2(\R^d)^d$.}\]
In the last decade, many works have managed to further establish sharp convergence rates, which mainly require to investigate quantitatively the sublinearity of the corrector~$\varphi^1$; see~\cite{GNO2,GNO-quant,AKM2,AKM-book}.
In dimension $d>2$, the corrector can be chosen itself as a stationary field with bounded moments: it is then uniquely defined up to an additive {\it deterministic} constant, which is fixed for instance by choosing $\expec{\varphi^1}=0$. In dimension $d\le2$, the corrector cannot be chosen stationary as it has some nontrivial (sublinear) growth at infinity.
Based on optimal corrector estimates, it is deduced that for all $q<\infty$,
\[\expec{\|\nabla u_{\e,f}-\nabla \Fc_\e^1[\bar u^1_f]\|_{\Ld^2(\R^d)}^q}^\frac1q\,\lesssim_{f,q}\,\e\times\left\{\begin{array}{lll}
1&:&\text{$d>2$},\\
\Log^\frac12&:&\text{$d=2$},\\
\e^{-\frac12}&:&\text{$d=1$}.
\end{array}\right.\]

Next, we recall how this intrinsic description of oscillations is pursued to higher orders.
While the first corrector $\varphi^1$ is defined to correct Euclidean coordinates into $\Aa$-harmonic ones, higher-order correctors are defined iteratively to correct higher-order polynomials and make them adapted to the heterogeneous elliptic operator $\nabla\cdot\Aa\nabla$. More precisely, assuming that correctors $\varphi^1,\ldots,\varphi^{n-1}$ are defined as stationary fields, say with vanishing expectation,
the next-order corrector~$\varphi^{n}$ is uniquely defined up to an additive random constant by the following properties:
\begin{enumerate}[$\bullet$]
\item for any $n$th-order polynomial $\bar q$, the corrected polynomial $\Fc^n[\bar q]:=\sum_{k=0}^n\varphi^k_{i_1\ldots i_k}\nabla^k_{i_1\ldots i_k}\bar q$ captures oscillations of the operator in the sense that $\nabla\cdot\Aa\nabla \Fc^n[\bar q]$ is deterministic;
\smallskip\item the gradient $\nabla\varphi^n$ is a stationary field with vanishing expectation and with bounded second moments.
\end{enumerate}
In fact, this definition implies that
$\nabla\cdot\Aa\nabla \Fc^n[\bar q]=\nabla\cdot(\sum_{k=1}^{n-1}\bar\Aa^k\nabla^{k-1})\nabla\bar q$
for some higher-order effective tensors~$\bar\Aa^1,\ldots,\bar\Aa^{n-1}$.
As recalled in Section~\ref{sec:cor-th} below, higher-order correctors are alternatively defined by a hierarchy of cell problems, and higher-order tensors are given by
\begin{equation}\label{eq:def-barak}
\bar\Aa^k_{i_1\ldots i_{k-1}}\ee_{i_k}\,=\,\expecm{\Aa\big(\nabla\varphi^k_{i_1\ldots i_k}+\varphi^{k-1}_{i_1\ldots i_{k-1}}\ee_{i_k}\big)}.
\end{equation}
These tensors define corrections to the effective field-flux constitutive law: while the heterogeneous law $\nabla w_\e\mapsto\Aa(\frac\cdot\e)\nabla w_\e$ is replaced on large scales by the effective law $\nabla\bar w\mapsto\bar\Aa^1\nabla\bar w$ at leading order, dispersive corrections must be included when looking for finer accuracy,
\begin{equation*}
\nabla\bar w\,\mapsto\,\bar\Ac_\e^n\nabla\bar w\,:=\,\sum_{k=1}^n\e^{k-1}\bar\Aa_{i_1\ldots i_{k-1}}^k\nabla\nabla^{k-1}_{i_1\ldots i_{k-1}}\bar w.
\end{equation*}
To order $n$, the homogenized equation~\eqref{eq:first-def-ubar} is then formally replaced by
\begin{equation}\label{eq:homog-high}
-\nabla\cdot\bar\Ac_\e^n\nabla\bar U_{\e,f}^n\,=\,\nabla\cdot f,\qquad\text{in $\R^d$}.
\end{equation}
This equation is however ill-posed in general for $n\ge2$, as the symbol may be indefinite, and we shall use a suitable proxy $\nabla\bar\Oc^n_\e[f]$ for $\nabla\bar U_{\e,f}^n$, see~Definition~\ref{def:high-eq}.
In these terms, the two-scale expansion~\eqref{eq:2sc} is replaced by its higher-order version
\begin{equation}\label{eq:2sc-high}
\Fc_\e^n[ \bar\Oc_\e^n[f]]\,:=\,\sum_{k=0}^n\e^{k}\varphi^k_{i_1\ldots i_k}(\tfrac\cdot\e)\nabla^k_{i_1\ldots i_k}\bar\Oc_\e^n[f].
\end{equation}
In case of a periodic coefficient field $\Aa$, all correctors can be constructed as bounded periodic fields, and the approximation $\nabla u_{\e,f}-\nabla \Fc_\e^n[\bar\Oc_\e^n[f]]=O(\e^n)$ can be justified to all orders~$n\ge1$.
In contrast, a key specificity of the random setting is that only a finite number of correctors can be defined as stationary fields.
More precisely, as shown in~\cite{Gu-17,BFFO-17,DO1}, the corrector $\varphi^n$ can be chosen stationary with bounded moments only for~$n<\ell:=\lceil\frac d2\rceil$.
Although higher-order correctors could still be defined, their lack of stationarity and their spatial growth would make them unsuitable for error estimates.
The accuracy of the two-scale expansion~\eqref{eq:2sc-high} is then shown to saturate at order $O(\e^{d/2})$: for all~$q<\infty$,
\begin{equation}\label{eq:2sc-exp-err-res}
\expec{\|\nabla u_{\e,f}- \nabla\Fc_\e^\ell[\bar\Oc_\e^\ell [f]]\|_{\Ld^2(\R^d)}^q}^\frac1q\,\lesssim_{f,q}\,\e^\frac d2\times\left\{\begin{array}{lll}
1&:&\text{$d$ odd},\\
\Log^\frac12&:&\text{$d$ even}.
\end{array}\right.
\end{equation}
This accuracy is optimal as random fluctuations are known to become dominant at that order~$O(\e^{d/2})$; see~\cite{DGO1,DO1}.

In this context, the present work is devoted to the effective description of the ensemble-averaged field $\expec{\nabla u_{\e,f}}$ and flux $\expec{\Aa(\tfrac\cdot\e)\nabla u_{\e,f}}$ with optimal accuracy. Although relevant for applications, see~\cite[Section~4.2]{DGL}, this question has only been addressed very recently by Sigal~\cite{Sigal}.
A naïve estimate follows by taking the expectation in~\eqref{eq:2sc-exp-err-res},
which gives
\begin{equation}\label{eq:bnd-2sc-exp}
\|\expec{\nabla u_{\e,f}}-\nabla\bar\Oc^\ell_\e [f]\|_{\Ld^2(\R^d)}\,\lesssim_{f}\,\e^\frac d2\times\left\{\begin{array}{lll}
1&:&d~\text{odd},\\
\Log^\frac12&:&d~\text{even}.
\end{array}\right.
\end{equation}
While the limitation to accuracy $O(\e^{d/2})$ in~\eqref{eq:2sc-exp-err-res} is related to random fluctuations, this optimality argument fails for~\eqref{eq:bnd-2sc-exp}, thus indicating that this effective description of the ensemble-averaged field might well be pursued to higher order.
Following preliminary calculations by Sigal~\cite{Sigal} and Spencer, a striking work by Bourgain~\cite{Bourgain-18} and its refinement by Kim and Lemm~\cite{Lemm-18} have suggested that the expansion might in fact be pursued to order~$O(\e^{2d-})$, see~\cite{DGL}.
More precisely, the following so-called Bourgain--Spencer conjecture was formulated.
(The case of dimension $d=1$ is naturally omitted since a direct computation simply yields $\expec{\nabla u_{\e,f}}=\nabla\bar u_f^1$ in that case.)

\begin{conj1}[Bourgain, Spencer]\label{conj}
Let $d>1$.
There exists a collection $\{\bar\Aa^n\}_{1\le n\le2d}$ of constant tensors such that for all $f\in C^\infty_c(\R^d)^d$ and $\eta>0$,
\[\|\expec{\nabla u_{\e,f}}-\nabla\bar\Oc_\e^{2d}[f]\|_{\Ld^2(\R^d)}\,\lesssim_{f,\eta}\,\e^{2d-\eta},\]
where $\nabla\bar\Oc_\e^{2d}[f]$ is a proxy for the solution of the higher-order homogenized equation~\eqref{eq:homog-high} with coefficients $\{\bar\Aa^n\}_{1\le n\le2d}$.
\end{conj1}

In the perturbative regime when the random coefficient field $\Aa$ has small ellipticity ratio~$|\Aa-\Id\!|\ll1$,
Bourgain's work~\cite{Bourgain-18} and its refinement by Kim and Lemm~\cite{Lemm-18} entail that this conjecture holds to order
\begin{equation}\label{eq:Bourgain}
O\Big(\e^{2d-C\|\Aa-\Id\|_{\Ld^\infty(\R^d)}}\Big),
\end{equation}
that is, with a loss $\eta=C\|\Aa-\Id\|_{\Ld^\infty(\R^d)}$ proportional to the ellipticity ratio. This result relies on a fine analysis of perturbative expansions at small ellipticity ratio.
It is obtained in the model framework of discrete elliptic equations with iid coefficients, but can be extended to the present continuum setting (which we postpone to a future work).
This result was originally expressed in terms of harmonic analysis: as explained in~\cite{DGL}, there exists a uniformly elliptic matrix-valued kernel $\calB$ such that the ensemble-averaged field satisfies
\[-\nabla\cdot \calB(\e\nabla)\expec{\nabla u_{\e,f}}=\nabla\cdot f,\]
and Sigal originally proposed to investigate the $C^{\alpha}$ regularity of the kernel $\calB$ at small wavenumbers,
which turns out to be indeed equivalent to an effective description of $\expec{\nabla u_{\e,f}}$ to order~$O(\e^{\alpha})$ as in the above conjecture.

As the loss proportional to the ellipticity ratio in~\eqref{eq:Bourgain} seems unavoidable in the perturbative approach by Bourgain and followers~\cite{Bourgain-18,Lemm-18},
it is unclear whether the above conjecture should really be expected to hold to order $O(\e^{2d-})$ in general as stated.
In the present work, we take a different perspective and
develop non-perturbative quantitative homogenization techniques to investigate the conjecture away from the perturbative regime $|\Aa-\Id|\ll1$.
Our main result proves the general validity of ``half'' of the conjecture, that is, the accuracy of an effective description to order $O(\e^{d-})$, see~Theorem~\ref{th:main} below. Note that this is twice better than the naïve estimate~\eqref{eq:bnd-2sc-exp}.
Our argument splits into two parts:
\begin{enumerate}[(i)]
\item \emph{Weak corrector theory:}\\
Although higher-order correctors $\varphi^n$ with $n\ge\ell=\lceil\frac d2\rceil$ cannot be constructed as stationary fields with bounded moments, we show that stationary correctors can be constructed up to order~\mbox{$n<d$} in a Schwartz-like distributional sense on the probability space. More precisely, for~$n<d$, weak-type expressions of the form~$\expecm{X\varphi^{n}(x)}$ can be meaningfully defined for all ``smooth and local'' test random variables~$X$, in a way that is compatible with corrector equations and stationarity. At order $n=d$, only the corrector gradient $\nabla\varphi^d$ can be defined as a stationary object in the same weak sense.
Denoting by $\Sc(\Omega)$ the space of suitable test random variables, correctors $\{\varphi^n\}_{1\le n\le d}$ are defined as (smooth) maps $\R^d\to\Sc'(\Omega)$ with values in the space of linear functionals on $\Sc(\Omega)$.
The Gelfand triple $\Sc(\Omega)\subset\Ld^2(\Omega)\subset\Sc'(\Omega)$ is viewed as a stochastic version of the usual triple $\Sc(\R^d)\subset\Ld^2(\R^d)\subset\Sc'(\R^d)$ of Schwartz test functions and distributions.
We refer to Section~\ref{sec:weak-cor0} for details.
\smallskip\item \emph{Weak two-scale expansions:}\\
In terms of these weak correctors $\{\varphi^n\}_{1\le n\le d}$, we may construct as in~\eqref{eq:2sc-high} the corresponding two-scale expansion $\Fc_\e^d[\bar\Oc_\e^d[f]]$ to order $d$.
The difficulty is that $\nabla\Fc_\e^d[\bar\Oc_\e^d[f]]$ is only defined in the weak sense of $\Sc'(\Omega;\Ld^2(\R^d))$ and can in particular not be used in any energy estimate to get error bounds as is usually done in the classical theory~\cite{Gu-17,DO1}. By duality, we however manage to prove the accuracy of this two-scale expansion in a suitable weak sense, and we deduce the validity of Conjecture~\ref{conj} to order~$O(\e^{d-})$.
We refer to Section~\ref{sec:BSconj0} for details.
\end{enumerate}
Such a description of the field $\nabla u_{\e,f}$ in a distributional sense on the probability space was originally inspired by our work with Shirley~\cite{DS1,DS2} on random Schrödinger operators, and it constitutes a new type of results in the field.
Note that spaces of Schwartz-like distributions on the probability space first appeared in the works of Kondratiev and Hida in the context of SPDEs, see e.g.~\cite{HOUZ-96} and references therein, but our definition is quite different and is closer to~\cite{AKL-16}, see~Section~\ref{sec:R'om}.

\subsection*{Notation}
\begin{enumerate}[\quad$\bullet$]
\item We denote by $C\ge1$ any constant that only depends on the dimension $d$, on the ellipticity constant $\lambda$ in~\eqref{eq:elliptic}, as well as on $\|a_0\|_{W^{1,\infty}}$ and $\int_{\R^d}[\calC_0]_\infty$ in~\eqref{eq:def-A} and~\eqref{eq:cov-L1} below. We use the notation~$\lesssim$ (resp. $\gtrsim$) for $\le C\times$ (resp. $\ge\frac1C\times$) up to such a multiplicative constant~$C$.
We write~$\simeq$ when both $\lesssim$ and $\gtrsim$ hold. The notation $\ll$ stands for $\le\frac1C\times$ for some large enough constant $C$.
We add subscripts to $C$, $\lesssim$, $\gtrsim$, $\simeq$, $\ll$ to indicate dependence on other parameters.
\smallskip\item For a matrix field $H=(H_{ij})_{1\le i,j\le d}$, we define its divergence $\nabla\cdot H$ as the vector field given by $(\nabla\cdot H)_i:=\nabla_jH_{ij}$, where we use Einstein's summation convention on repeated indices. For a vector field $h=(h_i)_{1\le i\le d}$, we define its curl as the skew-symmetric matrix field given by $(\nabla\times h)_{ij}:=\nabla_ih_j-\nabla_jh_i$.
\smallskip\item The ball centered at $x$ of radius $r$ in $\R^d$ is denoted by $B_r(x)$, and we simply write $B(x):=B_1(x)$, $B_r:=B_r(0)$, and $B:=B_1(0)$.
We also write $Q_r:=[-\frac r2,\frac r2)^d$ for the cube centered at the origin with side length $r$.
\smallskip\item For $r\in\R$ we denote by $\lceil r\rceil$ the smallest integer~$\ge r$, and we recall the notation~$\ell=\lceil\frac d2\rceil$. For $r,s\in\R$ we write $r\vee s:=\max\{r,s\}$ and $r\wedge s:=\min\{r,s\}$.
For $x\in\R^d$ we set $\langle x\rangle:=(1+|x|^2)^{1/2}$, and similarly $\langle\nabla\rangle:=(1-\triangle)^{1/2}$.
\smallskip\item For a function $g$ and for $1\le p<\infty$, we write $[g]_p(x):=(\fint_{B(x)}|g|^p)^{1/p}$ for the local moving $\Ld^p$ average, and similarly $[g]_\infty(x):=\sup_{B(x)}|g|$. For averages at the scale $\e$, we write $[g]_{p;\e}(x):=(\fint_{B_\e(x)}|g|^p)^{1/p}$.
\smallskip\item We often use cartesian products of functions and of operators to shorten the notation. Given two functions $h:E\to\R^n$ and $h':E\to\R^{n'}$, we let $(h,h')$ denote their cartesian product, that is, the function $E\to\R^{n+n'}$ given by $(h,h')(x):=(h(x),h'(x))$ for~$x\in E$.
Given two operators $T,T'$ on a Banach space~$\mathcal B$ of functions $E\to\R$, we similarly denote by $(T,T')$ the operator given by $(T,T')h:=(Th,T'h)$ for $h\in\mathcal B$.
\end{enumerate}

\section{Main results}
This section is devoted to the statement and discussion of our main results. For ease of reading, detailed statements are postponed to the proof sections.

\subsection{Assumptions}\label{sec:as}
We focus on the model Gaussian setting for the random coefficient field $\Aa$, in which case Malliavin calculus is available and substantially simplifies the analysis.
More precisely, we assume that the coefficient field takes the form
\begin{equation}\label{eq:def-A}
\Aa(x):=a_0(G(x)),
\end{equation}
where $a_0\in C^1_b(\R^\kappa)^{d\times d}$ is such that the ellipticity and boundedness assumptions~\eqref{eq:elliptic} are pointwise satisfied, and where $G:\R^d\times\Omega\to\R^\kappa$ is an $\R^\kappa$-valued centered stationary Gaussian random field on $\R^d$ with covariance function $\calC$, constructed on a probability space~$(\Omega,\Pm)$.\footnote{By centered stationary Gaussian random field, we mean the following: $G=\{G(x)\}_{x\in\R^d}$ is a collection of $\R^\kappa$-valued Gaussian random variables indexed by $\R^d$, with expectation $\expec{G(x)}=0$ and with covariance of the form $\expec{G(x)\otimes G(y)}=\calC(x-y)$, where~$\calC$ belongs to $\Ld^\infty(\R^d)^{\kappa\times\kappa}$ and satisfies $\calC(-x)=\calC(x)$.}
Note that we write $G(x)=G(x,\cdot):\Omega\to\R^\kappa$ for the evaluation at $x$.
In addition, we assume that~$G$ has integrable correlations in the following sense: starting from the representation
\begin{equation}\label{eq:G-rep}
G=\calC_0\ast\xi,
\end{equation}
where~$\xi$ is a $\kappa$-dimensional Gaussian white noise on $\R^d$ and where the {model}~$\calC_0:\R^d\to\R^{\kappa\times\kappa}$ satisfies $\calC_0\ast\calC_0=\calC$, we assume that $\calC_0$ satisfies the integrability condition
\begin{equation}\label{eq:cov-L1}
\int_{\R^d}[\calC_0]_\infty\,<\,\infty.
\end{equation}
In particular, this entails that the covariance function $\calC$ satisfies the same integrability condition $\int_{\R^d}[\calC]_\infty<\infty$.
Moreover, $\calC$ is necessarily continuous, so that $G$ and $\Aa$ are stochastically continuous and jointly measurable on $\R^d\times\Omega$.

Restricting to this Gaussian setting is essential as our approach relies on Malliavin calculus techniques.
Our analysis is easily repeated mutatis mutandis in the Poisson setting and in the iid discrete setting, using corresponding stochastic calculus tools that are available in those cases, e.g.~\cite{Peccati-Reitzner,D-20a}.
The general case of an $\alpha$-mixing coefficient field is more involved and is postponed to future work.
Regarding the integrability condition~\eqref{eq:cov-L1}, it could be easily relaxed: the case of non-integrable correlations could be treated similarly but would yield different scalings; see also~\cite{GNO-quant,DGO2,DFG1}.

Without loss of generality, we can assume that the probability space $(\Omega,\Pm)$ is endowed with the $\sigma$-algebra generated by the underlying white noise $\xi$ in~\eqref{eq:G-rep}.
This countably generated $\sigma$-algebra coincides with the one generated by the Gaussian field~$G$ under assumptions~\eqref{eq:G-rep}--\eqref{eq:cov-L1}.
We then consider the following model subspace of ``smooth and local'' random variables,
\begin{multline}\label{eq:defRom-0}
\Sc(\Omega)\,:=\,\bigg\{g\Big(\int_{\R^d}h_1\cdot\xi,\ldots,\int_{\R^d}h_n\cdot\xi\Big)~:~n\in\N,~g\in C^\infty_b(\R^n),\\~
h_1,\ldots,h_n\in C^\infty_c(\R^d)^\kappa\bigg\},
\end{multline}
and the choice of the $\sigma$-algebra ensures that this subspace is dense in $\Ld^q(\Omega)$ for all~$q<\infty$.
We denote by~$\Sc'(\Omega)$ the space of continuous linear functionals on $\Sc(\Omega)$, viewed as a space of Schwartz-like distributions on the probability space, see~Section~\ref{sec:R'om}.
For $X\in\Sc'(\Omega)$ and $Y\in\Sc(\Omega)$, we abusively use the notation $\expec{XY}=\langle X,Y\rangle_{\Sc'(\Omega),\Sc(\Omega)}$ for the duality product.

\subsection{Towards the Bourgain--Spencer conjecture}\label{sec:BSconj0}
The following main result justifies the effective description of the ensemble-averaged field and flux to order $O(\e^{d-})$, {thus providing a first nontrivial step towards Conjecture~\ref{conj}}.
We do not know whether this is optimal in general.
We refer to Theorem~\ref{th:main-2} for a more detailed statement.
In contrast with the $\Ld^2$-estimate stated in Conjecture~\ref{conj}, note that the error estimate is rather obtained here in a mixed norm~$\|[\cdot]_{2;\e}\|_{\Ld^p(\R^d)}$, which amounts to an $\Ld^2$-norm on small $O(\e)$ scales and to an $\Ld^p$-norm on large scales with some $p\gg2$; we do not further investigate this technical difference.

\begin{theor1}[Effective description of ensemble averages]\label{th:main}
Let $d>1$
and let the higher-order effective tensors $\{\bar\Aa^n\}_{1\le n\le d}$ be defined in~\eqref{eq:def-efftens} below.
For all $f\in C^\infty_c(\R^d)^d$, denoting by $\nabla\bar\Oc_\e^{d}[f]$ a proxy for the solution of the associated higher-order homogenized equation~\eqref{eq:homog-high} with $n=d$, see~Definition~\ref{def:high-eq}, there holds for all {$0<\eta\ll1$ and $p>\frac d{\eta}$},
\begin{align*}
&\big\|\big[\,\expec{\nabla u_{\e,f}}-\nabla\bar\Oc_\e^d[f]\,\big]_{2;\e}\big\|_{\Ld^p(\R^d)}\\
+~&\big\|\big[\,\expecmm{\Aa(\tfrac\cdot\e)\nabla u_{\e,f}}-\bar\Ac_\e^d\nabla\bar\Oc_\e^d[f]\,\big]_{2;\e}\big\|_{\Ld^p(\R^d)}
\qquad\,\lesssim_{f,p,\eta}\,~\e^{d-\eta}.
\qedhere
\end{align*}
\end{theor1}

Although this statement is only concerned with ensemble averages, our proof is based on a two-scale expansion for $\nabla u_{\e,f}$ in the weak sense of $\Sc'(\Omega;\Ld^2(\R^d))$.
In this spirit, a natural side question concerns the intrinsic description of fluctuations $\nabla u_{\e,f}-\expec{\nabla u_{\e,f}}$ in a similar weak sense; this is briefly addressed in Appendix~\ref{sec:fluct}.

\subsection{Weak corrector theory}\label{sec:weak-cor0}
While only the {first $\ell-1=\lceil\frac d2\rceil-1$ correctors} can be constructed as stationary fields with bounded moments, see~Theorem~\ref{th:cor} below, we show that essentially twice as many stationary correctors can be constructed in the distributional sense on the probability space. We do not know whether this is optimal in general.
We refer to Theorem~\ref{th:weak-cor} for a more detailed statement, where correctors are further controlled in some dual Malliavin--Sobolev spaces.

\begin{theor1}[Weak correctors]\label{th:weak-cor0}
Let $d>1$.
Higher-order weak correctors $\{\varphi^n\}_{1\le n\le d}$ are uniquely defined in (a suitable subspace of) $C^{\infty}(\R^d;\Sc'(\Omega))$
such that the standard cell problems {(see~\eqref{eq:phidef}--\eqref{eq:baradef} below)} are satisfied in a corresponding weak sense,
such that~$\varphi^n$ is stationary with $\expec{\varphi^n}=0$ for all $n<d$,
and such that $\nabla\varphi^d$ is stationary with $\expecmm{\nabla\varphi^d}=0$ and with the anchoring condition $\int_B\varphi^d=0$.
In addition, the following estimates hold:
\begin{enumerate}[(i)]
\item \emph{Weak sublinearity of $\varphi^d$:}
for all $X\in\Sc(\Omega)$ and $\eta>0$,
\begin{equation*}
\qquad\big|\expecmm{X\varphi^d(x)}\big|
\,\lesssim_{X,\eta}\,\langle x\rangle^\eta.
\end{equation*}
\item \emph{Weak fluctuation scaling:}
for all $n< d$, $X\in\Sc(\Omega)$, $g\in C^\infty_c(\R^d)$, and $p<\frac{d}{n}$,
\begin{equation*}
\qquad\bigg|\expecM{X\int_{\R^d}g\,\nabla\varphi^{n+1}}\bigg|+\bigg|\expecM{X\int_{\R^d}g\,\varphi^{n}}\bigg|
\,\lesssim_{X,p}\,\|[g]_2\|_{\Ld^p(\R^d)}.\qedhere
\end{equation*}
\end{enumerate}
\end{theor1}

\smallskip
\begin{rems}\label{rems1}
A few comments are in order.
\nopagebreak
\begin{enumerate}[(a)]
\item{\it Criticality at order $d$:}\\
Rather than constructing the corrector $\varphi^{n}$ itself, the proof of Theorem~\ref{th:weak-cor0} can be viewed as constructing its Malliavin derivative $D\varphi^{n}$
and monitoring the decay of the latter: item~(ii) is formally understood as $D_z\nabla\varphi^{n+1}(x)=O(\langle x-z\rangle^{n-d})$ in a weak sense.
The limitation is manifest: for $n\ge d$, no decay is left and the Malliavin derivative could at best be defined up to $(n-d)$th-order $\Aa$-harmonic polynomials.
This makes it unclear how the theory could be extended to higher order.

\smallskip\noindent
Criticality at order~$d$ is new in stochastic homogenization, but we note that the same actually holds in the following simple deterministic example: given a deterministic coefficient field of the form \mbox{$\Aa=\Id+\Cc$} with~$\Cc$ supported in the unit ball $B$, corresponding correctors $\{\psi^n\}_{1\le n\le d}$ can be constructed as unique decaying solutions of cell problems~\mbox{\eqref{eq:phidef}--\eqref{eq:qdef}} with~$|\nabla\psi^{n+1}(x)|\lesssim\langle x\rangle^{n-d}$, while for $n>d$ correctors would similarly be defined only up to 
$(n-d)$th-order harmonic polynomials.
\medskip\item {\it Explicit 1D case:}\\
In dimension $d=1$, the first corrector is explicit and its analysis provides an instructive illustration of the above result.
Choosing it with anchoring \mbox{$\phi^1(0)=0$}, it takes the form
\[\quad\nabla\phi^1(x)\,:=\,\expecm{\tfrac1\Aa}^{-1}\big(\tfrac1\Aa-\expecm{\tfrac1\Aa}\big),\qquad\phi^1(x)\,:=\,\expecm{\tfrac1\Aa}^{-1}\int_0^x\big(\tfrac1\Aa-\expecm{\tfrac1\Aa}\big).\]
Comparing $\frac1\Aa-\expec{\frac1\Aa}$ to a white noise, $\phi^1$ is compared to a Brownian motion.
The central limit theorem then yields $\|\phi^1(x)\|_{\Ld^q(\Omega)}\simeq_q\langle x\rangle^{1/2}$ for all $q<\infty$, which entails that no random constant can be added to $\phi^1$ to make it stationary. Formally, a stationary version of $\phi^1$ would be given by
\[\quad\varphi^1(x)\,:=\,\expecm{\tfrac1\Aa}^{-1}\int_{-\infty}^x\big(\tfrac1\Aa-\expecm{\tfrac1\Aa}\big),\]
but this integral makes no pointwise sense. Such an object can however be defined as follows: assuming for simplicity that the covariance function $\calC$ is compactly supported in $B$, so that $\Aa$ has a unit range of dependence, we can formally ``compute'' conditional expectations from the above formula: for all $R>0$,
\begin{equation}\label{eq:1D-explicit}
\quad\expeC{\varphi^1(x)}{\Aa|_{[-R,R]}}\,:=\,\expecm{\tfrac1\Aa}^{-1}\int_{-(R+1)}^x\expeC{\tfrac1\Aa-\expecm{\tfrac1\Aa}}{\Aa|_{[-R,R]}},
\end{equation}
where the right-hand side now makes perfect sense as a bounded random variable.
In other words, this stationary corrector $\varphi^1$ is only defined via its conditional expectations with respect to the coefficient field $\Aa$ on bounded sets. This gives sense to weak-type expressions of the form $\expecmm{X\varphi^1(x)}$ for all {local} test random variables $X$.
Note that this explicit 1D argument makes no use of Malliavin calculus.
\medskip\item {\it Improving on the topology:}\\
As stated in Theorem~\ref{th:weak-cor}, our analysis provides a control on correctors in some dual Malliavin--Sobolev spaces, which quantifies how smooth and local the test random variable $X$ should be for $\expec{X\varphi^n}$ to be well-defined.
However, we do not know whether this control is optimal, and in particular whether Malliavin calculus could be avoided.
At order~\mbox{$\ell=\lceil\frac d2\rceil$}, our analysis yields the following improved estimate: if the covariance function $\calC$ is compactly supported, then the stationary corrector~$\varphi^\ell$ is uniquely defined in $C^\infty(\R^d;\Sc'(\Omega))$ with $\expecmm{\varphi^\ell}=0$, and it satisfies for all $x\in\R^d$, $R>0$, and $q<\infty$,
\[\big\|\expeCm{\varphi^\ell(x)}{\Aa|_{B_R}}\big\|_{\Ld^q(\Omega)}\,\lesssim_{\calC,q}\,\left\{\begin{array}{lll}
\langle R\rangle^\frac12&:&\text{$d$ odd},\\
\log(2+R)^\frac12&:&\text{$d$ even}.
\end{array}\right.\]
This generalizes the above explicit 1D construction~\eqref{eq:1D-explicit}, and we do not know whether such estimates extend to higher orders $n>\ell$.
\medskip\item{\it Evaluating higher-order effective tensors:}\\
By the definition of weak correctors in Theorem~\ref{th:weak-cor0}, formula~\eqref{eq:def-barak} now makes sense and defines higher-order effective tensors for all $1\le n\le d$,
\begin{equation}\label{eq:def-efftens}
\quad\bar\Aa^n_{i_1\ldots i_{n-1}}\ee_{i_n}\,=\,\expecm{\Aa\big(\nabla\varphi^n_{i_1\ldots i_n}+\varphi^{n-1}_{i_1\ldots i_{n-1}}\ee_{i_n}\big)}.
\end{equation}
This formula can be substantially simplified in view of the following observation: if~$\Aa$ is pointwise symmetric, then we have for all~$1\le n\le d$,
\[\quad\nabla\cdot\bar\Aa^n_{i_1\ldots i_{n-1}}\nabla\nabla^{n-1}_{i_1\ldots i_{n-1}}\,=\,\left\{\begin{array}{lll}
\nabla\cdot\bar\Bb^n_{i_1\ldots i_{n-1}}\nabla\nabla^{n-1}_{i_1\ldots i_{n-1}}&:&\text{if $n$ is odd},\\
0&:&\text{if $n$ is even},
\end{array}\right.\]
where we have defined for all $n=2m+1$ with $m<\ell=\lceil\frac d2\rceil$,
\begin{equation*}
\qquad\ee_j\cdot\bar\Bb^n_{i_1\ldots i_{n-1}}\ee_{i_n}
:=(-1)^{m+1}\E\Big[\nabla\varphi^{m+1}_{ji_n\ldots i_{m+2}}\cdot\Aa\nabla\varphi^{m+1}_{i_1\ldots i_{m+1}}-\varphi^m_{ji_n\ldots i_{m+3}}\ee_{i_{m+2}}\cdot\Aa\varphi_{i_1\ldots i_m}^m\ee_{i_{m+1}}\Big].
\end{equation*}
(This identity states that the tensors $\bar\Aa^n$ and $\bar\Bb^n$ coincide at least after symmetrizing their indices, which is what matters when considering associated differential operators.)
This shows that~$\bar\Aa^n$ can be computed in terms of correctors of order $\le\lceil\frac n2\rceil$ only, so that the definition of $\{\bar\Aa^k\}_{1\le k\le d}$ only requires the use of standard ``strong'' correctors $\{\varphi^k\}_{1\le k\le\ell}$.
This identity was first noticed in our recent work~\cite[Remark~2.5]{DO1} with Otto, and also independently in~\cite[Theorem~3.5]{Pouch-19}.
\qedhere
\end{enumerate}
\end{rems}

\section{Preliminary: higher-order homogenization}
In this section, we recall some tools and notation from the higher-order quantitative theory of stochastic homogenization.
We start with large-scale regularity theory for linear elliptic equations with random coefficients~\cite{AS,AKM-book,GNO-reg}; for the purpose of this work, we focus on annealed $\Ld^p$ regularity as we introduced with Otto in~\cite{DO1}.
Next, we recall the higher-order intrinsic description of spatial oscillations by means of two-scale expansions~\cite{Gu-17,BFFO-17,DO1}, which is the starting point for our analysis.

\subsection{Annealed regularity theory}\label{sec:Lpreg}
Given $h\in \Ld^2(\R^d;\Ld^2(\Omega)^d)$, we consider the unique gradient solution $\nabla w_h\in\Ld^2(\R^d;\Ld^2(\Omega)^d)$ of the linear elliptic equation
\begin{equation}\label{eq:test}
-\nabla\cdot\Aa\nabla w_h=\nabla\cdot h,\qquad\text{in $\R^d$},
\end{equation}
and we denote by $\calH$ the associated solution operator, or heterogeneous Helmholtz projection on $\Ld^2(\R^d;\Ld^2(\Omega)^d)$,
\begin{equation}\label{eq:test-defT}
\calH h:=\nabla w_h,\qquad\calH=\nabla(-\nabla\cdot\Aa\nabla)^{-1}\nabla\cdot
\end{equation}
Due to the random character of the coefficient field $\Aa$, aside from the energy inequality and from Meyers' perturbative improvements, maximal $\Ld^p$ regularity cannot hold in a deterministic form. Appealing to homogenization, however, the heterogeneous elliptic operator can be replaced on large scales by a homogeneous one, cf.~\eqref{eq:first-def-ubar}, for which the standard constant-coefficient regularity theory is available. For this reason, the solution to~\eqref{eq:test} is expected to enjoy improved regularity properties on large scales.
In this spirit, a quenched large-scale regularity theory was first outlined by Armstrong and Smart~\cite{AS}, and further developed in~\cite{Armstrong-Mourrat-16,AKM2,AKM-book,GNO-reg}. 
The following annealed regularity result is a useful variant of this theory and was first established with Otto in~\cite[Section~6]{DO1}: it states that maximal $\Ld^p$ regularity holds as in the constant-coefficient setting after taking ensemble averages (up to a tiny loss in stochastic integrability).

\begin{theor}[Annealed $\Ld^p$ regularity; \cite{DO1}]\label{th:CZ-ann}
With the above notation, there holds for all $h\in C^\infty_c(\R^d;\Ld^\infty(\Omega)^d)$, \mbox{$1<p,q<\infty$}, and $\delta>0$,
\begin{equation*}
\|[\calH h]_2\|_{\Ld^p(\R^d;\Ld^{q}(\Omega))}\,\lesssim_{p,q,\delta}\,
\|[h]_2\|_{\Ld^p(\R^d;\Ld^{q+\delta}(\Omega))}.
\qedhere
\end{equation*}
\end{theor}

\subsection{Higher-order description of oscillations}\label{sec:cor-th}
Correctors are defined to correct Euclidean polynomials and make them adapted to the heterogeneous elliptic operator, and they constitute the key ingredient to describe spatial oscillations for solutions of the heterogeneous Poisson problem~\eqref{eq:first-def-ups}.
The following standard result recalls the hierarchy of corrector equations.
Item~(i) gives sharp corrector estimates~\cite{Gu-17,BFFO-17,DO1}, and item~(ii) states that higher-order correctors have less integrable correlations~\cite{MO,Gu-17} so that their large-scale averages have worse fluctuation scaling~\cite[proof of Lemma~3.3]{DO1}.

\begin{theor}[Higher-order correctors; \cite{Gu-17,BFFO-17,DO1}]\label{th:cor}
Let $d\ge1$ and recall $\ell=\lceil\frac d2\rceil$. Higher-order correctors $\{\varphi^n\}_{0\le n\le\ell}$, flux correctors $\{\sigma^n\}_{0\le n\le \ell}$, fluxes $\{q^n\}_{1\le n\le\ell}$, and effective tensors $\{\bar\Aa^n\}_{1\le n\le\ell}$ are uniquely well-defined iteratively as follows:
\begin{enumerate}[$\bullet$]
\item $\varphi^0:=1$ and for all $1\le n\le\ell$ we define $\varphi^n:=(\varphi^n_{i_1\ldots i_n})_{1\le i_1,\ldots,i_n\le d}$ where $\varphi^n_{i_1\ldots i_n}$ is the unique weak solution in $H^1_\loc(\R^d;\Ld^2(\Omega))$ of
\begin{equation}\label{eq:phidef}
-\nabla\cdot\Aa\nabla\varphi^n_{i_1\ldots i_n}=\nabla\cdot\big((\Aa\varphi^{n-1}_{i_1\ldots i_{n-1}}-\sigma^{n-1}_{i_1\ldots i_{n-1}})\,\ee_{i_n}\big),\qquad\text{in $\R^d$},
\end{equation}
such that $\varphi^n$ is stationary with $\expec{\varphi^n}=0$ for all $n<\ell$, and such that $\nabla\varphi^\ell$ is stationary with $\expecmm{\nabla\varphi^\ell}=0$ and with anchoring $\int_B\varphi^\ell=0$.
\smallskip\item $\sigma^0:=0$ and for all $1\le n\le\ell$ we define $\sigma^n:=(\sigma^n_{i_1\ldots i_n})_{1\le i_1,\ldots,i_n\le d}$ where $\sigma^n_{i_1\ldots i_n}$ is the unique weak solution in $H^1_\loc(\R^d;\Ld^2(\Omega)^{d\times d})$, with skew-symmetric matrix values, of
\begin{equation*}
-\triangle\sigma^n_{i_1\ldots i_n}=\nabla\times q^n_{i_1\ldots i_n},\qquad\text{in $\R^d$},
\end{equation*}
such that $\sigma^n$ is stationary with $\expec{\sigma^n}=0$ for all $n<\ell$, and such that $\nabla\sigma^\ell$ is stationary with $\expecmm{\nabla\sigma^\ell}=0$ and with anchoring $\int_B\sigma^\ell=0$. In particular, it satisfies
\begin{equation*}
\nabla\cdot\sigma^n_{i_1\ldots i_n}=q^n_{i_1\ldots i_n},\qquad\text{in $\R^d$}.
\end{equation*}
\item For all $1\le n\le\ell$ we define $q^n:=(q^n_{i_1\ldots i_n})_{1\le i_1,\ldots,i_n\le d}$ with $q^n_{i_1\ldots i_n}$ given by
\begin{equation}\label{eq:qdef}
q^n_{i_1\ldots i_n}:=\Aa\nabla\varphi^{n}_{i_1\ldots i_n}+(\Aa\varphi^{n-1}_{i_1\ldots i_{n-1}}-\sigma^{n-1}_{i_1\ldots i_{n-1}})\,\ee_{i_n}-\bar\Aa_{i_1\ldots i_{n-1}}^{n}\ee_{i_n}.
\end{equation}
\item For all $1\le n\le\ell$ we define $\bar\Aa^n:=(\bar\Aa^n_{i_1\ldots i_{n-1}})_{1\le i_1,\ldots,i_{n-1}\le d}$ with $\bar\Aa^n_{i_1\ldots i_{n-1}}$ the matrix given by
\begin{equation}\label{eq:baradef}
\bar\Aa^n_{i_1\ldots i_{n-1}}\ee_{i_n}:=\expecm{\Aa\big(\nabla\varphi^{n}_{i_1\ldots i_n}+\varphi^{n-1}_{i_1\ldots i_{n-1}}\ee_{i_n}\big)},
\end{equation}
which precisely ensures $\expec{q^n}=0$ in~\eqref{eq:qdef}.
\end{enumerate}
\smallskip
In addition, the following estimates hold:
\begin{enumerate}[(i)]
\item\emph{Corrector estimates:}
for all $0\le n<\ell$ and $q<\infty$,
\[\qquad\expecm{[(\varphi^n,\sigma^n)]_2^q}^\frac1q+\expecm{[(\nabla\varphi^{n+1},\nabla\sigma^{n+1})]_2^q}^\frac1q\,\lesssim_q\,1,\]
and at critical order $n=\ell$, for all $q<\infty$,
\[\qquad\expecm{[(\varphi^\ell,\sigma^\ell)]_2^q(x)}^\frac1q\,\lesssim_q\,
\left\{\begin{array}{lll}
\langle x\rangle^\frac12&:&\text{$d$ odd},\\
\log(2+|x|)^\frac12&:&\text{$d$ even}.
\end{array}\right.\]
\item\emph{Fluctuation scaling:}
for all $0\le n<\ell$, $g\in C^\infty_c(\R^d)$, and $q<\infty$,
\begin{equation*}
\quad\expec{\Big|\int_{\R^d}g\,(\nabla\varphi^{n+1},\nabla\sigma^{n+1})\Big|^q}^\frac1q+\expec{\Big|\int_{\R^d}g \,(\varphi^{n},\sigma^{n})\Big|^q}^\frac1q
\,\lesssim_q\,\|[g]_2\|_{\Ld^\frac{2d}{d+2n}(\R^d)}.\qedhere
\end{equation*}
\end{enumerate}
\end{theor}

\medskip
As recalled in the introduction, with these correctors at hand, we define higher-order two-scale expansions as
\begin{equation}\label{eq:def-2sc-exp}
\Fc_\e^n[\bar w]\,:=\,\sum_{k=0}^n\e^k\varphi^k_{i_1\ldots i_k}(\tfrac\cdot\e)\nabla^k_{i_1\ldots i_k}\bar w.
\end{equation}
Homogenization theory states that the heterogeneous constitutive law $\nabla w_\e\mapsto\Aa(\frac\cdot\e)\nabla w_\e$ is replaced on large scales by the effective law $\nabla\bar w\mapsto\bar\Aa^1\nabla\bar w$, and that dispersive corrections must be added when looking for finer accuracy,
\begin{equation}\label{eq:const-law-Aepsn}
\nabla\bar w\,\mapsto\,\bar\Ac_\e^n\nabla\bar w\,:=\,\sum_{k=1}^n\e^{k-1}\bar\Aa_{i_1\ldots i_{k-1}}^{k-1}\nabla\nabla^{k-1}_{i_1\ldots i_{k-1}}\bar w,
\end{equation}
in terms of the higher-order effective tensors $\{\bar\Aa^n\}_{1\le n\le\ell}$.
The first-order homogenized equation~\eqref{eq:first-def-ubar}
is then formally replaced by the following corrected version,
\begin{align}\label{eq:homog-formal}
-\nabla\cdot\bar\Ac_\e^n\nabla\bar U_{\e,f}^n=\nabla\cdot f,\qquad\text{in $\R^d$}.
\end{align}
This corrected equation is however ill-posed in general for $n\ge2$, as the symbol may be indefinite, and we shall use the following well-defined proxy.

\begin{defin}[Higher-order homogenized equations]\label{def:high-eq}
For all $1\le n\le\ell$,
the $n$th-order homogenized solution operator is defined as follows for all $f\in C^\infty_c(\R^d)^d$,
\[\nabla\bar \Oc_\e^n[f]\,:=\,\sum_{k=1}^n\e^{k-1}\nabla\tilde u^k_f,\]
where $\nabla\tilde u^1_f:=\nabla\bar u^1_f$ is the unique gradient solution in $\Ld^2(\R^d)^d$ of the first-order homogenized equation,
\[-\nabla\cdot\bar\Aa^1\nabla\bar u^1_f=\nabla\cdot f,\qquad\text{in $\R^d$},\]
and where for $2\le n\le\ell$ we inductively define the $n$th-order correction $\nabla\tilde u^n_f$ as the unique gradient solution in~$\Ld^2(\R^d)^d$ of
\[-\nabla\cdot\bar\Aa^1\nabla\tilde u^n_f=\nabla\cdot\sum_{k=2}^n\bar\Aa^k_{i_1\ldots i_{k-1}}\nabla\nabla^{k-1}_{i_1\ldots i_{k-1}}\tilde u^{n+1-k}_f,\qquad\text{in $\R^d$}.\qedhere\]
\end{defin}

This function $\nabla\bar\Oc_\e^n[f]$ is indeed viewed as a proxy for a solution $\nabla\bar U_{\e,f}^n$ of the $n$th-order homogenized equation~\eqref{eq:homog-formal} since the following identity holds,
\begin{equation*}
-\nabla\cdot\bar\Ac_\e^n\nabla\bar\Oc_\e^n[f]\,=\,\nabla\cdot f
-\nabla\cdot\Big(\sum_{k=2}^n~\sum_{l=n+2-k}^n\,\e^{k+l-2}\,\bar\Aa^k_{i_1\ldots i_{k-1}}\nabla\nabla^{k-1}_{i_1\ldots i_{k-1}}\tilde u^l_f\Big),
\end{equation*}
where the second right-hand side term is an error term of order $O(\e^n)$.
With these definitions, using the algebraic structure of corrector equations together with corrector estimates, the accuracy of the higher-order two-scale expansion is easily justified, cf.~\cite[Theorem~1.6]{Gu-17} and~\cite[Proposition~2.7]{DO1}. Due to the maximal number $\ell=\lceil\frac d2\rceil$ of well-behaved correctors, the accuracy saturates at the optimal order~$O(\e^{d/2})$. Optimality is confirmed as random fluctuations become dominant at that order; see~\cite{DGO1,DO1}.

\begin{prop}[\cite{Gu-17,DO1}]\label{cor:2scale}
For all $1\le n\le\ell$ and $f\in C^\infty_c(\R^d)^d$, the $n$th-order two-scale expansion error satisfies
\begin{multline}\label{eq:2sc-identity-err}
-\nabla\cdot\Aa(\tfrac\cdot\e)\nabla(u_{\e,f}-\Fc^n_\e[\bar\Oc^n_\e[f]])=\nabla\cdot\Big(\sum_{k=2}^n\sum_{l=n+2-k}^n\e^{k+l-2}\,\bar\Aa^k_{i_1\ldots i_{k-1}}\nabla\nabla^{k-1}_{i_1\ldots i_{k-1}}\tilde u^l_f\Big)\\
+\nabla\cdot\big(\e^n(\Aa\varphi^n_{i_1\ldots i_n}-\sigma^n_{i_1\ldots i_n})(\tfrac\cdot\e)\nabla\nabla^n_{i_1\ldots i_n}\bar\Oc_\e^n[f]\big),
\end{multline}
which entails for all $1< p,q<\infty$,
\[\big\|\big[\,\nabla u_{\e,f}- \nabla\Fc_\e^\ell[\bar\Oc_\e^\ell[f]]\,\big]_2\big\|_{\Ld^p(\R^d;\Ld^q(\Omega))}\,\lesssim_{f,p,q}\,\e^\frac d2\times\left\{\begin{array}{lll}
1&:&\text{$d$ odd},\\
\Log^\frac12&:&\text{$d$ even}.
\end{array}\right.\qedhere\]
\end{prop}

\section{Functional setup}\label{sec:prelim}
In this section, we introduce some functional tools and notation. We start with a reminder on stationarity and on the associated stationary differential calculus on the probability space, as commonly used in stochastic homogenization~\cite{PapaVara,JKO94}, and we further develop the formalism for our needs. Next, we turn to a brief reminder on Malliavin calculus and introduce a new scale of ``refined'' Malliavin--Sobolev spaces. While stationary calculus is generated by spatial translations of stationary fields, Malliavin calculus is generated by changes in the underlying white noise. Both will be crucially combined in our construction of weak correctors in the sequel.

\subsection{Stationary calculus}\label{sec:stat}
In the present Gaussian setting~\eqref{eq:def-A}--\eqref{eq:G-rep}, recall that the probability space $(\Omega,\Pm)$ is chosen to be generated by the underlying white noise $\xi$.
Translations $x\mapsto \xi(\cdot+x)$ then induce an action $T=\{T_x\}_{x\in\R^d}$ of the additive group $(\R^d,+)$ on the space of random variables.
In particular, for $X\in\Sc(\Omega)$, say $X=g(\int_{\R^d}h_1\cdot\xi,\ldots,\int_{\R^d}h_n\cdot\xi)$, we define
\[T_xX=g\Big(\int_{\R^d} h_1\cdot\xi(\cdot+x),\ldots,\int_{\R^d} h_n\cdot\xi(\cdot+x)\Big).\]
As the law of $\xi(\cdot+x)$ does not depend on the shift $x\in\R^d$, the map $T_x$ is an isometry on $\Ld^q(\Omega)$ for all $q$.
Also note that, for any random variable $X$, the map $(x,\omega)\mapsto (T_xX)(\omega)$ is stochastically continuous and jointly measurable on $\R^d\times\Omega$.

In this setting, stationarity is defined as follows: a random field $\psi\in\Ld^1_\loc(\R^d;\Ld^1(\Omega))$ is stationary if there exists a random variable $X\in\Ld^1(\Omega)$ such that $\psi(x,\omega)=(T_xX)(\omega)$. In other words, a random field is stationary if it is covariant under spatial translations: spatial translations of the field amount to spatial translations of the underlying Gaussian field.
For all $q$, this provides a canonical isomorphism between random variables in $\Ld^q(\Omega)$ and stationary random fields in $\Ld^q_\loc(\R^d;\Ld^q(\Omega))$: for $X\in\Ld^q(\Omega)$ we define its stationary extension $X^\sharp(x,\omega)=(T_xX)(\omega)$, and for a stationary field $\psi\in\Ld^q_\loc(\R^d;\Ld^q(\Omega))$ we define $\psi^\flat\in\Ld^q(\Omega)$ such that $\psi(x,\omega)=(T_x\psi^\flat)(\omega)$. Clearly, $(\psi^\flat)^\sharp=\psi$ and $(X^\sharp)^\flat=X$.

A differential calculus on $\Ld^q(\Omega)$ is naturally developed via stationarity.
Indeed, for all~\mbox{$q<\infty$}, the action $T=\{T_x\}_{x\in\R^d}$ is easily checked to constitute a ($d$-parameter) $C_0$-group of isometries on $\Ld^q(\Omega)$,
and we then define the {\it stationary gradient}~$\nabla^\st$ as the generator of this group, which is a densely defined skew-adjoint operator on $\Ld^q(\Omega)$. Its domain is denoted by $W^{1,q}(\Omega)$ and obviously contains $\Sc(\Omega)$. For $q=2$, we write $H^1(\Omega)=W^{1,2}(\Omega)$. 

This stationary gradient can be reinterpreted via the isomorphism between random variables and stationary random fields: as the space of stationary fields in $\Ld^q_\loc(\R^d;\Ld^q(\Omega))$ is identified with $\Ld^q(\Omega)$, the weak spatial gradient $\nabla$ on locally $\Ld^q$ integrable functions turns into a densely defined linear operator on $\Ld^q(\Omega)$, which is easily checked to coincide with the stationary gradient $\nabla^\st$. More precisely, the space $W^{1,q}(\Omega)$ coincides with the space of stationary fields in $W^{1,q}_\loc(\R^d;\Ld^p(\Omega))$, and there holds for all~$X\in W^{1,q}(\Omega)$,
\[(\nabla^\st X)^\sharp=\nabla X^\sharp.\]

In the context of corrector equations, cf.~Theorem~\ref{th:cor}, we often consider linear elliptic equations of the following form, given a stationary field $\psi\in\Ld^2_\loc(\R^d;\Ld^2(\Omega)^d)$,
\[-\nabla\cdot\Aa\nabla w_\psi=\nabla\cdot \psi,\qquad\text{in $\R^d$}.\]
The gradient solution $\nabla w_\psi$ can be uniquely chosen as a stationary field in $\Ld^2_\loc(\R^d;\Ld^2(\Omega)^d)$ with $\expec{\nabla w_\psi}=0$ (see e.g.~\cite[proof of Theorem~2]{PapaVara}), but $w_\psi$ itself cannot be chosen stationary in general due to the absence of a Poincaré inequality in this stationary setting --- which is in turn the very reason why there exist only a finite number of stationary correctors.
We shall abusively use the same notation as in~\eqref{eq:test-defT},
\begin{equation}\label{eq:def-Hstat}
\calH\psi=\nabla w_\psi,
\end{equation}
for the solution operator, or heterogeneous Helmholtz projection, in this stationary setting.
In terms of stationary calculus, the above equation is equivalently written as follows,\footnote{These equations are understood in the following weak sense: there holds $\E[\nabla^\st Y\cdot(\Aa^\flat X_\psi+\psi^\flat)]=0$ for all $Y\in H^1(\Omega)$, and the random vector $X_\psi$ belongs to the closure of the space of stationary gradients, that is, $\expec{X_\psi}=0$ and $\expec{(\nabla_j^\st Y)(X_\psi)_k}=\expec{(\nabla_k^\st Y)(X_\psi)_j}$ for all $j,k$ and all $Y\in H^1(\Omega)$.}
\[\nabla w_\psi=X_\psi^\sharp,\qquad-\nabla^\st\cdot(\Aa^\flat X_\psi)=\nabla^\st\cdot \psi^\flat,\qquad\nabla^\st\times X_\psi=0,\qquad\expec{X_\psi}=0,\]
and we set $\calH^\st \psi^\flat:=(\calH \psi)^\flat=X_\psi$ on $\Ld^2(\Omega)^d$.
In parallel with the annealed $\Ld^p$ regularity in Theorem~\ref{th:CZ-ann},
we prove a corresponding regularity result for this stationary Helmholtz projection in $\Ld^q(\Omega)$ up to a tiny loss of stochastic integrability.

\begin{cor}[Stationary $\Ld^q$ regularity]\label{cor:CZ-stat}
With the above notation, there holds for all $X\in\Ld^\infty(\Omega)$, $1< q<\infty$, and $\delta>0$,
\[\|[\calH X^\sharp]_2^\flat\|_{\Ld^q(\Omega)}\,\lesssim_{q,\delta}\,
\|[X^\sharp]_2^\flat\|_{\Ld^{q+\delta}(\Omega)},\]
where we recall the notation $[X^\sharp]_2^\flat=(\fint_B|X^\sharp|^2)^{1/2}$ for the local quadratic average.
\end{cor}

\begin{proof}
By a duality argument, it suffices to consider the case $2\le q<\infty$.
Appealing to quenched large-scale regularity in form of~\cite[Lemma~9.2]{BFFO-17}, and recalling that in the present Gaussian setting with integrable correlations the so-called minimal radius $r_*$ in~\cite{BFFO-17} satisfies $\expecm{(r_*)^{dq}}\le (Cq)^q$ for all $q<\infty$, see~\cite[Theorem~3]{GNO-reg}, we obtain for all $2\le q<d$ and~$\alpha>1$, setting $q_*:=\frac{dq}{d-q}$,
\[\|[\calH X^\sharp]_2^\flat\|_{\Ld^{\frac{q_*}\alpha}(\Omega)}\,\lesssim\,q_*^3\big(\tfrac{q_*}{\alpha-1}\big)^\frac12\Big(\|[X^\sharp]_2^\flat\|_{\Ld^{q_*}(\Omega)}+\|[\calH X^\sharp]_2^\flat\|_{\Ld^q(\Omega)}\Big).\]
Iterating this inequality on $\calH X^\sharp$ and starting with the energy estimate
\[\|[\calH X^\sharp]_2^\flat\|_{\Ld^2(\Omega)}\,=\,\|\calH^\st X\|_{\Ld^2(\Omega)}\,\lesssim\,\|X\|_{\Ld^2(\Omega)}\,=\,\|[X^\sharp]_2^\flat\|_{\Ld^2(\Omega)},\]
the conclusion follows.
\end{proof}

\subsection{Malliavin calculus}\label{chap:Mall}
We recall some classical notation and tools from Malliavin calculus; we refer to~\cite{Nualart,NP-book} for details.
The underlying white noise $\xi$ in~\eqref{eq:G-rep} is viewed as a random Schwartz distribution: by definition, the collection of linear random variables
\[\xi(h):=\int_{\R^d}h\cdot\xi,\qquad h\in C^\infty_c(\R^d)^\kappa,\]
are jointly Gaussian random variables with expectation $\expec{\xi(h)}=0$ and with covariance
\[\langle\xi(h'),\xi(h)\rangle_{\Ld^2(\Omega)}\,=\,\expec{\xi(h')\xi(h)}\,=\,\int_{\R^d}h'\cdot h\,=\,\langle h',h\rangle_{\Ld^2(\R^d)}.\]
By density, we may define $\xi(h)\in\Ld^2(\Omega)$ for all $h\in\Ld^2(\R^d)^\kappa$, which provides an embedding $\Ld^2(\R^d)^\kappa\hookrightarrow\Ld^2(\Omega)$.
Similarly as $C^\infty_c(\R^d)^\kappa$ is viewed as a model dense subspace of $\Ld^2(\R^d)^\kappa$, we recall the following model subspace of ``smooth and local'' random variables, cf.~\eqref{eq:defRom-0},
\begin{equation*}
\Sc(\Omega)\,:=\,\Big\{g\big(\xi(h_1),\ldots,\xi(h_n)\big)~:~n\in\N,~g\in C_b^\infty(\R^n),~
h_1,\ldots,h_n\in C^\infty_c(\R^d)^\kappa\Big\}.
\end{equation*}
For all $q<\infty$, as this subspace $\Sc(\Omega)$ is dense in $\Ld^q(\Omega)$, it allows to define operators and prove properties on this simpler subspace before extending them to $\Ld^q(\Omega)$.

For a random variable $X\in\Sc(\Omega)$, say $X=g(\xi(h_1),\ldots,\xi(h_n))$, we define its Malliavin derivative $DX\in\Ld^q(\Omega;\Ld^2(\R^d)^\kappa)$ as
\begin{equation}\label{eq:D-expl}
DX:=\,\sum_{j=1}^nh_j\,(\partial_jg)\big(\xi(h_1),\ldots,\xi(h_n)\big),
\end{equation}
and we use the short-hand notation $D_xX:=(DX)(x)\in\Ld^q(\Omega)^\kappa$.
Similarly, for all $k\ge1$, the $k$th Malliavin derivative $D^kX\in\Ld^q(\Omega;(\Ld^2(\R^d)^\kappa)^{\otimes k})$ is given by
\begin{equation*}
D^kX:=\,\sum_{1\le j_1,\ldots,j_k\le n}\big(h_{j_1}\otimes\ldots\otimes h_{j_k}\big)\,(\partial_{j_1\ldots j_k}^kg)\big(\xi(h_1),\ldots,\xi(h_n)\big).
\end{equation*}
This operator $D^k:\Sc(\Omega)\subset\Ld^q(\Omega)\to\Ld^q(\Omega;(\Ld^2(\R^d)^\kappa)^{\otimes k})$ is easily checked to be closable on $\Ld^q(\Omega)$ for any $q<\infty$, and we still denote by $D^k$ its closure. We also set
\begin{equation}\label{eq:def-Dk2}
\|X\|_{\Dm^{k,q}(\Omega)}^q\,:=\,\|X\|_{\Ld^q(\Omega)}^q+\sum_{j=1}^k\|D^jX\|_{\Ld^q(\Omega;(\Ld^2(\R^d)^\kappa)^{\otimes j})}^q,
\end{equation}
and we define the Malliavin--Sobolev space $\Dm^{k,q}(\Omega)$ as the closure of $\Sc(\Omega)\subset\Ld^q(\Omega)$ for this norm. This is also sometimes referred to as the Watanabe--Sobolev space.

Next, we define the divergence operator~$D^*$ as the adjoint of the Malliavin derivative $D$, and we construct the corresponding Ornstein--Uhlenbeck operator
\[\Lc:=D^*D.\]
This operator $\Lc:\Sc(\Omega)\subset\Ld^q(\Omega)\to\Ld^q(\Omega)$ is well-defined, symmetric, and closable. Its closure is still denoted by $\Lc$ and is a non-negative operator with domain $\Dm^{2,q}(\Omega)$. For~$q=2$, it is also self-adjoint.
In the sequel, we will often let the operator $D$ act on functions $w\in\Ld^1_\loc(\R^d;\Dm^{1,2}(\Omega))$ by freezing the space variable, that is, $(D_zw)(x):=D_z(w(x))$, and similarly for $\Lc$.
The following lemma collects some useful properties; although standard, a short proof is included for the reader's convenience.

\begin{lem}\label{lem:Mall}$ $
\begin{enumerate}[(i)]
\item The Ornstein--Uhlenbeck operator $\Lc$ commutes with stationarity, in the sense that $\Lc T_xX=T_x\Lc X$ for all~$x\in\R^d$ and $X\in\Sc(\Omega)$.
\smallskip\item For all $q$, the operator $(1+\Lc)^{-1}$ is a contraction on $\Ld^q(\Omega)$. Moreover, for all $w\in C^\infty_c(\R^d;\Rc(\Omega))$ and all $p,q$, we have $\|[(1+\Lc)^{-1}w]_p\|_{\Ld^q(\Omega)}\le\|[w]_p\|_{\Ld^q(\Omega)}$.
\smallskip\item For all $X,Y\in \Dm^{1,2}(\Omega)$, the following Helffer--Sjöstrand identity holds
\begin{equation*}
\quad\covm YX\,=\,\expec{\langle DY,(1+\Lc)^{-1}D X\rangle_{\Ld^2(\R^d)}},
\end{equation*}
which implies in particular Poincaré's inequality
\begin{equation*}
\quad\var{X}\,\le\,\|DX\|_{\Ld^2(\Omega;\Ld^2(\R^d))}^2.
\qedhere
\end{equation*}
\end{enumerate}
\end{lem}

\begin{proof}
We consider the Ornstein-Uhlenbeck semigroup $\{e^{-t\Lc}\}_{t\ge0}$ and we recall Mehler's formula: choosing an iid copy $\xi'$ of the white noise $\xi$, and denoting by $\E'[\cdot]$ the expectation with respect to $\xi'$, there holds for all $t\ge0$ and $X\in\Sc(\Omega)$,
\begin{equation}\label{eq:Mehler}
e^{-t\Lc}X=\E'[X_t],
\end{equation}
where $X_t$ is obtained by replacing $\xi$ with its interpolation $\xi_t:=e^{-t}\xi+\sqrt{1-e^{-2t}}\xi'$.
This formula is easily justified by explicitly computing the generator of both sides on $\Sc(\Omega)$; see e.g.~\cite[Theorem~2.8.2]{NP-book}.
Based on this formula, there obviously holds $e^{-t\Lc}T_xX=T_xe^{-t\Lc}X$, which proves~(i).

\medskip\noindent
We turn to the proof of~(ii).
Writing
$(1+\Lc)^{-1}X\,=\,\int_0^\infty e^{-t}e^{-t\Lc}X\,dt$,
and using that $\xi_t$ has the same law as $\xi$, which entails
\begin{equation}\label{eq:bnd-etL}
\|e^{-t\Lc}X\|_{\Ld^q(\Omega)}=\|\E'[X_t]\|_{\Ld^q(\Omega)}\le\|X\|_{\Ld^q(\Omega)},
\end{equation}
we find
\[\|(1+\Lc)^{-1}X\|_{\Ld^q(\Omega)}\,\le\,\int_0^\infty e^{-t}\,\|\E'[X_t]\|_{\Ld^q(\Omega)}\,dt\,\le\,\|X\|_{\Ld^q(\Omega)}.\]
The same computation holds when inserting local averages, and the conclusion~(ii) follows.

\medskip\noindent
It remains to prove~(iii), and we start with Poincaré's inequality.
As Mehler's formula~\eqref{eq:Mehler} yields $e^{-t\Lc}X\to\expec{X}$ in $\Ld^2(\Omega)$ as $t\uparrow\infty$, we can write
\[\var{X}\,=\,-\int_0^\infty\partial_t\,\expec{(e^{-t\Lc}X)^2}dt\,=\,2\int_0^\infty\expec{(e^{-t\Lc}X)\Lc(e^{-t\Lc}X)}dt,\]
and thus, integrating by parts with $\Lc=D^*D$,
\[\var{X}\,=\,2\int_0^\infty\|D(e^{-t\Lc}X)\|_{\Ld^2(\R^d;\Ld^2(\Omega))}^2\,dt.\]
Noting that Mehler's formula~\eqref{eq:Mehler} yields
\begin{equation}\label{eq:commut}
De^{-t\Lc}X=e^{-t}e^{-t\Lc}DX,
\end{equation}
and using~\eqref{eq:bnd-etL} again, Poincaré's inequality follows,
\[\var{X}\,=\,2\int_0^\infty e^{-2t}\,\|e^{-t\Lc}DX\|_{\Ld^2(\R^d;\Ld^2(\Omega))}^2\,dt\,\le\,\|DX\|_{\Ld^2(\R^d;\Ld^2(\Omega))}^2.\]
We now turn to the Helffer--Sjöstrand identity.
Let $X,Y\in\Sc(\Omega)$ with $\expec{X}=\expec{Y}=0$.
Poincaré's inequality together with the Lax--Milgram theorem ensures that there exists $Z\in\Dm^{2,2}(\Omega)$ with $X=\Lc Z$.
Integrating by parts with $\Lc=D^*D$, we then find
\[\cov{Y}{X}\,=\,\expec{YX}\,=\,\expec{Y(\Lc Z)}\,=\,\expecm{\langle DY,DZ\rangle_{\Ld^2(\R^d)}}.\]
Noting that identity~\eqref{eq:commut} entails $D\Lc Z=(1+\Lc)DZ$, we deduce $DZ=(1+\Lc)^{-1}DX$, and the conclusion follows.
\end{proof}

\subsection{``Refined'' Malliavin--Sobolev spaces}
We define the following ``refined'' version of Malliavin--Sobolev spaces, which play a key role in the sequel of this work and are a variant of spaces first introduced in~\cite{AKL-16}: for $1\le p,q<\infty$ and integer $k\ge0$, we set
\begin{equation}\label{eq:def-Mpqk}
\|X\|_{\Md^k_{p,q}(\Omega)}\,:=\,\|[X^\sharp]_2^\flat\|_{\Ld^q(\Omega)}+\sum_{j=1}^k\,\|[D^j_0X^\sharp]_2\|_{\Ld^p((\R^d)^j;\Ld^q(\Omega))},
\end{equation}
and we define the refined Malliavin--Sobolev space $\Md^k_{p,q}(\Omega)$ as the closure of $\Sc(\Omega)$ for this norm,
which is easily checked to define a separable Banach space.
By stationarity, note that the norm $\|[D_x^jX^\sharp]_2\|_{\Ld^p((\R^d)^j;\Ld^q(\Omega))}$ is independent of $x\in\R^d$ (see the proof below), and we have simply chosen~$x=0$ in the definition of the norm.
We refer to Remark~\ref{rem:refMS-choice} for the motivation of this precise definition and for the comparison with the slightly different choice in~\cite{AKL-16}.
We start with a few general properties.

\begin{samepage}
\begin{lem}\label{lem:Poinc-MS}$ $
\begin{enumerate}[(i)]
\item For $p=q=2$, the Banach spaces $\Md^k_{2,2}(\Omega)$ and $\Dm^{k,2}(\Omega)$ are isomorphic.
\smallskip\item The Banach space $\Md^k_{p,q}(\Omega)$ embeds into $\Md^l_{r,s}(\Omega)$ whenever $k\ge l$, $p\le r$, $q\ge s$.
\smallskip\item For all $X,Y\in\Sc(\Omega)$ and $1\le p,q\le\infty$, with $\frac1p+\frac1{p'}=1$ and $\frac1q+\frac1{q'}=1$,
\begin{equation*}
\qquad\cov YX
\,\le\,\|[D_0Y^\sharp]_2\|_{\Ld^{p'}(\R^d;\Ld^{q'}(\Omega))}\|[D_0X^\sharp]_2\|_{\Ld^p(\R^d;\Ld^q(\Omega))}.
\end{equation*}
\item The following Poincaré inequality holds for all $X\in\Sc(\Omega)$ and $2\le q<\infty$,
\[\qquad\|[(X-\expec X)^\sharp]_2^\flat\|_{\Ld^{q}(\Omega)}\,\lesssim\, q^\frac12\|[ D_0X^\sharp]_2\|_{\Ld^2(\R^d;\Ld^q(\Omega))}.\qedhere\]
\end{enumerate}
\end{lem}
\end{samepage}

\begin{proof}
We start with the following observation: the definition~\eqref{eq:D-expl} of the Malliavin derivative and the definition of stationary extensions lead to
\begin{equation}\label{eq:rel-D-stat}
T_xD_yX=D_{y+x}T_xX,
\end{equation}
and thus, as $T_x$ is an isometry on $\Ld^q(\Omega)$,
\begin{multline}
\|D_0X^\sharp\|_{\Ld^p(\R^d;\Ld^q(\Omega))}^p\,=\,\int_{\R^d}\expecm{|D_0T_xX|^q}^\frac{p}qdx
\,=\,\int_{\R^d}\expecm{|T_xD_{-x}X|^q}^\frac{p}qdx\\
\,=\,\int_{\R^d}\expecm{|D_{-x}X|^q}^\frac{p}qdx
\,=\,\|DX\|_{\Ld^p(\R^d;\Ld^q(\Omega))}^p.\label{eq:D0psi-Dpsi0}
\end{multline}
We turn to the proof of~(i). For $q=2$, local quadratic averages can be removed by stationarity and we find
\begin{equation*}
\|X\|_{\Md^k_{p,2}(\Omega)}\,=\,\|X\|_{\Ld^2(\Omega)}+\sum_{j=1}^k\| D^j_0X^\sharp\|_{\Ld^p((\R^d)^j;\Ld^2(\Omega))}.
\end{equation*}
Combined with~\eqref{eq:D0psi-Dpsi0}, this becomes
\begin{equation*}
\|X\|_{\Md^k_{p,2}(\Omega)}\,=\,\|X\|_{\Ld^2(\Omega)}+\sum_{j=1}^k\| D^jX\|_{\Ld^p((\R^d)^j;\Ld^2(\Omega))}.
\end{equation*}
In particular, for $p=2$, we deduce that the norms of $\Md^k_{2,2}(\Omega)$ and $\Dm^{k,2}(\Omega)$ are Lipschitz-equivalent, which proves~(i).

\medskip\noindent
Next, for $p\le r$ and $q\ge s$, item~(ii) follows from Jensen's inequality, $\Ld^q(\Omega)\subset\Ld^s(\Omega)$, together with the following discrete $\ell^p-\ell^r$ inequality, taking advantage of local averages,
\begin{multline}\label{eq:ellp-ellr}
\|[D^j_0X^\sharp]_2\|_{\Ld^p((\R^d)^j;\Ld^q(\Omega))}\,\simeq\,\bigg(\sum_{z\in\frac1C\Z^d}\|[D^j_0X^\sharp]_2(z)\|_{\Ld^q(\Omega)}^p\bigg)^\frac1p\\
\,\lesssim\,\bigg(\sum_{z\in\frac1C\Z^d}\|[D^j_0X^\sharp]_2(z)\|_{\Ld^q(\Omega)}^r\bigg)^\frac1r\,\simeq\,\|[D^j_0X^\sharp]_2\|_{\Ld^r((\R^d)^j;\Ld^q(\Omega))}.
\end{multline}

\medskip\noindent
We turn to the proof of~(iii). Starting point is the Helffer--Sjöstrand formula in Lemma~\ref{lem:Mall}.
Using~\eqref{eq:rel-D-stat}, recalling that $\Lc$ commutes with stationarity in the sense of Lemma~\ref{lem:Mall}(i), and that $T$ is a group of isometries, it becomes
\begin{eqnarray*}
\cov YX&=&\expecM{\int_{\R^d}(D_xY)\,(1+\Lc)^{-1}(D_xX)\,dx}\\
&=&\expecM{\int_{\R^d}(T_{-x}D_xY)\,(1+\Lc)^{-1}(T_{-x}D_xX)\,dx}\\
&=&\expecM{\int_{\R^d}(D_0Y^\sharp)(x)\,(1+\Lc)^{-1}(D_0X^\sharp)(x)\,dx},
\end{eqnarray*}
and thus, by Hölder's inequality,
\begin{equation*}
\cov YX
\,\le\,\|[D_0Y^\sharp]_2\|_{\Ld^{p'}(\R^d;\Ld^{q'}(\Omega))}\|[(1+\Lc)^{-1}D_0X^\sharp]_2\|_{\Ld^p(\R^d;\Ld^q(\Omega))}.
\end{equation*}
Appealing to Lemma~\ref{lem:Mall}(ii) then yields the conclusion~(iii).

\medskip\noindent
We turn to the proof of~(iv).
Note that some care is needed to deal with local quadratic averages.
Let $X\in\Sc(\Omega)$ with $\expec X=0$.
Starting point is the decomposition
\begin{equation}\label{eq:split-Lq-aver}
\|[X^\sharp]_2^\flat\|_{\Ld^q(\Omega)}^q\,=\,\expecm{([X^\sharp]_2^\flat)^{q-1}}\expecm{[X^\sharp]_2^\flat}+\covm{([X^\sharp]_2^\flat)^{q-1}}{[X^\sharp]_2^\flat},
\end{equation}
and we separately analyze the two right-hand side terms.
On the one hand, using Jensen's inequality and Poincaré's inequality of Lemma~\ref{lem:Mall}(iii), we find
\begin{eqnarray*}
\expecm{([X^\sharp]_2^\flat)^{q-1}}\expecm{[X^\sharp]_2^\flat}&\le&\|[X^\sharp]_2^\flat\|_{\Ld^q(\Omega)}^{q-1}\|X\|_{\Ld^2(\Omega)}\\
&\le&\|[X^\sharp]_2^\flat\|_{\Ld^q(\Omega)}^{q-1}\|DX\|_{\Ld^2(\Omega;\Ld^2(\R^d))},
\end{eqnarray*}
and thus, by (4.9),
\begin{equation}\label{eq:split-Lq-aver-1}
\expecm{([X^\sharp]_2^\flat)^{q-1}}\expecm{[X^\sharp]_2^\flat}
\,\le\,\|[X^\sharp]_2^\flat\|_{\Ld^q(\Omega)}^{q-1}\|[D_0X^\sharp]_2\|_{\Ld^2(\R^d;\Ld^2(\Omega))}.
\end{equation}
On the other hand, the Helffer--Sjöstrand identity of Lemma~\ref{lem:Mall}(iii) together with identity~\eqref{eq:rel-D-stat} yields
\begin{eqnarray*}
\covm{([X^\sharp]_2^\flat)^{q-1}}{[X^\sharp]_2^\flat}
&=&\expecM{\int_{\R^d} D_y\big(([X^\sharp]_2^\flat)^{q-1}\big)\,(1+\Lc)^{-1} D_y([X^\sharp]_2^\flat)\,dy}\\
&=&\expecM{\int_{\R^d} D_0\big(([X^\sharp]_2)^{q-1}\big)\,(1+\Lc)^{-1} D_0([X^\sharp]_2)},
\end{eqnarray*}
and thus, computing $ D_0\big(([X^\sharp]_2)^{q-1}\big)=(q-1)([X^\sharp]_2)^{q-2} D_0([X^\sharp]_2)$, using Hölder's inequality, appealing to Lemma~\ref{lem:Mall}(ii), and noting that $| D_0([X^\sharp]_2)|\le[ D_0X^\sharp]_2$, we find
\begin{eqnarray*}
\covm{([X^\sharp]_2^\flat)^{q-1}}{[X^\sharp]_2^\flat}\,\le\,(q-1)\|[X^\sharp]_2^\flat\|_{\Ld^q(\Omega)}^{q-2}\big\|[ D_0X^\sharp]_2\big\|_{\Ld^2(\R^d;\Ld^q(\Omega))}^2.
\end{eqnarray*}
Inserting this into~\eqref{eq:split-Lq-aver} together with~\eqref{eq:split-Lq-aver-1}, the claim~(iv) follows.
\end{proof}

Next, as a consequence of Theorem~\ref{th:CZ-ann} and Corollary~\ref{cor:CZ-stat}, we establish a corresponding regularity estimate for the stationary Helmholtz projection $\calH^\st$ in Malliavin--Sobolev spaces; recall the definition~\eqref{eq:def-Hstat} and $(\calH^\st X)^\sharp=\calH X^\sharp$.

\begin{cor}[Stationary $\Md^1_{p,q}$ regularity]\label{cor:CZ-M}
For all $X\in\Sc(\Omega)$, $1<p,q<\infty$, and $\delta>0$,
\begin{equation*}
\|\calH^\st X\|_{\Md^1_{p,q}(\Omega)}\,\lesssim_{p,q,\delta}\,\|X\|_{\Md^1_{p,q+\delta}(\Omega)}.
\qedhere
\end{equation*}
\end{cor}

\begin{proof}
Let $X\in\Sc(\Omega)$, $1<p,q<\infty$, and $\delta>0$ be fixed.
Appealing to the stationary~$\Ld^q$ regularity estimate of Corollary~\ref{cor:CZ-stat}, we find
\begin{equation}\label{eq:nabpsi-Lp-1}
\|[(\calH^\st X)^\sharp]_2^\flat\|_{\Ld^q(\Omega)}\,\lesssim_{q,\delta}\,\|[X^\sharp]_2^\flat\|_{\Ld^{q+\delta}(\Omega)},
\end{equation}
and it remains to show
\begin{equation}\label{eq:nabpsi-Lp-1-todo}
\|[D_0(\calH^\st X)^\sharp]_2\|_{\Ld^p(\R^d;\Ld^q(\Omega))}\,\lesssim_{q,\delta}\,\|X\|_{\Md^1_{p,q+\delta}(\Omega)}.
\end{equation}
By definition, $(\calH^\st X)^\sharp=\nabla w_X$ is the unique stationary gradient solution in $\Ld^2_\loc(\R^d;\Ld^2(\Omega))$ of
\[-\nabla\cdot\Aa\nabla w_X=\nabla\cdot X^\sharp,\qquad\text{in $\R^d$},\]
with $\expec{\nabla w_X}=0$,
and we deduce the following relation by taking the Malliavin derivative in this equation,
\[-\nabla\cdot\Aa\nabla D_0w_X=\nabla\cdot D_0 X^\sharp+\nabla\cdot( D_0\Aa)\nabla w_X.\]
The annealed $\Ld^p$ regularity estimate of Theorem~\ref{th:CZ-ann} then gives
\begin{multline*}
\|[D_0(\calH^\st X)^\sharp]_2\|_{\Ld^p(\R^d;\Ld^q(\Omega))}\,=\,\|[\nabla D_0w_X]_2\|_{\Ld^p(\R^d;\Ld^q(\Omega))}\\
\,\lesssim_{p,q,\delta}\,\|[D_0X^\sharp]_2\|_{\Ld^p(\R^d;\Ld^{q+\delta}(\Omega))}+\|[(D_0\Aa)\nabla w_X]_2\|_{\Ld^p(\R^d;\Ld^{q+\delta}(\Omega))}.
\end{multline*}
In view of~\eqref{eq:def-A}--\eqref{eq:G-rep}, we find $|D_0\Aa(x)|\lesssim|\calC_0(x)|$, and the integrability condition~\eqref{eq:cov-L1} then allows to bound the above by
\begin{equation}\label{eq:bnd-DT}
\|[D_0(\calH^\st X)^\sharp]_2\|_{\Ld^p(\R^d;\Ld^q(\Omega))}\,\lesssim_{p,q,\delta}\,\|[ D_0X^\sharp]_2\|_{\Ld^p(\R^d;\Ld^{q+\delta}(\Omega))}+
\|[(\calH^\st X)^\sharp]_2^\flat\|_{\Ld^{q+\delta}(\Omega)}.
\end{equation}
Combining this with~\eqref{eq:nabpsi-Lp-1}, and recalling the definition~\eqref{eq:def-Mpqk} of the norm of $\Md^1_{p,q}(\Omega)$, the claim~\eqref{eq:nabpsi-Lp-1-todo} follows.
\end{proof}

\begin{rem}\label{rem:refMS-choice}
We comment on our definition of refined Malliavin--Sobolev spaces and we compare it with the slightly different choice in~\cite{AKL-16}.
Removing local quadratic averages in the definition~\eqref{eq:def-Mpqk}, and using identity~\eqref{eq:D0psi-Dpsi0}, the norm of $\Md^k_{p,q}(\Omega)$ would be reduced to the following,
\begin{equation*}
\|X\|_{\Ld^q(\Omega)}+\sum_{j=1}^k\|D^jX\|_{\Ld^p((\R^d)^j;\Ld^q(\Omega))},
\end{equation*}
which would then coincide with the choice in~\cite{AKL-16} provided that $\Ld^p((\R^d)^j;\Ld^q(\Omega))$ is further replaced by $\Ld^q(\Omega;\Ld^p((\R^d)^j))$. This makes two differences with~\cite{AKL-16}: we add local quadratic averages and we reverse spatial and probabilistic norms. This choice is designed to ensure the validity of the above regularity result for $\calH^\st$, cf.~Corollary~\ref{cor:CZ-M}.
\end{rem}

\subsection{Schwartz distributions on the probability space}\label{sec:R'om}
We denote by $\Sc'(\Omega)$  the space of continuous linear functionals on $\Sc(\Omega)$, which is viewed as a space of Schwartz-like distributions on the probability space.\footnote{The topology on $\Sc(\Omega)$ is chosen as follows: a sequence $(X_\e)_\e\subset\Sc(\Omega)$ is said to converge to $X$ in~$\Sc(\Omega)$ as $\e\downarrow0$ if there exists $n\ge1$ such that for all $\e$ small enough the element $X_\e$ can be represented as $X_\e=g_\e(\xi(h^1_\e),\ldots,\xi(h^n_\e))$ with $h^j_\e\to h^j$ in $C^\infty_c(\R^d)^\kappa$, $g_\e\to g$ in $C^\infty_b(\R^n)$, and $X=g(\xi(h^1),\ldots,\xi(h^n))$.}
It is easily checked that $\Sc'(\Omega)$, endowed with the weak-* topology, is separable and that $\Sc(\Omega)$ embeds as a dense linear subspace.
For $X\in\Sc'(\Omega)$, the stationary extension $X^\sharp\in C^\infty(\R^d;\Sc'(\Omega))$ and the stationary gradient~$\nabla^\st X\in\Sc'(\Omega)$ are defined by duality via the following relations for all~$Y\in\Sc(\Omega)$,
\begin{eqnarray*}
\langle X^\sharp(x),Y\rangle_{\Sc'(\Omega),\Sc(\Omega)}&:=&\langle X,Y^\sharp(-x)\rangle_{\Sc'(\Omega),\Sc(\Omega)},\\
\langle\nabla^\st X,Y\rangle_{\Sc'(\Omega),\Sc(\Omega)}&:=&-\langle X,\nabla^\st Y\rangle_{\Sc'(\Omega),\Sc(\Omega)}.
\end{eqnarray*}
Conversely, an element $\psi\in C^\infty(\R^d;\Sc'(\Omega))$ is said to be stationary if it coincides with the stationary extension of an element $\psi^\flat\in\Sc'(\Omega)$.
For $X\in\Sc'(\Omega)$ and $Y\in\Sc(\Omega)$, we use for simplicity the abusive notation $\expec{XY}=\langle X,Y\rangle_{\Sc'(\Omega),\Sc(\Omega)}$ for the duality product.

As for usual Schwartz distributions, we may refine the space $\Sc'(\Omega)$ by considering dual Malliavin--Sobolev spaces:
we denote by $(\Md^k_{p,q})'(\Omega)\subset\Sc'(\Omega)$ the dual space of $\Md^k_{p,q}(\Omega)$, that is, the space of continuous linear functionals on $\Md^k_{p,q}(\Omega)$.
For all $1\le p,q<\infty$ and integers $k\ge0$, the space $(\Md^k_{p,q})'(\Omega)$ is weakly-* separable and $\Sc(\Omega)$ embeds as a dense linear subspace.
Such dual spaces were already used in~\cite{AKL-16}, and might also be compared to some extent to the Kondratiev and Hida spaces, see e.g.~\cite{HOUZ-96}.

\section{Weak corrector theory}\label{sec:weak-cor}

This section is devoted to the proof of Theorem~\ref{th:weak-cor0}: while only the first $\ell=\lceil\frac d2\rceil$ correctors can be constructed as stationary fields with bounded moments, see~Theorem~\ref{th:cor}, we show that twice as many stationary correctors can be constructed in a distributional sense.
We do not know whether this is optimal in general.
More precisely, we prove the following result in dual Malliavin--Sobolev spaces, which is the cornerstone of our approach to the Bourgain--Spencer conjecture.

\begin{theor}[Weak correctors]\label{th:weak-cor}$ $
Let $d>1$. Higher-order weak correctors $\{\varphi^n\}_{0\le n\le d}$, flux correctors $\{\sigma^n\}_{0\le n\le d}$, fluxes $\{q^n\}_{1\le n\le d}$, and effective tensors $\{\bar\Aa^n\}_{1\le n\le d}$ are uniquely well-defined iteratively as follows:
\begin{enumerate}[$\bullet$]
\item $\varphi^0:=1$ and for all $1\le n\le d$ we define $\varphi^n:=(\varphi^n_{i_1\ldots i_n})_{1\le i_1,\ldots,i_n\le d}$ where $\varphi^n_{i_1\ldots i_n}$ is the unique distributional solution in $W^{1,\infty}_\loc(\R^d;(\Md^1_{p,q})'(\Omega))$ for some $1<p,q<\infty$ of
\begin{equation}\label{eq:phidef2}
-\nabla\cdot\Aa\nabla\varphi^n_{i_1\ldots i_n}=\nabla\cdot\big((\Aa\varphi^{n-1}_{i_1\ldots i_{n-1}}-\sigma^{n-1}_{i_1\ldots i_{n-1}})\,\ee_{i_n}\big),\qquad\text{in $\R^d$},
\end{equation}
such that $\varphi^n$ is stationary with $\expec{\varphi^n}=0$ for all $n<d$, and such that $\nabla\varphi^d$ is stationary with $\expec{\nabla\varphi^d}=0$ and with anchoring $\int_B\varphi^d=0$.
\smallskip\item $\sigma^0:=0$ and for all $1\le n\le d$ we define $\sigma^n:=(\sigma^n_{i_1\ldots i_n})_{1\le i_1,\ldots,i_n\le d}$ where $\sigma^n_{i_1\ldots i_n}$ is the unique distributional solution in $W^{1,\infty}_\loc(\R^d;(\Md^1_{p,q})'(\Omega)^{d\times d})$ for some $1<p,q<\infty$, with skew-symmetric matrix values, of
\begin{equation*}
-\triangle\sigma^n_{i_1\ldots i_n}=\nabla\times q^n_{i_1\ldots i_n},\qquad\text{in $\R^d$},
\end{equation*}
such that $\sigma^n$ is stationary with $\expec{\sigma^n}=0$ for all $n<d$, and such that $\nabla\sigma^d$ is stationary with $\expec{\nabla\sigma^d}=0$ and with anchoring $\int_B\sigma^d=0$. In particular, it satisfies
\begin{equation*}
\nabla\cdot\sigma^n_{i_1\ldots i_n}=q^n_{i_1\ldots i_n},\qquad\text{in $\R^d$}.
\end{equation*}
\item For all $1\le n\le d$ we define $q^n:=(q^n_{i_1\ldots i_n})_{1\le i_1,\ldots,i_n\le d}$ with $q^n_{i_1\ldots i_n}$
given by
\begin{equation}\label{eq:qdef2}
q^n_{i_1\ldots i_n}:=\Aa\nabla\varphi^{n}_{i_1\ldots i_n}+(\Aa\varphi^{n-1}_{i_1\ldots i_{n-1}}-\sigma^{n-1}_{i_1\ldots i_{n-1}})\,\ee_{i_n}-\bar\Aa_{i_1\ldots i_{n-1}}^{n}\ee_{i_n}.
\end{equation}
\item For all $1\le n\le d$ we define $\bar\Aa^n:=(\bar\Aa^n_{i_1\ldots i_{n-1}})_{1\le i_1,\ldots,i_{n-1}\le d}$ with $\bar\Aa^n_{i_1\ldots i_{n-1}}$ the matrix given by
\begin{equation}\label{eq:baradef2}
\bar\Aa^n_{i_1\ldots i_{n-1}}\ee_{i_n}:=\expecm{\Aa\big(\nabla\varphi^{n}_{i_1\ldots i_n}+\varphi^{n-1}_{i_1\ldots i_{n-1}}\ee_{i_n}\big)}.
\end{equation}
\end{enumerate}
For $ n<\ell=\lceil\frac d2\rceil$, these weak correctors $\varphi^n,\sigma^n$ coincide with the strong correctors in Theorem~\ref{th:cor}.
In addition, the following estimates hold:
\begin{enumerate}[(i)]
\item \emph{Weak corrector estimates:}
for all $1\le n<d$, $X\in\Sc(\Omega)$, $r<\frac dn$, and $\delta>0$,
\begin{equation*}
\qquad\big|\expecm{X(\varphi^n,\sigma^n)(x)}\big|+\big|\expecm{X(\nabla\varphi^{n+1},\nabla\sigma^{n+1})(x)}\big|
\,\lesssim_{r,\delta}\,\|[ D_{0}X^\sharp]_2\|_{\Ld^r(\R^d;\Ld^{2+\delta}(\Omega))},
\end{equation*}
and at critical order $n=d$, for all $X\in\Sc(\Omega)$, $r<\frac{d}{d-\eta}$, and $\eta,\delta>0$,
\begin{equation*}
\qquad\big|\expecm{X(\varphi^d,\sigma^d)(x)}\big|
\,\lesssim_{r,\eta,\delta}\,\langle x\rangle^\eta\| [D_0X^\sharp]_2\|_{\Ld^r(\R^d;\Ld^{2+\delta}(\Omega))}.
\end{equation*}
\item \emph{Weak fluctuation scaling:}
for all $1\le n\le d$, $X\in\Sc(\Omega)$, $g\in C^\infty_c(\R^d)$, $p,r$ with $\frac1p+\frac1r>1+\frac{n-1}d$, and $\delta>0$,
\begin{equation*}
\qquad\bigg|\expecM{X\int_{\R^d}g\,(\nabla\varphi^n,\nabla\sigma^n)}\bigg|
\,\lesssim_{p,r,\delta}\,\|g\|_{\Ld^p(\R^d)} \| [D_0X^\sharp]_2\|_{\Ld^r(\R^d;\Ld^{2+\delta}(\Omega))}.
\end{equation*}
\end{enumerate}
For $n\le\ell=\lceil\frac d2\rceil$, the space $\Ld^{2+\delta}(\Omega)$ can be replaced by $\Ld^{1+\delta}(\Omega)$ in these estimates.
\qedhere
\end{theor}

\begin{rem}[Explicit 1D case]
In dimension $d=1$, the first corrector is explicit, see~Remark~\ref{rems1}(b). It is easily checked to be uniquely defined as the stationary solution~$\varphi^1$ of~\eqref{eq:phidef2} in $W^{1,\infty}_\loc(\R;\Sc'(\Omega))$ with $\E[\varphi^1]=0$, and its explicit formula entails for all $X\in\Sc(\Omega)$, $g\in C^\infty_c(\R^d)$, and $\frac1p+\frac1r\ge1$,
\begin{align*}
\big|\expecmm{X\varphi^1(x)}\big|&~~\lesssim~~\|DX\|_{\Ld^1(\R^d;\Ld^{1}(\Omega))},\\
\bigg|\expecM{X\int_{\R^d}g\,\nabla\varphi^1}\bigg|&~~\lesssim~~\|g\|_{\Ld^p(\R^d)}\|DX\|_{\Ld^r(\R^d;\Ld^1(\Omega))}.
\qedhere
\end{align*}
\end{rem}

\subsection{Periodic approximation}
In order to prove Theorem~\ref{th:weak-cor}, we proceed by periodic approximation.
For $L\ge3$, we replace the covariance function $\calC$ by its $L$-periodization,
\[\calC_L(x)\,:=\,\sum_{z\in L\Z^d}\calC(x+z),\]
we consider the $\R^\kappa$-valued centered stationary Gaussian random field $G_L$ on $\R^d$ with covariance function~$\calC_L$, and we define the associated coefficient field
\begin{equation}\label{eq:def-AL}
\Aa_L(x)\,:=a_0(G_L(x)).
\end{equation}
As the covariance function $\calC_L$ is $L$-periodic, we note that $G_L$ and $\Aa_L$ are $L$-periodic almost surely.
In this periodic setting, using Poincaré's inequality on the periodic cell \mbox{$Q_L=[-\frac L2,\frac L2)^d$}, the associated periodic correctors are obviously well-defined to all orders.

\begin{lem}[Periodized correctors]\label{def:cor}
Let $d\ge1$. Higher-order periodized correctors $\{\varphi^n_L\}_{n\ge0}$, flux correctors $\{\sigma^n_L\}_{n\ge0}$, fluxes $\{q^n_L\}_{n\ge1}$, and effective tensors $\{\bar\Aa^n_L\}_{n\ge1}$ are uniquely well-defined iteratively as follows:
\begin{enumerate}[$\bullet$]
\item $\varphi^0_L:=1$ and for all $n\ge1$ we define $\varphi^n_L:=(\varphi^n_{L;i_1\ldots i_n})_{1\le i_1,\ldots,i_n\le d}$ where $\varphi^n_{L;i_1\ldots i_n}$ is the unique weak solution in $H^1_\per(Q_L;\Ld^2(\Omega))$, with $\fint_{\T_L^d}\varphi^n_{L;i_1\ldots i_n}=0$, of
\begin{equation}\label{eq:phin-Lper}
\textstyle
-\nabla\cdot\Aa_L\nabla\varphi^n_{L;i_1\ldots i_n}=\nabla\cdot\big((\Aa_L\varphi^{n-1}_{L;i_1\ldots i_{n-1}}-\sigma^{n-1}_{L;i_1\ldots i_{n-1}})\,\ee_{i_n}\big),\qquad\text{in $Q_L$}.
\end{equation}
\item $\sigma^0_L:=0$ and for all $n\ge1$ we define $\sigma^n_L:=(\sigma^n_{L;i_1\ldots i_n})_{1\le i_1,\ldots,i_n\le d}$ where $\sigma^n_{L;i_1\ldots i_n}$ is the unique weak solution in $H^1_\per(Q_L;\Ld^2(\Omega)^{d\times d})$, with $\fint_{\T_L^d}{\sigma_{L;i_1\ldots i_n}^n}=0$ and with skew-symmetric matrix values, of
\begin{equation*}
\textstyle
-\triangle\sigma^n_{L;i_1\ldots i_n}=\nabla\times q^n_{L;i_1\ldots i_n},\qquad\text{in $Q_L$}.
\end{equation*}
In particular, it satisfies
\begin{equation}\label{eq:sigman-Lper}
\textstyle
\nabla\cdot\sigma^n_{L;i_1\ldots i_n}=q^n_{L;i_1\ldots i_n},\qquad\text{in $Q_L$}.
\end{equation}
\item For all $n\ge1$ we define $q^n_L:=(q^n_{L;i_1\ldots i_n})_{1\le i_1,\ldots,i_n\le d}$ with $q^n_{L;i_1\ldots i_n}$
given by
\begin{equation}
q^n_{L;i_1\ldots i_n}:=\Aa_L\nabla\varphi^{n}_{L;i_1\ldots i_n}+(\Aa_L\varphi^{n-1}_{L;i_1\ldots i_{n-1}}-\sigma^{n-1}_{L;i_1\ldots i_{n-1}})\,\ee_{i_n}-\bar\Aa_{L;i_1\ldots i_{n-1}}^{n}\ee_{i_n}.
\end{equation}
\item For all $n\ge1$ we define $\bar\Aa^n_L:=(\bar\Aa^n_{L;i_1\ldots i_{n-1}})_{1\le i_1,\ldots,i_{n-1}\le d}$ with $\bar\Aa^n_{L;i_1\ldots i_{n-1}}$ the matrix given by
\begin{equation}\label{eq:bara-Lper}
\bar\Aa^n_{L;i_1\ldots i_{n-1}}\ee_{i_n}:=\fint_{Q_L}{\Aa_L\big(\nabla\varphi^{n}_{L;i_1\ldots i_n}+\varphi^{n-1}_{L;i_1\ldots i_{n-1}}\ee_{i_n}\big)}.
\qedhere
\end{equation}
\end{enumerate}
\end{lem}

\begin{proof}
Let indices $i_1,\ldots,i_n$ be fixed and omitted in the notation. If $\varphi_L^{n-1},\sigma_L^{n-1}$ are defined in $H^1_\per(Q_L;\Ld^2(\Omega))$, then the Lax--Milgram theorem together with Poincaré's inequality ensures that the corrector equation~\eqref{eq:phin-Lper} admits a unique solution $\varphi^n_L$ in $H^1_\per(Q_L;\Ld^2(\Omega))$ with $\fint_{Q_L}\varphi_L^n=0$.
Next, the same argument ensures that the flux corrector equation $-\triangle\sigma^n_L=\nabla\times q_L^n$ admits a unique solution $\sigma^n_L$ in $H^1_\per(Q_L;\Ld^2(\Omega))$ with $\fint_{Q_L}\sigma_L^n=0$. Taking the divergence in this flux corrector equation and noting that the corrector equation~\eqref{eq:phin-Lper} entails $\nabla\cdot q_L^n=0$, we find $-\triangle(\nabla\cdot\sigma_L^n)=-\triangle q_L^n$, which proves the relation $\nabla\cdot\sigma_L^n=q_L^n$ since both quantities have vanishing average.
Energy estimates for $\varphi_L^n,\sigma_L^n$ and Poincaré's inequality give
\[L^{-1}\|(\phi_L^n,\sigma_L^n)\|_{\Ld^2(Q_L)}\,\lesssim\,\|(\nabla\phi_L^n,\nabla\sigma_L^n)\|_{\Ld^2(Q_L)}\,\lesssim\,\|(\phi_L^{n-1},\sigma_L^{n-1})\|_{\Ld^2(Q_L)}.\]
The conclusion follows by iteration.
\end{proof}

We may naturally couple the periodized field $G_L$ with $G$ via the underlying white noise~$\xi$ in the representation~\eqref{eq:G-rep}.
Indeed, denoting by $\calC_{0,L}$ the $L$-periodization of~$\calC_0$,
\[\calC_{0,L}(x)\,:=\,\sum_{z\in L\Z^d}\calC_0(x+z),\]
and noting that $\calC_L(x)=\int_{Q_L}\calC_{0,L}(x-y)\,\calC_{0,L}(y)\,dy$,
the following representation formula holds,
\begin{equation}\label{eq:rep-GL}
G_L(x)\,=\,
\int_{Q_L}\calC_{0,L}(x-\cdot)\,\xi.
\end{equation}
In particular, the periodized field $G_L$ is now constructed on the same probability space~$(\Omega,\Pm)$ as $G$, and is $\sigma(\xi|_{Q_L})$-measurable. 
We denote by $\Ld^q(\Omega)$ the subspace of $\sigma(\xi|_{Q_L})$-measurable elements of $\Ld^q(\Omega)$, and we define the dense subspace $\Sc(\Omega)$ as in~\eqref{eq:defRom-0} with test functions~$h_j$ supported in $Q_L$.
Periodized correctors $\varphi^n_L,\sigma^n_L$ are elements of $H^1_\per(Q_L;\Ld^2(\Omega))$.

The periodized Gaussian field $G_L$ is stationary in the sense that its law is invariant under spatial translations, but its stationarity is clearly no longer described via the same action~$T$ as in Section~\ref{sec:stat}.
Denote by $\xi_L:=\sum_{z\in L\Z^d}(\xi\mathds1_{Q_L})(\cdot+z)$ the periodization of the restriction~$\xi|_{Q_L}$.
Since $\xi|_{Q_L}=\xi_L|_{Q_L}$, the translation $x\mapsto \xi_L(\cdot+x)|_{Q_L}$ induces an action $T_L=\{T_{L;x}\}_{x\in\R^d}$ of the additive group $(\R^d,+)$ on the space of $\sigma(\xi|_{Q_L})$-measurable random variables.
As the law of $\xi_L(\cdot+x)|_{Q_L}$ does not depend on the shift $x\in\R^d$, the map $T_{L;x}$ is an isometry on $\Ld^q(\Omega)$ for all $1\le q\le\infty$.
Also note that, for any $\sigma(\xi|_{Q_L})$-measurable random variable $X_L$, the map $(x,\omega)\mapsto (T_{L;x}X_L)(\omega)$ is $L$-periodic in~$x$, and is stochastically continuous and jointly measurable on $\R^d\times\Omega$.

In this periodized setting, stationarity is defined as follows: an $L$-periodic random field $\psi_L\in\Ld^1(Q_L;\Ld^1(\Omega))$ is stationary if there exists a random variable $X_L\in\Ld^1(\Omega)$ such that $\psi_L(x,\omega)=(T_{L;x}X_L)(\omega)$.
As before, for all $q$, this provides a canonical isomorphism between random variables in $\Ld^q(\Omega)$ and $L$-periodic stationary random fields in $\Ld^q(Q_L;\Ld^q(\Omega))$: for $X_L\in\Ld^q(\Omega)$ we define its $L$-periodic stationary extension $X_L^\sharp(x,\omega)=(T_{L;x}X_L)(\omega)$, and for an $L$-periodic stationary field $\psi_L\in\Ld^q(Q_L;\Ld^q(\Omega))$ we define $\psi_L^\flat\in\Ld^q(\Omega)$ such that $\psi_L(x,\omega)=(T_{L;x}\psi_L^\flat)(\omega)$.

\subsection{Weak bounds on periodized correctors}
The following result states that the weak corrector estimates of Theorem~\ref{th:weak-cor} hold for periodized correctors uniformly with respect to the period.
The argument is as follows: to estimate weak expressions like $\expec{X_L\varphi_L^n}$, since the corrector $\varphi^n_L$ roughly takes the form $[(-\nabla\cdot\Aa_L\nabla)^{-1}\nabla]^{n}\Aa_L$, we may migrate half of the iterated heterogeneous Riesz operators $(-\nabla\cdot\Aa_L\nabla)^{-1}\nabla$ to the test function $X_L$, so that in the end we only need strong $\Ld^2$ estimates on $\varphi^{m}_L$ for $m\le\frac n2$, which then allows to define twice as many weak correctors.
We already mentioned this symmetrization trick in a related work with Otto~\cite[Remark~2.5]{DO1} (see also~\cite[Theorem~3.5]{Pouch-19}), but this is the first time that it really plays a key role.

\begin{prop}[Weak bounds on periodized correctors]$ $\label{prop:weak-cor}
\begin{enumerate}[(i)]
\item \emph{Weak corrector estimates:}
for all $1\le n<d$, $X_L\in\Sc(\Omega)$, $r<\frac dn$, and $\delta>0$,
\begin{multline}\label{eq:bnd-phin-1}
\qquad\big|\expecm{X_L(\varphi_L^n,\sigma_L^n)(x)}\big|+\big|\expecm{X_L(\nabla\varphi_L^{n+1},\nabla\sigma_L^{n+1})(x)}\big|\\
\qquad\,\lesssim_{r,\delta}\,\|[ D_{0}X_L^\sharp]_2\|_{\Ld^r(Q_L;\Ld^{2+\delta}(\Omega))},
\end{multline}
and at critical order~$n=d$, for all $X_L\in\Sc(\Omega)$, $\eta,\delta>0$, and $r<\frac{d}{d-\eta}$,
\begin{equation}\label{eq:bnd-phin-2}
\qquad\bigg|\expecM{X_L\Big((\varphi_L^d,\sigma_L^d)(x)-\fint_{B}(\varphi_L^d,\sigma_L^d)\Big)}\bigg|
\,\lesssim_{r,\eta,\delta}\,\langle x\rangle^\eta\|[ D_{0}X_L^\sharp]_2\|_{\Ld^r(Q_L;\Ld^{2+\delta}(\Omega))}.
\end{equation}
\item \emph{Weak fluctuation scaling:}
for all $1\le n\le d$, $X_L\in\Sc(\Omega)$, $g\in C^\infty_\per(Q_L)$, $p,r$ with \mbox{$\frac1p+\frac1r>1+\frac{n-1}d$}, and~$\delta>0$,
\begin{equation}\label{eq:average-phin}
\qquad\bigg|\expecM{X_L\int_{Q_L}g_L\,(\nabla\varphi_L^n,\nabla\sigma_L^n)}\bigg|
\,\lesssim_{p,\delta}\,\|g_L\|_{\Ld^p(Q_L)} \|[ D_{0}X_L^\sharp]_2\|_{\Ld^r(Q_L;\Ld^{2+\delta}(\Omega))}.
\end{equation}
\end{enumerate}
For~$n\le\ell=\lceil\frac d2\rceil$, the space $\Ld^{2+\delta}(\Omega)$ can be replaced by $\Ld^{1+\delta}(\Omega)$ in these estimates.
\end{prop}

\begin{proof}
We split the proof into six steps.

\medskip
\step1 Useful short-hand notation.\\
Given an operator $T=(T_\alpha)_{1\le\alpha\le n}:\Ld^2(Q_L)\to\Ld^2(Q_L)^n$, understood as acting via $Tg=(T_\alpha g)_{1\le \alpha\le n}$,
and given another operator $T'=(T'_\alpha)_{1\le\alpha\le n'}:\Ld^2(Q_L)\to\Ld^2(Q_L)^{n'}$, we define their cartesian product as $(T,T')=(T_1,\ldots,T_n,T_1',\ldots,T_n'):\Ld^2(Q_L)\to\Ld^2(Q_L)^{n+n'}$, and their composition as $TT'=(T_\alpha T_\beta')_{1\le \alpha\le n,1\le \beta\le n'}:\Ld^2(Q_L)\to\Ld^2(Q_L)^{nn'}$.
In particular, this means that we do not perform matrix contractions when composing matrix-valued operators: for instance, letting $\calH_{L}:=\nabla(-\nabla\cdot\Aa_L\nabla)^{-1}\nabla$, and viewing $\Aa_L$ as a multiplication operator, the product operator $\calH_L\Aa_L$ stands for $\calH_L\Aa_L=(\calH_{L;ij}\Aa_{L;kl})_{1\le i,j,k,l\le d}$.
When computing norms, we use sup-norms on product spaces: more precisely, for $T=(T_\alpha)_{1\le\alpha\le n}$ and $g\in\Ld^2(Q_L)$, we define
\[|Tg|:=\max_{1\le \alpha\le n}|T_\alpha g|,\]
hence for instance
\[|\calH_L\Aa_Lg|=\max_{1\le i,j,k,l\le d}|\calH_{L;ij}\Aa_{L;kl}g|.\]
Similarly, for a vector field $h\in\Ld^2(Q_L)^m$, we define $Th=(T_\alpha h_\beta)_{1\le \alpha\le n,1\le \beta \le m}\in\Ld^2(Q_L)^{nm}$ and
\[|Th|:=\max_{1\le\alpha\le n,1\le \beta\le m}|T_\alpha h_\beta|.\]
Finally, given $h\in\Ld^2(Q_L)^m$, $f\in\Ld^2(Q_L)^p$, and an operator $T:\Ld^2(Q_L)\to\Ld^2(Q_L)^n$, an identity of the form $f=Th$ will be understood as the existence of a matrix $M\in\R^{p\times nm}$ such that $f=M(Th)$. For instance, inverting the periodic Laplacian in the equation for the flux corrector, cf.~\eqref{eq:sigman-Lper},
\[-\triangle\sigma_{L;i_1,\ldots,i_n}^n=\nabla\times q_{L;i_1,\ldots,i_n}^n,\]
we can write $\sigma_{L}^n=\calR_Lq_L^n$ in terms of the Riesz operator $\calR_L:=\triangle^{-1}\nabla$, as the curl simply amounts to applying a suitable projection.
This short-hand notation will prove particularly convenient when analyzing the hierarchy of corrector problems.

\medskip
\step2 Symmetrization: for all $n\ge\ell$, $X_L\in\Sc(\Omega)$, $g_L\in C^\infty_\per(Q_L)$, and~$1\le p,q\le \infty$, we have, using the short-hand notation of Step~1,
\begin{multline}\label{eq:share}
\bigg|\expecM{X_L\int_{Q_L}g_L\,(\nabla\varphi_L^n,\nabla\sigma_L^n)}\bigg|
\,\lesssim\,\big\|\big[ D_0(\calR_L\calK_L^*)^{n-\ell}({X_L^\sharp\ast_L g_L})\big]_2\big\|_{\Ld^p(Q_L;\Ld^q(\Omega))}\\
\times\big\|\big[ D_{0}(\calK_L \calR_L)^{\ell-1}\calK_L1\big]_2\big\|_{\Ld^{p'}(Q_L;\Ld^{q'}(\Omega))},
\end{multline}
where $\ast_L$ stands for the convolution of $L$-periodic functions,
and where we have defined the linear operators
\begin{eqnarray}
\calK_L&:=&\big(\calH_L^\circ,\calH_L,\calH_L^\circ\Aa_L,\calH_L\Aa_L,\calH_L^\circ\Aa_L\calH_L,\calH_L^\circ\Aa_L\calH_L\Aa_L),\label{eq:def-barTL}\\
\calK_L^*&:=&\big(\calH_L^\circ,\calH_L^*,\Aa_L^*\calH^\circ_L,\Aa_L^*\calH_L^*,\calH_L^*\Aa_L^*\calH_L^\circ,\Aa_L^*\calH_L^*\Aa_L^*\calH_L^\circ),\nonumber
\end{eqnarray}
in terms of the following Helmholtz and Riesz operators, which act on $L$-periodic functions and take values in functions with vanishing average,
\begin{gather*}
\calH_L:=\nabla(-\nabla\cdot\Aa_L\nabla)^{-1}\nabla,\qquad
\calH_L^*:=\nabla(-\nabla\cdot\Aa_L^*\nabla)^{-1}\nabla,\\
\calH^\circ_L:=\nabla(-\triangle)^{-1}\nabla,\qquad \calR_L:=\triangle^{-1}\nabla,
\end{gather*}
where $\Aa_L^*$ stands for the pointwise transpose of $\Aa_L$.
We emphasize that we do not share evenly the number of Riesz operators on $X_L^\sharp\ast_Lg_L$ and on $1$ in~\eqref{eq:share}: this is critical as we need to have as little as possible such operators on the test function $X_L^\sharp\ast_L g_L$ in order to obtain the best norm.

\medskip\noindent
We turn to the proof of~\eqref{eq:share} and start by using stationarity in form of
\begin{eqnarray}
\expecM{X_L\int_{Q_L}g_L\,(\nabla\varphi_L^n,\nabla\sigma_L^n)}&=&\expecM{\Big(\int_{Q_L}X_L^\sharp(-x)\, g_L(x)\,dx\Big)(\nabla\varphi_L^{n},\nabla\sigma_L^{n})^\flat}\nonumber\\
&=&\expecm{({X_L^\sharp\ast_L g_L})^\flat(\nabla\varphi_L^n,\nabla\sigma_L^n)^\flat}.\label{eq:rewr-stat-gphin}
\end{eqnarray}
In terms of the operators $\calH_L^\circ,\calH_L,\calR_L$,
the equations for higher-order periodized correctors, flux correctors, and fluxes in Definition~\ref{def:cor} take on the following form,
\begin{eqnarray*}
\nabla\varphi_L^1&=&\calH_L\Aa_L1,\\
\nabla\sigma_L^1&=&\big(\calH_L^\circ\Aa_L\nabla\varphi_L^1\,,\,\calH_L^\circ\Aa_L1\big),
\end{eqnarray*}
and for $n\ge2$,
\begin{eqnarray*}
\nabla\varphi_L^n&=&\big(\calH_L\Aa_L\calR_L\nabla\varphi_L^{n-1}\,,\,\calH_L\calR_L\nabla\sigma_L^{n-1}\big),\\
\nabla\sigma_L^n&=&\big(\calH_L^\circ\Aa_L\nabla\phi_L^n\,,\,\calH_L^\circ\Aa_L\calR_L\nabla\phi_L^{n-1}\,,\,\calH_L^\circ\calR_L \nabla\sigma_L^{n-1}\big).
\end{eqnarray*}
Iterating these identities yields for all $n\ge1$,
\[(\nabla\varphi_L^n,\nabla\sigma_L^n)=(\calK_L\calR_L)^{n-1}\calK_L1.\]
Recall that we use here the notation of Step~1: in particular, we do not perform matrix contractions when composing operators, and these identities for correctors are understood up to applying some constant matrix to the right-hand side.
Inserting this result into~\eqref{eq:rewr-stat-gphin}, and denoting by $\calK_L^\st, \calR^\st_L$ the corresponding lifted operators on $\Ld^2(\Omega)$, we find
\[\bigg|\expecM{X_L\int_{Q_L}g_L\,(\nabla\varphi_L^n,\nabla\sigma_L^n)}\bigg|\,\lesssim\,\big|\expecm{({X_L^\sharp\ast_L g_L})^\flat(\calK_L^\st \calR^\st_L)^{n-1}\calK_L^\st1}\big|,\]
or equivalently, using adjoints,
\[\bigg|\expecM{X_L\int_{Q_L}g_L\,(\nabla\varphi_L^n,\nabla\sigma_L^n)}\bigg|\,\lesssim\,\big|\E\big[{\big(( \calR^{\st}_L\calK_L^{\st,*})^{n-\ell}({X_L^\sharp\ast_L g_L})^\flat\big)\big((\calK_L^\st \calR^\st_L)^{\ell-1}\calK_L^\st1\big)}\big]\big|.\]
Now appealing to Lemma~\ref{lem:Poinc-MS}(iii) (restricted to $\Sc(\Omega)$), the claim~\eqref{eq:share} follows.

\medskip
\step3 Annealed estimates: for all $X_L\in\Sc(\Omega)$, $1<p,q<\infty$, and $\delta>0$,
\begin{eqnarray}
\|[D_{0}\calK_LX_L^\sharp]_2\|_{\Ld^p(Q_L;\Ld^{q}(\Omega))}&\lesssim_{p,q,\delta}&\|[D_{0}X_L^\sharp]_2\|_{\Ld^{p\wedge2}(Q_L;\Ld^{(q\vee2)+\delta}(\Omega))}+|\expec{X_L}|,\qquad\label{eq:TL-bnd}\\
\|[D_{0}\calK_L^*X_L^\sharp]_2\|_{\Ld^p(Q_L;\Ld^{q}(\Omega))}&\lesssim_{p,q,\delta}&\|[D_{0}X_L^\sharp]_2\|_{\Ld^{p\wedge2}(Q_L;\Ld^{(q\vee2)+\delta}(\Omega))},\nonumber
\end{eqnarray}
and for all $\frac{d}{d-1}<p<\infty$ and $1\le q\le\infty$,
\begin{equation}\label{eq:UL-bnd}
\|[D_{0}\calR_LX_L^\sharp]_2\|_{\Ld^p(Q_L;\Ld^{q}(\Omega))}\,\lesssim_p\,\|[ D_{0}X_L^\sharp]_2\|_{\Ld^\frac{dp}{d+p}(Q_L;\Ld^{q}(\Omega))}.
\end{equation}
By the definition~\eqref{eq:def-barTL} of $\calK_L,\calK_L^*$, the result~\eqref{eq:TL-bnd} follows from the following three estimates, for all $1<p,q<\infty$ and $\delta>0$,
\begin{eqnarray}
\|[ D_{0}\calH_L^\circ X_L^\sharp]_2\|_{\Ld^p(Q_L;\Ld^{q}(\Omega))}&\lesssim_{p,q}&\|[ D_{0}X_L^\sharp]_2\|_{\Ld^p(Q_L;\Ld^{q}(\Omega))},\label{eq:TL0-bnd}\\
\|[ D_{0}(\Aa_LX_L^\sharp)]_2\|_{\Ld^p(Q_L;\Ld^{q}(\Omega))}&\lesssim_{p,q}&\|[ D_{0}X_L^\sharp]_2\|_{\Ld^{p\wedge2}(Q_L;\Ld^{q\vee2}(\Omega))}+|\expec{X_L}\!|,\label{eq:aL-bnd}\\
\|[ D_{0}\calH_LX_L^\sharp]_2\|_{\Ld^p(Q_L;\Ld^{q}(\Omega))}&\lesssim_{p,q,\delta}&\|[ D_{0}X_L^\sharp]_2\|_{\Ld^{p\wedge2}(Q_L;\Ld^{(q+\delta)\vee2}(\Omega))}.\label{eq:TL+-bnd}
\end{eqnarray}
We start with the proof of~\eqref{eq:UL-bnd}. As the convolution kernel for the Riesz transform $\calR_L$ is bounded by $x\mapsto C|(x)_L|^{1-d}$, where we have set $(x)_L:=x\mod Q_L$,
this estimate follows from the Hardy--Littlewood--Sobolev inequality in the form
\begin{eqnarray*}
\lefteqn{\|[ D_{0}\calR_LX_L^\sharp]_2\|_{\Ld^p(Q_L;\Ld^q(\Omega))}\,=\,\|[\calR_L D_{0}X_L^\sharp]_2\|_{\Ld^p(Q_L;\Ld^{q}(\Omega))}}\\
&\qquad\lesssim&\bigg(\int_{Q_L}\Big(\int_{Q_L}|(x-y)_L|^{1-d}\,\|[D_{0}X_L^\sharp]_2(y)\|_{\Ld^{q}(\Omega)}\,dy\Big)^pdx\bigg)^\frac1p\\
&\qquad\lesssim_p&\|[ D_{0}X_L^\sharp]_2\|_{\Ld^\frac{dp}{d+p}(Q_L;\Ld^{q}(\Omega))}.
\end{eqnarray*}
We turn to the proof of~\eqref{eq:TL0-bnd}. In view of Banach-valued Fourier multiplier theorems, e.g.~in form of the extrapolation result in~\cite[Theorem~3.15]{Kunstmann-Weis-04}, this follows from the maximal $\Ld^p$ regularity for the Helmholtz projection~$\calH_L^\circ$ with values in $\Ld^q(\Omega)$.

\medskip\noindent
We turn to the proof of~\eqref{eq:aL-bnd}.
In view of~\eqref{eq:def-AL} and~\eqref{eq:rep-GL}, we find $|D_{0}\Aa_L(x)|\lesssim |\calC_{0,L}(x)|$,
and the integrability condition~\eqref{eq:cov-L1} yields $\int_{Q_L}[\calC_{0,L}]_\infty\le\int_{\R^d}[\calC_{0}]_\infty\le1$. We deduce
\begin{eqnarray*}
\lefteqn{\|[D_{0}(\Aa_LX_L^\sharp)]_2\|_{\Ld^p(Q_L;\Ld^{q}(\Omega))}}\\
&\lesssim&\|[ D_{0}X_L^\sharp]_2\|_{\Ld^p(Q_L;\Ld^{q}(\Omega))}+\|[( D_{0}\Aa_L)X_L^\sharp]_2\|_{\Ld^p(Q_L;\Ld^{q}(\Omega))}\\
&\lesssim&\|[ D_{0}X_L^\sharp]_2\|_{\Ld^p(Q_L;\Ld^{q}(\Omega))}+\|[X_L^\sharp]_2^\flat\|_{\Ld^{q}(\Omega)}.
\end{eqnarray*}
Decomposing $\|[X_L^\sharp]_2^\flat\|_{\Ld^{q}(\Omega)}\le \|[(X_L-\expec{X_L})^\sharp]_2^\flat\|_{\Ld^{q}(\Omega)}+|\expec{X_L}\!|$,
appealing to Poincaré's inequality in form of~Lemma~\ref{lem:Poinc-MS}(iv) (restricted to $\Sc(\Omega)$), and using the discrete $\ell^p-\ell^r$ inequality as in~\eqref{eq:ellp-ellr},
the claim~\eqref{eq:aL-bnd} follows.

\medskip\noindent
It remains to establish~\eqref{eq:TL+-bnd}. We start from~\eqref{eq:bnd-DT} in the proof of Corollary~\ref{cor:CZ-M} (restricted to $\Sc(\Omega)$), in form of
\begin{equation*}
\|[D_{0}\calH_LX_L^\sharp]_2\|_{\Ld^p(Q_L;\Ld^q(\Omega))}
\,\lesssim_{p,q,\delta}\,\|[ D_{0}X_L^\sharp]_2\|_{\Ld^p(Q_L;\Ld^{q+\delta}(\Omega))}+\|[\calH_LX_L^\sharp]_2^\flat\|_{\Ld^{q+\delta}(\Omega)}.
\end{equation*}
Noting that $\calH_LX_L^\sharp=\calH_L(X_L-\expec{X_L})^\sharp$, appealing to the stationary $\Ld^q$ regularity estimate of Corollary~\ref{cor:CZ-stat}, and using Poincaré's inequality in form of Lemma~\ref{lem:Poinc-MS}(iv), the last right-hand side term is bounded by
\begin{eqnarray*}
\|[\calH_LX_L^\sharp]_2^\flat\|_{\Ld^{q+\delta}(\Omega)}&\lesssim_{q,\delta}&\|[(X_L-\expec{X_L})^\sharp]_2^\flat\|_{\Ld^{q+2\delta}(\Omega)}\\
&\lesssim_{q,\delta}&\|[ D_{0}X_L^\sharp]_2\|_{\Ld^2(Q_L;\Ld^{(q+2\delta)\vee2}(\Omega))},
\end{eqnarray*}
and the claim~\eqref{eq:TL+-bnd} follows.

\medskip
\step4 Weak fluctuation scaling~\eqref{eq:average-phin}.\\
We start with the case $\ell<n\le d$.
Starting point is~\eqref{eq:share} with $q=2$, that is, for any~$1\le p\le\infty$,
\begin{multline}\label{eq:share-bis}
\bigg|\expecM{X_L\int_{Q_L}g_L\,(\nabla\varphi_L^n,\nabla\sigma_L^n)}\bigg|
\,\lesssim\,\big\|\big[ D_0(\calR_L\calK_L^*)^{n-\ell}({X_L^\sharp\ast_L g_L})\big]_2\big\|_{\Ld^p(Q_L;\Ld^2(\Omega))}\\
\times\big\|\big[ D_{0}(\calK_L \calR_L)^{\ell-1}\calK_L1\big]_2\big\|_{\Ld^{p'}(Q_L;\Ld^{2}(\Omega))},
\end{multline}
and we turn to the estimation of the two right-hand side factors.
On the one hand, for all~$1<p<\frac{d}{\ell-1}$ (which ensures $\frac{d(p'\wedge2)}{d+(\ell-1)(p'\wedge2)}>1$),
an iterative use of~\eqref{eq:TL-bnd} and~\eqref{eq:UL-bnd} yields
\begin{equation}\label{eq:estim-1fact}
\big\|\big[ D_{0}(\calK_L \calR_L)^{\ell-1}\calK_L1\big]_2\big\|_{\Ld^{p'}(Q_L;\Ld^{2}(\Omega))}\,\lesssim_{p}\,1.
\end{equation}
On the other hand, for all $1<p<\infty$ with $\frac{dp}{d+p}>1$, the estimate~\eqref{eq:UL-bnd} gives
\begin{multline*}
\big\|\big[D_0(\calR_L\calK_L^*)^{n-\ell}({X_L^\sharp\ast_L g_L})\big]_2\big\|_{\Ld^p(Q_L;\Ld^2(\Omega))}\\
\,\lesssim_{p}\,\big\|\big[ D_0\calK_L^*(\calR_L\calK_L^*)^{n-\ell-1}({X_L^\sharp\ast_L g_L})\big]_2\big\|_{\Ld^\frac{dp}{d+p}(Q_L;\Ld^2(\Omega))},
\end{multline*}
and thus, noting that the condition $p<\frac{d}{\ell-1}$ implies $\frac{dp}{d+p}<2$, an iterative use of~\eqref{eq:TL-bnd} and~\eqref{eq:UL-bnd} leads us to the following: for all $1<p<\frac{d}{\ell-1}$ with $\frac{dp}{d+(n-\ell)p}>1$, and all $\delta>0$,
\begin{equation*}
\big\|\big[D_0(\calR_L\calK_L^*)^{n-\ell}({X_L^\sharp\ast_L g_L})\big]_2\big\|_{\Ld^p(Q_L;\Ld^2(\Omega))}
\,\lesssim_{p,\delta}\,\|[ D_0({X_L^\sharp\ast_L g_L})]_2\|_{\Ld^\frac{dp}{d+(n-\ell)p}(Q_L;\Ld^{2+\delta}(\Omega))}.
\end{equation*}
Inserting these estimates back into~\eqref{eq:share-bis}, we deduce for all $1<p<\frac{d}{\ell-1}$ with $\frac{dp}{d+(n-\ell)p}>1$, and all $\delta>0$,
\begin{equation*}
\bigg|\expecM{X_L\int_{Q_L}g_L\,(\nabla\varphi_L^n,\nabla\sigma_L^n)}\bigg|
\,\lesssim_{p,\delta}\,\|[ D_0({X_L^\sharp\ast_L g_L})]_2\|_{\Ld^\frac{dp}{d+(n-\ell)p}(Q_L;\Ld^{2+\delta}(\Omega))}.
\end{equation*}
Letting $r=\frac{dp}{d+(n-\ell)p}$, this can be reformulated as follows: for all $1<r<\frac{d}{n-1}$ and $\delta>0$,
\begin{equation}\label{eq:bnd-aver-nabla-Lper}
\bigg|\expecM{X_L\int_{Q_L}g_L\,(\nabla\varphi_L^n,\nabla\sigma_L^n)}\bigg|
\,\lesssim_{r,\delta}\,\|[ D_0({X_L^\sharp\ast_L g_L})]_2\|_{\Ld^r(Q_L;\Ld^{2+\delta}(\Omega))}.
\end{equation}
Next, writing $D_0({X_L^\sharp\ast_L g_L})=g_L\ast_L(D_{0}X_L^\sharp)$ and appealing to Young's convolution inequality, the conclusion~\eqref{eq:average-phin} follows for $\ell<n\le d$.

\medskip\noindent
We turn to the case $n\le\ell$, which is simpler and is essentially already contained in~\cite[Lemma~7.1]{DO1}.
More precisely, it suffices to avoid the symmetrization trick of Step~1 in that case: we replace~\eqref{eq:share-bis} by
\begin{multline*}
\bigg|\expecM{X_L\int_{Q_L}g_L\,(\nabla\varphi_L^n,\nabla\sigma_L^n)}\bigg|
\,\lesssim\,\|[D_{0}({X_L^\sharp\ast_L g_L})]_2\|_{\Ld^p(Q_L;\Ld^q(\Omega))}\\
\times\|[D_{0}(\calK_L \calR_L)^{n-1}\calK_L1]_2\|_{\Ld^{p'}(Q_L;\Ld^{q'}(\Omega))}.
\end{multline*}
By an iterative use of~\eqref{eq:TL-bnd} and~\eqref{eq:UL-bnd}, arguing as for~\eqref{eq:estim-1fact}, we find for all $1<p<\frac{d}{n-1}$ and $q>1$,
\begin{equation}\label{eq:bnd-aver-nabla-Lper-bis}
\bigg|\expecM{X_L\int_{Q_L}g_L\,(\nabla\varphi_L^n,\nabla\sigma_L^n)}\bigg|
\,\lesssim_p\,\|[D_{0}({X_L^\sharp\ast_L g_L})]_2\|_{\Ld^p(Q_L;\Ld^q(\Omega))}.
\end{equation}
Writing again $D_0({X_L^\sharp\ast_L g_L})=g_L\ast_L(D_{0}X_L^\sharp)$, appealing to Young's convolution inequality, and setting $q=1+\delta>1$, the conclusion~\eqref{eq:average-phin} follows for $1\le n\le \ell$ with $\Ld^{2+\delta}(\Omega)$ replaced by $\Ld^{1+\delta}(\Omega)$.

\medskip
\step5 Weak corrector estimates~\eqref{eq:bnd-phin-1}.\\
Let $1\le n<d$. The bound on $\nabla(\varphi_L^{n+1},\sigma_L^{n+1})$ already follows from~\eqref{eq:bnd-aver-nabla-Lper} and~\eqref{eq:bnd-aver-nabla-Lper-bis} up to replacing~$g_L$ by a Dirac measure, and it remains to estimate $(\varphi_L^{n},\sigma_L^{n})$ itself.
In view of the anchoring $\fint_{Q_L}(\varphi^n_L,\sigma_L^n)=0$, we may decompose
\begin{eqnarray}
(\varphi_L^n,\sigma_L^n)(x)&=&(\varphi_L^n,\sigma_L^n)(x)-\fint_{Q_L}(\varphi_L^n,\sigma_L^n)\nonumber\\
&=&\int_0^L\Big(\fint_{Q_t(x)}\tfrac{x-y}t\cdot(\nabla\varphi_L^n,\nabla\sigma_L^n)(y)\,dy\Big)\,dt,\label{eq:decomp-phix0L}
\end{eqnarray}
and we separately analyze the contribution of the integral for $t\in[0,1]$ and for $t\in[1,L]$.
On the one hand, appealing to~\eqref{eq:average-phin} with $p=1$, we find for all $r<\frac{d}{n-1}$ and $\delta>0$,
\begin{eqnarray}
\lefteqn{\expecM{X_L\int_0^1\Big(\fint_{Q_t(x)}\tfrac{x-y}t\cdot(\nabla\varphi_L^n,\nabla\sigma_L^n)(y)\,dy\Big)dt}}\nonumber\\
&\qquad\lesssim_{r,\delta}&\|[D_{0}X_L^\sharp]_2\|_{\Ld^r(Q_L;\Ld^{2+\delta}(\Omega))}\int_0^1\fint_{Q_t(x)}\big|\tfrac{x-y}t\big|\,dy\,dt\nonumber\\
&\qquad\le&\|[D_{0}X_L^\sharp]_2\|_{\Ld^r(Q_L;\Ld^{2+\delta}(\Omega))}.\label{eq:bnd-phix-01}
\end{eqnarray}
On the other hand, appealing again to~\eqref{eq:average-phin}, we obtain for all $p,r$ with $\frac1p+\frac1r>1+\frac{n-1}d$, and $\delta>0$,
\begin{multline*}
\expecM{X_L\int_1^L\Big(\fint_{Q_t(x)}\tfrac{x-y}t\cdot(\nabla\varphi_L^n,\nabla\sigma_L^n)(y)\,dy\Big)dt}\\
\lesssim_{p,r,\delta}\,\|[D_{0}X_L^\sharp]_2\|_{\Ld^r(Q_L;\Ld^{2+\delta}(\Omega))}\int_1^L t^{-d(1-\frac1p)}\,dt,
\end{multline*}
and thus, for all $r<\frac dn$ and $\delta>0$, choosing $p>\frac{d}{d-1}$ such that $\frac1p+\frac1r>1+\frac{n-1}d$,
\begin{equation*}
\expecM{X_L\int_1^L\Big(\fint_{B_t(x)}\tfrac{x-y}t\cdot(\nabla\varphi_L^n,\nabla\sigma_L^n)(y)\,dy\Big)dt}
\,\lesssim_{r,\delta}\,\|[D_{0}X_L^\sharp]_2\|_{\Ld^r(Q_L;\Ld^{2+\delta}(\Omega))}.
\end{equation*}
Combining this with~\eqref{eq:decomp-phix0L} and~\eqref{eq:bnd-phix-01}, the conclusion~\eqref{eq:bnd-phin-1} follows.
In the case~$n\le\ell$, as in~\eqref{eq:average-phin}, the space $\Ld^{2+\delta}(\Omega)$ can be replaced by $\Ld^{1+\delta}(\Omega)$.

\medskip
\step6 Weak sublinearity estimate~\eqref{eq:bnd-phin-2} for $(\varphi^d_L,\sigma_L^d)$.\\
In view of the anchoring $\fint_B(\varphi_L^d,\sigma_L^d)=0$, we can decompose
\begin{eqnarray}
\lefteqn{(\varphi_L^d,\sigma_L^d)(x)-\fint_{B}(\varphi_L^d,\sigma_L^d)}\nonumber\\
&=&\Big((\varphi_L^d,\sigma_L^d)(x)-\fint_{B(x)}(\varphi_L^d,\sigma_L^d)\Big)+\Big(\fint_{B(x)}(\varphi_L^d,\sigma_L^d)-\fint_{B}(\varphi_L^d,\sigma_L^d)\Big)\nonumber\\
&=&\int_0^1\Big(\fint_{B_t(x)}\tfrac{x-y}t\cdot(\nabla\varphi_L^d,\nabla\sigma_L^d)(y)\,dy\Big)\,dt+\Big(\fint_{B(x)}(\varphi_L^d,\sigma_L^d)-\fint_{B}(\varphi_L^d,\sigma_L^d)\Big),\qquad\label{eq:decomp-phix+d}
\end{eqnarray}
and we analyze the two right-hand side terms separately.
The first one is estimated as in~\eqref{eq:bnd-phix-01}, for all $r<\frac{d}{d-1}$ and $\delta>0$,
\begin{equation}\label{eq:bnd-phix-01+d}
\expecM{X_L\int_0^1\Big(\fint_{B_t(x)}\tfrac{x-y}t\cdot(\nabla\varphi_L^d,\nabla\sigma_L^d)(y)\,dy\Big)dt}
\,\lesssim_{r,\delta}\,\|[D_{0}X_L^\sharp]_2\|_{\Ld^r(Q_L;\Ld^{2+\delta}(\Omega))}.
\end{equation}
We turn to the second right-hand side term in~\eqref{eq:decomp-phix+d}.
Denoting by $\chi_L\in H^1_\per(Q_L)$ the unique periodic mean-zero solution of
\[-\triangle\chi_{L}=\frac{\mathds1_{B}}{|B|}-L^{-d},\qquad\text{in $Q_L$},\]
and setting $\nabla\chi_{x,L}:=\nabla\chi_L(\cdot-x)-\nabla\chi_L$,
we can write
\[\expecM{X_L\Big(\fint_{B(x)}(\varphi_L^d,\sigma_L^d)-\fint_{B}(\varphi_L^d,\sigma_L^d)\Big)}\,=\,\expecM{X_L\int_{Q_L}\nabla\chi_{x,L}\cdot(\nabla\varphi_L^d,\nabla\sigma_L^d)}.\]
Appealing to~\eqref{eq:average-phin}, we deduce for all $p,r$ with $\frac1p+\frac1r>1+\frac{d-1}d$, and $\delta>0$,
\begin{multline}\label{eq:bnd-phix-02+d}
\expecM{X_L\Big(\fint_{B(x)}(\varphi_L^d,\sigma_L^d)-\fint_{B}(\varphi_L^d,\sigma_L^d)\Big)}\\
\,\lesssim_{p,r,\delta}\,\|\nabla\chi_{x,L}\|_{\Ld^p(Q_L)}\|[D_{0}X_L^\sharp]_2\|_{\Ld^r(Q_L;\Ld^{2+\delta}(\Omega))}.
\end{multline}
Noting that a direct computation yields
\[\|\nabla\chi_{x,L}\|_{\Ld^p(\R^d)}\,\lesssim\,\left\{\begin{array}{lll}
1&:&p>\frac{d}{d-1},\\
\log(2+|x|)^{1-\frac1d}&:&p=\frac{d}{d-1},\\
\langle x\rangle^{d(\frac1p+\frac1d-1)}&:&p<\frac{d}{d-1},
\end{array}\right.\]
and combining~\eqref{eq:decomp-phix+d}, \eqref{eq:bnd-phix-01+d}, and~\eqref{eq:bnd-phix-02+d}, the conclusion~\eqref{eq:bnd-phin-2} easily follows.
In the case~$n\le\ell$, as in~\eqref{eq:average-phin}, the space $\Ld^{2+\delta}(\Omega)$ can be replaced by $\Ld^{1+\delta}(\Omega)$.
\end{proof}

\subsection{Proof of Theorem~\ref{th:weak-cor}}\label{sec:concl-approx}
It remains to pass to the infinite-period limit in the periodized corrector estimates of Proposition~\ref{prop:weak-cor}, and to show that the limiting correctors are indeed uniquely defined in a weak sense.
We start with the following uniqueness statement in dual Malliavin--Sobolev spaces.

\begin{lem}[Uniqueness]\label{lem:unique}
If $\psi$ belongs to $W^{1,\infty}_\loc(\R^d;(\Md^1_{p,q})'(\Omega))$ for some $1< p,q<\infty$, if $\nabla\psi$ is stationary with $\expec{\nabla\psi}=0$, and if the equation $-\nabla\cdot\Aa\nabla\psi=0$ is satisfied in $\R^d$ in the distributional sense, then there holds $\nabla\psi=0$.
\end{lem}

\begin{proof}
Let $1<p,q<\infty$ be fixed. We split the proof into three steps.

\nopagebreak\medskip
\step1
Special case $\Aa=\Id$: If $\psi$ belongs to $W^{1,\infty}_\loc(\R^d;(\Md^1_{p,q})'(\Omega))$, if $\nabla\psi$ is stationary with $\expec{\nabla\psi}=0$, and if the equation $-\triangle\psi=0$ is satisfied in $\R^d$ in the distributional sense, then there holds $\nabla\psi=0$.

\medskip\noindent
Given $X\in\Sc(\Omega)$, we consider the function $u_X=\expec{X\psi}\in W^{1,\infty}_{\loc}(\R^d)$.
As $\nabla\psi$ is stationary, we can write
\begin{equation}\label{eq:rep-nabuX}
\nabla u_X(x)=\expec{X\nabla\psi(x)}=\expecm{(T_{-x}X)(\nabla\psi)^\flat}.
\end{equation}
As $X\in\Sc(\Omega)$, note that this identity actually ensures $u_X\in C^{\infty}(\R^d)$.
For $(\nabla\psi)^\flat\in(\Md^1_{p,q})'(\Omega)$, as the norm of $T_{-x}X$ in $\Md^1_{p,q}(\Omega)$ is bounded independently of~$x$, we also deduce from~\eqref{eq:rep-nabuX} that $\nabla u_X$ is uniformly bounded. As the relation $-\triangle\psi=0$ implies that $u_X$ is harmonic, we deduce that $\nabla u_X$ is constant.
In particular, identity~\eqref{eq:rep-nabuX} then yields for all $R>0$,
\[\textstyle\nabla u_X(x)=\fint_{B_R}\nabla u_X=\expecm{(\fint_{B_R}X^\sharp)(\nabla\psi)^\flat}.\]
Noting that $\fint_{B_R}X^\sharp$ converges to $\expec{X}$ in $\Md^1_{p,q}(\Omega)$ as $R\uparrow\infty$ by ergodicity, and recalling that $\expec{\nabla\psi}=0$, we deduce
\[\nabla u_X(x)=\expec{X}\E[(\nabla\psi)^\flat]=0.\]
This means $\expec{X\nabla\psi}=\nabla u_X=0$, and the arbitrariness of $X$ yields the claim $\nabla\psi=0$.

\medskip
\step2
Density result:
If $\psi$ belongs to $W_\loc^{1,\infty}(\R^d;(\Md^1_{p,q})'(\Omega))$ with $\nabla\psi$ stationary and $\expec{\nabla\psi}=0$, then there exists a sequence $(X_n)_n\subset\Sc(\Omega)$ such that $\nabla^\st X_n$ converges to $(\nabla\psi)^\flat$ weakly-* in $(\Md^1_{p,q})'(\Omega)$.

\medskip\noindent
Let $\psi\in W^{1,\infty}_\loc(\R^d;(\Md^1_{p,q})'(\Omega))$ with $\nabla\psi$ stationary and $\expec{\nabla\psi}=0$.
As $(\nabla\psi)^\flat\in(\Md^1_{p,q})'(\Omega)^d$, by density, there is a sequence $(Y_m)_m\subset\Sc(\Omega)^d$ that converges to $(\nabla\psi)^\flat$ weakly-* in $(\Md^1_{p,q})'(\Omega)$.
Next, for all $m,r\ge1$, we define $X_{m,r}\in H^1(\Omega)$ as the unique solution of
\begin{equation}\label{eq:approxY}
(\tfrac1r-\triangle^\st)X_{m,r}=-\nabla^\st\cdot Y_m.
\end{equation}
The energy estimate for this equation takes the form
\[r^{-\frac12}\|X_{m,r}\|_{\Ld^2(\Omega)}+\|\nabla^\st X_{m,r}\|_{\Ld^2(\Omega)}\lesssim\|Y_m\|_{\Ld^2(\Omega)},\]
and the proof of Corollary~\ref{cor:CZ-M} (in the simpler case $\Aa=\Id$) yields, up to a duality argument,
\[\|\nabla^\st X_{m,r}\|_{(\Md^1_{p,q})'(\Omega)}\lesssim_{p,q}\|Y_m\|_{(\Md^1_{p,q})'(\Omega)}.\]
A weak compactness argument then allows to find a diagonal subsequence $(\nabla^\st X_n)_n\subset(\nabla^\st X_{m,r})_{m,r}$ that converges weakly-* in $(\Md^1_{p,q})'(\Omega)$ to some limit $L$ with $\expec{L}=0$,
in such a way that we may pass to the limit in~\eqref{eq:approxY} and obtain
\begin{equation}\label{eq:relation-L-pre}
-\nabla^\st\cdot L=-\nabla^\st\cdot (\nabla\psi)^\flat.
\end{equation}
In addition, the gradient structure of $(\nabla^\st X_n)_n$ allows us to write the limit as $L=(\nabla\phi)^\flat$ for some $\phi\in W^{1,\infty}_\loc(\R^d;(\Md^1_{p,q})'(\Omega))$.
The relation~\eqref{eq:relation-L-pre} then becomes
\begin{equation*}
-\nabla^\st\cdot(\nabla\phi)^\flat=-\nabla^\st\cdot(\nabla\psi)^\flat.
\end{equation*}
In view of Step~1, we can conclude $L^\sharp=\nabla\phi=\nabla\psi$.
Further approximating the $X_{n}$'s by elements in $\Sc(\Omega)$, the claimed density result follows.

\medskip
\step3 Conclusion.\\
Let $\psi$ belong to $W^{1,\infty}_\loc(\R^d;(\Md^1_{p,q})'(\Omega))$ with $\nabla\psi$ stationary and $\expec{\nabla\psi}=0$, and assume that $-\nabla\cdot\Aa\nabla\psi=0$ is satisfied in $\R^d$ in the distributional sense.
In particular, $(\nabla\psi)^\flat$ belongs to $(\Md^1_{p,q})'(\Omega)$ and satisfies for all $X\in\Sc(\Omega)$,
\begin{equation}\label{eq:nabphi-test}
\expecm{\nabla^\st X\cdot\Aa^\flat(\nabla\psi)^\flat}\,=\,0.
\end{equation}
Given $Y\in\Sc(\Omega)^d$, consider the unique solution $w_Y\in W^{1,\infty}_\loc(\R^d;\Ld^2(\Omega))$ of
\[-\nabla\cdot\Aa^*{\nabla w_Y}=\nabla\cdot Y^\sharp,\]
such that $\nabla w_Y$ is stationary with $\expec{\nabla w_Y}=0$ and with anchoring $\int_Bw_Y=0$.
For all $X\in\Sc(\Omega)$, this equation yields
\begin{equation}\label{eq:nabpsi-test}
\expecm{(\nabla w_Y)^\flat\cdot\Aa^\flat\nabla^\st X}=-\expecm{Y\cdot\nabla^\st X},
\end{equation}
and we note that Corollary~\ref{cor:CZ-M} with $Y\in\Sc(\Omega)^d$ ensures that $w_Y\in W^{1,\infty}_\loc(\R^d;\Md^1_{p,q}(\Omega))$ for all $1<p,q<\infty$. Appealing to the density result of Step~2, we deduce from~\eqref{eq:nabphi-test}--\eqref{eq:nabpsi-test},
\[0\,=\,\expecm{(\nabla w_Y)^\flat\cdot\Aa^\flat(\nabla\psi)^\flat}\,=\,-\expecm{Y\cdot(\nabla\psi)^\flat},\]
and the arbitrariness of $Y$ yields the conclusion $\nabla\psi=0$.
\end{proof}

With this uniqueness result at hand, we may now pass to the infinite-period limit in the estimates of Proposition~\ref{prop:weak-cor} and conclude with the proof of Theorem~\ref{th:weak-cor}.

\begin{proof}[Proof of Theorem~\ref{th:weak-cor}]
For all $X\in\Sc(\Omega)$ and $g\in C^\infty_c(\R^d)$, there is some large enough $L_0>0$ such that $X\in\Sc(\Omega)$ and $g\in C^\infty_\per(Q_L)$ for all $L\ge L_0$. Therefore, the bounds of Proposition~\ref{prop:weak-cor} ensure that for all $n\le d$ the sequence $(\varphi_L^n,\sigma_L^n)_{L\ge1}$ is uniformly bounded in $C^\infty(\R^d;\Sc'(\Omega))$. By a weak-* compactness argument, we deduce that up to extraction of a subsequence $(\varphi_L^n,\sigma_L^n)$ converges weakly-* to some $(\varphi^n,\sigma^n)$ in $C^\infty(\R^d;\Sc'(\Omega))$, such that $(\varphi^n,\sigma^n)$ is stationary with $\expec{(\varphi^n,\sigma^n)}=0$ for all $n<d$, and such that $(\nabla\varphi^d,\nabla\sigma^d)$ is stationary with $\expecmm{(\nabla\varphi^d,\nabla\sigma^d)}=0$ and anchoring $\int_B(\varphi^d,\sigma^d)=0$.
Passing to the limit in the periodized corrector equations~\eqref{eq:phin-Lper}--\eqref{eq:bara-Lper} and in the estimates~\eqref{eq:bnd-phin-1}--\eqref{eq:average-phin}, we easily deduce that the limiting collection $\{\varphi^n,\sigma^n\}_{1\le n\le d}$ satisfies equations~\eqref{eq:phidef2}--\eqref{eq:baradef2} in the distributional sense as well as the corresponding limiting estimates.
Finally, Lemma~\ref{lem:unique} ensures that those limiting objects are uniquely defined.
\end{proof}

\section{Effective description of ensemble averages}\label{sec:BSconj}
This section is devoted to the proof of Theorem~\ref{th:main}.
More precisely, we establish the following error estimates. We do not know whether this is optimal in general.

\begin{theor}[Effective description of ensemble averages]\label{th:main-2}
Let $d>1$,
let the higher-order effective tensors $\{\bar\Aa^n\}_{1\le n\le d}$ be defined in Theorem~\ref{th:weak-cor},
and let the associated higher-order homogenized solution operators be given in Definition~\ref{def:high-eq}.
Then, for all $f\in C^\infty_c(\R^d)^d$,
we have for all $n<d$, $\frac{d}{d-n}\vee2<p<\infty$, and $\eta>0$,
\begin{align}
&\hspace{-0.4cm}\big\|\big[\,\expec{\nabla u_{\e,f}}-\nabla\bar\Oc_\e^n[f]\,\big]_{2;\e}\big\|_{\Ld^p(\R^d)}+\big\|\big[\,\E[\Aa(\tfrac\cdot\e)\nabla u_{\e,f}]-\bar\Ac_\e^n\nabla\bar\Oc_\e^n[f]\,\big]_{2;\e}\big\|_{\Ld^p(\R^d)}\label{eq:bnd-nabu-n-fin}\\
&\hspace{8.3cm}\,\lesssim_{p,\eta}\,\e^{n}\|\langle\cdot\rangle^\eta\langle\nabla\rangle^{2n}f\|_{\Ld^p(\R^d)},\nonumber
\end{align}
while at critical order $n=d$ we have for all $0<\eta\ll1$ and $\frac{d}{\eta}<p<\infty$,
\begin{align}
&\hspace{-0.3cm}\big\|\big[\,\expec{\nabla u_{\e,f}}-\nabla\bar\Oc_\e^d[f]\,\big]_{2;\e}\big\|_{\Ld^p(\R^d)}+\big\|\big[\,\E[\Aa(\tfrac\cdot\e)\nabla u_{\e,f}]-\bar\Ac_\e^d\nabla\bar\Oc_\e^d[f]\,\big]_{2;\e}\big\|_{\Ld^p(\R^d)}\label{eq:bnd-nabu-d-fin}\\
&\hspace{8cm}\,\lesssim_{p,\eta}\,\e^{d-\eta}\,\|\langle\cdot\rangle^\eta\langle\nabla\rangle^{2d}f\|_{\Ld^p(\R^d)}.\nonumber\qedhere
\end{align}
\end{theor}

\begin{proof}
We consider the ensemble-averaged field and the ensemble-averaged flux separately, and we split the proof into four steps.

\medskip
\step1 Preliminary: Given a test function $g\in C^\infty_c(\R^d)^d$, we consider the unique almost sure gradient solution $\nabla v_{\e,g}$ in $\Ld^2(\R^d)^d$ of the auxiliary equation
\begin{equation}\label{eq:aux-veps}
-\nabla\cdot\Aa^*(\tfrac\cdot\e)\nabla v_{\e,g}\,=\,\nabla\cdot g,\qquad\text{in $\R^d$},
\end{equation}
and we show that for all $h\in C^\infty_c(\R^d)$, $X\in\Sc(\Omega)$, $0<\eta,\delta\ll1$, and $1<r<\frac{d}{d-\eta}$,
\begin{multline}\label{eq:bnd-Aeps}
A_\e(X,h,g):=\expec{X\int_{\R^d}\e^d(\varphi^d,\sigma^d)(\tfrac\cdot\e)\,h\nabla v_{\e,g}}\\
\,\lesssim_{r,\eta,\delta}\,
\e^{d-\eta}\big(\|X\|_{\Ld^{\infty}(\Omega)}+\|[D_0X^\sharp]_\infty\|_{\Ld^r(\R^d;\Ld^{\infty}(\Omega))}\big)\|\langle \cdot\rangle^\eta h\|_{\Ld^{r'}(\R^d)}\|[g]_{2+\delta;\e}\|_{\Ld^r(\R^d)}.
\end{multline}
Starting point is the weak corrector estimate of Theorem~\ref{th:weak-cor}(i), which yields for all $\eta,\delta>0$ and $r<\frac{d}{d-\eta}$,
\begin{align*}
A_\e(X,h,g)&~\le~\e^d\int_{\R^d}|h(x)|\big|\expecm{X\nabla v_{\e,g}(x)\,(\varphi^d,\sigma^d)(\tfrac x\e)}\!\big|\,dx\\
&~\lesssim_{r,\eta,\delta}~\e^{d-\eta}\int_{\R^d}\langle x\rangle^\eta |h(x)|\\
&\qquad\qquad\times\bigg(\int_{\R^d}\Big\|\Big(\fint_{B(y)}\big|D_0\big(X^\sharp(y')\,T_{y'}\nabla v_{\e,g}(x)\big)\big|^2dy'\Big)^\frac12\Big\|_{\Ld^{2+\delta}(\Omega)}^rdy\bigg)^\frac1rdx.
\end{align*}
Using Leibniz' rule for $D_0$, we split the right-hand side into two terms: $A_\e^1(X,h,g)$ and $A_\e^2(X,h,g)$, where the first one corresponds to $(D_0X^\sharp(y'))T_{y'}\nabla v_{\e,g}(x)$, while the second one corresponds to $X^\sharp(y')(D_0T_{y'}\nabla v_{\e,g}(x))$.
Straightforward estimates lead to
\begin{eqnarray*}
A^1_\e(X,h,g)&\lesssim_{r,\eta,\delta}&\e^{d-\eta}\|[D_0X^\sharp]_{\infty}\|_{\Ld^r(\R^d;\Ld^{\infty}(\Omega))}\int_{\R^d}\langle x\rangle^\eta |h(x)|\|\nabla v_{\e,g}(x)\|_{\Ld^{2+\delta}(\Omega)}\,dx,\\
A^2_\e(X,h,g)&\lesssim_{r,\eta,\delta}&\e^{d-\eta}\|X\|_{\Ld^{\infty}(\Omega)}\\
&&\hspace{-0.3cm}\times\int_{\R^d}\langle x\rangle^\eta |h(x)|\bigg(\int_{\R^d}\Big\|\Big(\fint_{B(y)}|D_0T_{y'}\nabla v_{\e,g}(x)|^{2}dy'\Big)^\frac1{2}\Big\|_{\Ld^{2+\delta}(\Omega)}^rdy\bigg)^\frac1r\,dx,
\end{eqnarray*}
and thus, by Hölder's inequality,
\begin{eqnarray}
\hspace{-0.6cm}A^1_\e(X,h,g)&\lesssim_{r,\eta,\delta}&\e^{d-\eta}\|[D_0X^\sharp]_{\infty}\|_{\Ld^r(\R^d;\Ld^{\infty}(\Omega))}\|\langle\cdot\rangle^\eta h\|_{\Ld^{r'}(\R^d)}\|\nabla v_{\e,g}\|_{\Ld^r(\R^d;\Ld^{2+\delta}(\Omega))},\label{eq:bnd-pre-A1eps}\\
\hspace{-0.6cm}A^2_\e(X,h,g)&\lesssim_{r,\eta,\delta}&\e^{d-\eta}\|X\|_{\Ld^{\infty}(\Omega)}\|\langle\cdot\rangle^\eta h\|_{\Ld^{r'}(\R^d)}\nonumber\\
\hspace{-0.6cm}&&\quad\times\bigg(\int_{\R^d\times\R^d}\Big\|\Big(\fint_{B(y)}|D_0T_{y'}\nabla v_{\e,g}(x)|^{2}dy'\Big)^\frac1{2}\Big\|_{\Ld^{2+\delta}(\Omega)}^rdxdy\bigg)^\frac1r.\label{eq:bnd-pre-A2eps}
\end{eqnarray}
We now estimate the two terms separately, and we start with $A_\e^1$.
In order to apply annealed maximal regularity theory, we smuggle in local averages at the scale $\e$, for $r\le2$,
\[\|\nabla v_{\e,g}\|_{\Ld^r(\R^d;\Ld^{2+\delta}(\Omega))}\,\le\,\|[\nabla v_{\e,g}]_{2+\delta;\e}\|_{\Ld^r(\R^d;\Ld^{2+\delta}(\Omega))},\]
and we need to replace $[\cdot]_{2+\delta;\e}$ by local quadratic averages $[\cdot]_{2;\e}$.
For that purpose, since $\nabla v_{\e,g}$ satisfies the elliptic equation~\eqref{eq:aux-veps},
we appeal to Meyers' perturbative argument in the following form, for $0<\delta\ll1$,
\[\Big(\fint_{B_\e(x)}|\nabla v_{\e,g}|^{2+\delta}\Big)^\frac1{2+\delta}\,\lesssim\,\Big(\fint_{B_{2\e}(x)}|\nabla v_{\e,g}|^{2}\Big)^\frac1{2}+\Big(\fint_{B_{2\e}(x)}|g|^{2+\delta}\Big)^\frac1{2+\delta},\]
so that the above becomes
\[\|\nabla v_{\e,g}\|_{\Ld^r(\R^d;\Ld^{2+\delta}(\Omega))}\,\lesssim\,\|[\nabla v_{\e,g}]_{2;\e}\|_{\Ld^r(\R^d;\Ld^{2+\delta}(\Omega))}+\|[g]_{2+\delta;\e}\|_{\Ld^r(\R^d)}.\]
We are now in position to apply the ($\e$-rescaled) annealed maximal regularity estimate of Theorem~\ref{th:CZ-ann}, which yields for all $1<r\le2$,
\begin{equation}\label{eq:apriori-ann-Meyers}
\|\nabla v_{\e,g}\|_{\Ld^r(\R^d;\Ld^{2+\delta}(\Omega))}\,\lesssim_{r}\,\|[g]_{2+\delta;\e}\|_{\Ld^r(\R^d)}.
\end{equation}
Inserting this into~\eqref{eq:bnd-pre-A1eps}, we deduce for all $1<r<\frac{d}{d-\eta}$ and $0<\eta,\delta\ll1$,
\begin{equation}\label{eq:bnd-Aeps1}
A_\e^1(X,h,g)\,\lesssim_{r,\eta,\delta}\,\e^{d-\eta}\|[D_0X^\sharp]_\infty\|_{\Ld^r(\R^d;\Ld^{\infty}(\Omega))}\|\langle \cdot\rangle^\eta h\|_{\Ld^{r'}(\R^d)}\|[g]_{2+\delta;\e}\|_{\Ld^r(\R^d)}.
\end{equation}
We turn to the estimate on $A_\e^2$.
Acting with $T$ and taking the Malliavin derivative in the equation~\eqref{eq:aux-veps} for~$\nabla v_{\e,g}$, we find
\begin{equation}\label{eq:DTnabv}
-\nabla\cdot\big(\Aa^*(\tfrac\cdot\e+y)(D_0T_y\nabla v_{\e,g})\big)=\nabla\cdot \big((D_0\Aa^*(\tfrac\cdot\e+y))T_y\nabla v_{\e,g}\big).
\end{equation}
Smuggling in local averages at the scale $\e$, using Meyers' argument and the annealed maximal regularity as above for equation~\eqref{eq:DTnabv}, we find for all $1<r\le2$,
\begin{eqnarray*}
\lefteqn{\bigg(\int_{\R^d\times\R^d}\Big\|\Big(\fint_{B(y)}|D_0T_{y'}\nabla v_{\e,g}(x)|^{2}dy'\Big)^\frac1{2}\Big\|_{\Ld^{2+\delta}(\Omega)}^rdxdy\bigg)^\frac1r}\\
&\le&\bigg(\int_{\R^d\times\R^d}\Big\|\Big(\fint_{B_\e(x)\times B(y)}|D_0T_{y'}\nabla v_{\e,g}(x')|^{2+\delta}dx'dy'\Big)^\frac1{2+\delta}\Big\|_{\Ld^{2+\delta}(\Omega)}^rdxdy\bigg)^\frac1r\\
&\lesssim_{r,\delta}&\bigg(\int_{\R^d\times\R^d}\Big\|\Big(\fint_{B_\e(x)\times B(y)}|(D_0\Aa^*(\tfrac{x'}\e+y))T_y\nabla v_{\e,g}(x')|^{2+\delta}dx'dy'\Big)^\frac1{2+\delta}\Big\|_{\Ld^{2+2\delta}(\Omega)}^rdxdy\bigg)^\frac1r.
\end{eqnarray*}
Recalling that $|D_0\Aa^*(x)|\lesssim|\calC_0(x)|$ and using the integrability condition~\eqref{eq:cov-L1}, straightforward computations lead to
\begin{multline*}
\lefteqn{\bigg(\int_{\R^d\times\R^d}\Big\|\Big(\fint_{B(y)}|D_0T_{y'}\nabla v_{\e,g}(x)|^{2}dy'\Big)^\frac1{2}\Big\|_{\Ld^{2+\delta}(\Omega)}^rdxdy\bigg)^\frac1r}\\
\,\lesssim_{r,\delta}\,\|[\nabla v_{\e,g}]_{2+2\delta;\e}\|_{\Ld^r(\R^d;\Ld^{2+2\delta}(\Omega))}.
\end{multline*}
Using again Meyers' argument and the annealed maximal regularity similarly as above, we conclude for all $1<r\le2$,
\begin{equation*}
\bigg(\int_{\R^d\times\R^d}\Big\|\Big(\fint_{B(y)}|D_0T_{y'}\nabla v_{\e,g}(x)|^{2}dy'\Big)^\frac1{2}\Big\|_{\Ld^{2+\delta}(\Omega)}^rdxdy\bigg)^\frac1r
\,\lesssim_{r,\delta}\,\|[g]_{2+2\delta;\e}\|_{\Ld^r(\R^d)}.
\end{equation*}
Inserting this into~\eqref{eq:bnd-pre-A2eps}, we deduce for all $1<r<\frac{d}{d-\eta}$ and $0<\eta,\delta\ll1$,
\begin{equation*}
A_\e^2(X,h,g)\,\lesssim_{r,\eta,\delta}\,\e^{d-\eta}\|X\|_{\Ld^{\infty}(\Omega)}\|\langle\cdot\rangle^\eta h\|_{\Ld^{r'}(\R^d)}\|[g]_{2+2\delta;\e}\|_{\Ld^r(\R^d)}.
\end{equation*}
Combined with~\eqref{eq:bnd-Aeps1}, this yields the claim~\eqref{eq:bnd-Aeps}.

\medskip
\step2 Approximation of ensemble-averaged field:
for all $n<d$, $\frac{d}{d-n}\vee2<p<\infty$, and~$\delta>0$,
\begin{equation}\label{eq:bnd-Enabu-n}
\big\|\big[\,\expec{\nabla u_{\e,f}}-\nabla\bar\Oc_\e^n[f]\,\big]_{2-\delta;\e}\big\|_{\Ld^p(\R^d)}
\,\lesssim_{p,\delta}\,\e^{n}\|\langle\nabla\rangle^{2n-1}f\|_{\Ld^p(\R^d)},
\end{equation}
while at critical order $n=d$ we have for all $0<\eta,\delta\ll1$ and $\frac{d}{\eta}<p<\infty$,
\begin{equation}\label{eq:bnd-Enabu-d}
\big\|\big[\,\expec{\nabla u_{\e,f}}-\nabla\bar\Oc_\e^d[f]\,\big]_{2-\delta;\e}\big\|_{\Ld^p(\R^d)}\,\lesssim_{p,\eta,\delta}\,\e^{d-\eta}\,\|\langle\cdot\rangle^\eta\langle\nabla\rangle^{2d-1}f\|_{\Ld^p(\R^d)}.
\end{equation}
Starting point is Proposition~\ref{cor:2scale}, where identity~\eqref{eq:2sc-identity-err} can be continued to higher orders in terms of the weak correctors defined in Theorem~\ref{th:weak-cor}: the following equation holds in the distributional sense on $\R^d\times\Omega$ for all $1\le n\le d$, 
\begin{multline}\label{eq:2sc-identity-err-2}
-\nabla\cdot\Aa(\tfrac\cdot\e)\nabla(u_{\e,f}-\Fc^n_\e[\bar\Oc^n_\e [f]])=\nabla\cdot\Big(\sum_{k=2}^n\sum_{l=n+2-k}^n\e^{k+l-2}\,\bar\Aa^k_{i_1\ldots i_{k-1}}\nabla\nabla^{k-1}_{i_1\ldots i_{k-1}}\tilde u^l_f\Big)\\
+\nabla\cdot\big(\e^n(\Aa\varphi^n_{i_1\ldots i_n}-\sigma^n_{i_1\ldots i_n})(\tfrac\cdot\e)\nabla\nabla^n_{i_1\ldots i_n}\bar \Oc^n_\e [f]\big).
\end{multline}
Given a test function $g\in C^\infty_c(\R^d)^d$, consider the solution $\nabla v_{\e,g}$ of the auxiliary problem~\eqref{eq:aux-veps}.
Testing the latter with $u_{\e,f}-\Fc^n_\e[\bar\Oc^n_\e [f]]$, we find in $\Sc'(\Omega)$,
\[\int_{\R^d} g\cdot\nabla(u_{\e,f}-\Fc^n_\e[\bar \Oc^n_\e [f]])\,=\,-\int_{\R^d} \nabla v_{\e,g}\cdot\Aa(\tfrac\cdot\e)\nabla(u_{\e,f}-\Fc^n_\e[\bar \Oc^n_\e [f]]),\]
where both sides are indeed well-defined in $\Sc'(\Omega)$ in view of the result~\eqref{eq:bnd-Aeps} of Step~1 (together with corresponding estimates for correctors $\varphi^n$ with $n\le d$).
Next, testing equation~\eqref{eq:2sc-identity-err-2} with $v_{\e,g}$, we deduce in $\Sc'(\Omega)$,
\begin{multline*}
\int_{\R^d} g\cdot\nabla(u_{\e,f}-\Fc^n_\e[\bar \Oc^n_\e [f]])
\,=\,\int_{\R^d}\nabla v_{\e,g}\cdot\sum_{k=2}^n\sum_{l=n+2-k}^n\e^{k+l-2}\,\bar\Aa^k_{i_1\ldots i_{k-1}}\nabla\nabla^{k-1}_{i_1\ldots i_{k-1}}\tilde u^l_f\\
+\e^n\int_{\R^d}\nabla v_{\e,g}\cdot(\Aa\varphi^n_{i_1\ldots i_n}-\sigma^n_{i_1\ldots i_n})(\tfrac\cdot\e)\nabla\nabla^n_{i_1\ldots i_n}\bar \Oc^n_\e [f].
\end{multline*}
Taking the expectation of both sides of this identity (or, more precisely, testing this identity with $1\in\Sc(\Omega)$), and noting that the definition of the higher-order two-scale expansion~\eqref{eq:def-2sc-exp} and the centering of weak correctors yield
\[\expecm{\nabla \Fc_\e^n[\bar \Oc^n_\e [f]]}=\nabla\bar \Oc^n_\e[f]+\mathds1_{n=d}\,\e^d\expecm{\varphi^d_{i_1\ldots i_d}(\tfrac\cdot\e)\nabla\nabla^d_{i_1\ldots i_d}\bar \Oc^d_\e [f]},\]
we find
\begin{multline}\label{eq:decomp-Enabu-hom}
\int_{\R^d} g\cdot\big(\expec{\nabla u_{\e,f}}-\nabla\bar \Oc^n_\e[f]\big)
\,=\,\expecM{\int_{\R^d}\nabla v_{\e,g}\cdot\sum_{k=2}^n\sum_{l=n+2-k}^n\e^{k+l-2}\,\bar\Aa^k_{i_1\ldots i_{k-1}}\nabla\nabla^{k-1}_{i_1\ldots i_{k-1}}\tilde u^l_f}\\
+\expecM{\e^n\int_{\R^d}\nabla v_{\e,g}\cdot (\Aa\varphi^n_{i_1\ldots i_n}-\sigma^n_{i_1\ldots i_n})(\tfrac\cdot\e)\nabla\nabla^n_{i_1\ldots i_n}\bar \Oc^d_\e [f]}\\
+\mathds1_{n=d}\,\expecM{\e^d\int_{\R^d}g\cdot\varphi^d_{i_1\ldots i_d}(\tfrac\cdot\e)\nabla\nabla^d_{i_1\ldots i_d}\bar \Oc^d_\e [f]}.
\end{multline}
We focus on the proof of~\eqref{eq:bnd-Enabu-d} for $n=d$, while the proof of~\eqref{eq:bnd-Enabu-n} is similar.
We analyze the three right-hand side terms in the above identity~\eqref{eq:decomp-Enabu-hom} separately.
For the first term, we use Hölder's inequality and we appeal to the annealed maximal regularity theory of Theorem~\ref{th:CZ-ann} in form of $\|[\nabla v_{\e,g}]_{2;\e}\|_{\Ld^r(\R^d;\Ld^2(\Omega))}\lesssim\|[g]_{2;\e}\|_{\Ld^r(\R^d)}$.
For the second term, we apply the result~\eqref{eq:bnd-Aeps} of Step~1, while for the third term the weak corrector estimate of Theorem~\ref{th:weak-cor}(i) suffices.
For all $0<\eta,\delta\ll1$ and $1<r<\frac{d}{d-\eta}$, we deduce
\begin{multline}\label{eq:pre-estim-EDu-bar}
\bigg|\int_{\R^d} g\cdot\big(\expec{\nabla u_{\e,f}}-\nabla\bar \Oc^d_\e[f]\big)\bigg|
\,\lesssim_{r,\eta,\delta}\,\sum_{k=2}^d\sum_{l=d+2-k}^d\e^{k+l-2}\,\|[g]_{2;\e}\|_{\Ld^r(\R^d)}\|\nabla^k\tilde u_f^l\|_{\Ld^{r'}(\R^d)}\\
+\e^{d-\eta}\|[g]_{2+\delta;\e}\|_{\Ld^r(\R^d)}\|\langle\cdot\rangle^\eta\nabla^{d+1}\bar\Oc_\e^d[f]\|_{\Ld^{r'}(\R^d)}.
\end{multline}
Recalling the hierarchy of higher-order homogenized equations in Definition~\ref{def:high-eq}, the standard weighted $\Ld^p$ regularity for the constant-coefficient Poisson equation allows to estimate by iteration, for all $1<p<\infty$, $0\le\eta<d\frac{p-1}p$, $m\ge0$, and $1\le n\le d$,
\[\|\langle\cdot\rangle^\eta\nabla^{m+1}\tilde u_f^n\|_{\Ld^p(\R^d)}\,\lesssim_{n,p,\eta}\,\|\langle\cdot\rangle^\eta\nabla^{m+n-1}f\|_{\Ld^p(\R^d)},\]
and therefore, as $\nabla\bar \Oc_\e^n[f]\,:=\,\sum_{k=1}^n\e^{k-1}\nabla\tilde u^k_f$,
\begin{equation}\label{eq:reg-high-hom}
\|\langle\cdot\rangle^\eta\nabla^{m+1}\bar\Oc_\e^n[f]\|_{\Ld^p(\R^d)}\,\lesssim_{n,p,\eta}\,\|\langle\cdot\rangle^\eta\nabla^{m}\langle\nabla\rangle^{n-1}f\|_{\Ld^p(\R^d)}.
\end{equation}
Using this, the bound~\eqref{eq:pre-estim-EDu-bar} becomes for all $0<\eta,\delta\ll1$ and $1<r<\frac{d}{d-\eta}$,
\begin{equation*}
\bigg|\int_{\R^d} g\cdot\big(\expec{\nabla u_{\e,f}}-\nabla\bar \Oc^d_\e[f]\big)\bigg|
\,\lesssim_{r,\eta,\delta}\,\e^{d-\eta}\|[g]_{2+\delta;\e}\|_{\Ld^r(\R^d)}\|\langle\cdot\rangle^\eta\langle\nabla\rangle^{2d-1}f\|_{\Ld^{r'}(\R^d)}.
\end{equation*}
Taking the supremum over~$g$, this precisely yields the claim~\eqref{eq:bnd-Enabu-d} for the ensemble-averaged field.

\medskip
\step3 Approximation of ensemble-averaged flux:
for all $n<d$, $\frac{d}{d-n}\vee2<p<\infty$, and~$\delta>0$,
\begin{equation}\label{eq:bnd-Eanabu-n}
\big\|\big[\,\E[\Aa(\tfrac\cdot\e)\nabla u_{\e,f}]-\bar\Ac_\e^n\nabla\bar\Oc_\e^n[f]\,\big]_{2-\delta;\e}\big\|_{\Ld^p(\R^d)}
\,\lesssim_{p,\delta}\,\e^{n}\|\langle\nabla\rangle^{2n-1}f\|_{\Ld^p(\R^d)},
\end{equation}
while at critical order $n=d$ we have for all $0<\eta,\delta\ll1$ and $\frac{d}{\eta}<p<\infty$,
\begin{equation}\label{eq:bnd-Eanabu-d}
\big\|\big[\,\E[\Aa(\tfrac\cdot\e)\nabla u_{\e,f}]-\bar\Ac_\e^d\nabla\bar\Oc_\e^d[f]\,\big]_{2-\delta;\e}\big\|_{\Ld^p(\R^d)}
\,\lesssim_{p,\eta,\delta}\,\e^{d-\eta}\,\|\langle\cdot\rangle^\eta\langle\nabla\rangle^{2d-1}f\|_{\Ld^p(\R^d)}.
\end{equation}
Starting point is the same identity~\eqref{eq:2sc-identity-err-2} as in Step~1, which we now test with the unique almost sure gradient solution $\nabla w_{\e,g}$ in $\Ld^2(\R^d)^d$ of the auxiliary problem
\[-\nabla\cdot\Aa^*(\tfrac\cdot\e)\nabla w_{\e,g}=\nabla\cdot \Aa^*(\tfrac\cdot\e)g,\qquad\text{in $\R^d$}.\]
This yields the following identity in $\Sc'(\Omega)$,
\begin{multline*}
\int_{\R^d}g\cdot\Aa(\tfrac\cdot\e)(\nabla u_{\e,f}-\nabla \Fc_\e^n[\bar\Oc_\e^n[f]])\,=\,\int_{\R^d}\nabla w_{\e,g}\cdot\sum_{k=2}^n\sum_{l=n+2-k}^n\e^{k+l-2}\bar\Aa^k_{i_1\ldots i_{k-1}}\nabla\nabla^{k-1}_{i_1\ldots i_{k-1}}\tilde u^l_f\\
+\e^n\int_{\R^d}\nabla w_{\e,g}\cdot(\Aa\varphi^n_{i_1\ldots i_n}-\sigma^n_{i_1\ldots i_n})(\tfrac\cdot\e)\nabla\nabla^n_{i_1\ldots i_n}\bar\Oc_\e^n[f].
\end{multline*}
Taking the expectation of both sides of this identity (or, more precisely, testing this identity with $1\in\Sc(\Omega)$), and noting that the definition of the higher-order two-scale expansion~\eqref{eq:def-2sc-exp} and of higher-order effective tensors~\eqref{eq:baradef2} yield
\[\expec{\Aa(\tfrac\cdot\e)\nabla \Fc_\e^n[\bar\Oc_\e^n[f]]}\,=\,\bar\Ac_\e^n\nabla\bar\Oc_\e^n[f]+\expec{\e^n(\Aa\varphi^n_{i_1\ldots i_n})(\tfrac\cdot\e)\nabla\nabla^n_{i_1\ldots i_n}\bar\Oc_\e^n[f]},\]
we find
\begin{multline*}
\int_{\R^d}g\cdot\big(\E[\Aa(\tfrac\cdot\e)\nabla u_{\e,f}]-\bar\Ac_\e^n\nabla\bar\Oc_\e^n[f]\big)\\
\,=\,\expecM{\int_{\R^d}\nabla w_{\e,g}\cdot\sum_{k=2}^n\sum_{l=n+2-k}^n\e^{k+l-2}\bar\Aa^k_{i_1\ldots i_{k-1}}\nabla\nabla^{k-1}_{i_1\ldots i_{k-1}}\tilde u^l_f}\\
+\expecM{\e^n\int_{\R^d}\nabla w_{\e,g}\cdot(\Aa\varphi^n_{i_1\ldots i_n}-\sigma^n_{i_1\ldots i_n})(\tfrac\cdot\e)\nabla\nabla^n_{i_1\ldots i_n}\bar \Oc^n_\e[f]}\\
+\expecM{\e^n\int_{\R^d}g\cdot(\Aa\varphi^n_{i_1\ldots i_n})(\tfrac\cdot\e)\nabla\nabla^n_{i_1\ldots i_n}\bar \Oc^n_\e[f]}.
\end{multline*}
Noting that the result~\eqref{eq:bnd-Aeps} of Step~1 also holds for $\nabla v_{\e,g}$ replaced by $\nabla w_{\e,g}$,
and arguing as in Step~2, the claim~\eqref{eq:bnd-Eanabu-n}--\eqref{eq:bnd-Eanabu-d} easily follows.

\medskip
\step4 Conclusion.\\
It remains to replace sub-quadratic averages $[\cdot]_{2-\delta}$ in~\eqref{eq:bnd-Enabu-n}--\eqref{eq:bnd-Enabu-d} and in~\eqref{eq:bnd-Eanabu-n}--\eqref{eq:bnd-Eanabu-d} by quadratic averages.
For that purpose, we may for instance appeal to the Sobolev embedding, which yields for all $0<\delta<1$ and $\alpha>\frac{d}{2}\delta$,
\[\|[h]_{2;\e}\|_{\Ld^p(\R^d)}\,\lesssim_{\alpha,\delta}\,\|[\langle\nabla\rangle^{\alpha}h]_{2-\delta;\e}\|_{\Ld^p(\R^d)}.\]
In addition, as by stationarity the operator $f\mapsto\expec{\nabla u_{\e,f}}$ commutes with translations (see also~\cite[Lemma~1.1]{DGL}), and as the same obviously holds for the higher-order homogenized solution operator \mbox{$f\mapsto\nabla\bar\Oc_\e^n[f]$}, we have
\[\langle\nabla\rangle^\alpha\big(\expec{\nabla u_{\e,f}}-\nabla\bar\Oc_\e^n[f]\big)\,=\,\expec{\nabla u_{\e,\langle\nabla\rangle^\alpha f}}-\nabla\bar\Oc_\e^n[\langle\nabla\rangle^\alpha f].\]
Therefore, the result~\eqref{eq:bnd-Enabu-n} of Step~2 yields for all $n<d$, $\frac{d}{d-n}\vee2<p<\infty$, and $\alpha>0$,
\begin{equation*}
\big\|\big[\,\expec{\nabla u_{\e,f}}-\nabla\bar\Oc_\e^n[f]\,\big]_{2;\e}\big\|_{\Ld^p(\R^d)}
\,\lesssim_{p,\delta}\,\e^{n}\|\langle\nabla\rangle^{2n+\alpha-1}f\|_{\Ld^p(\R^d)}.
\end{equation*}
Arguing similarly to upgrade~\eqref{eq:bnd-Enabu-d}, \eqref{eq:bnd-Eanabu-n}, and~\eqref{eq:bnd-Eanabu-d}, the conclusion follows.
\end{proof}

\appendix
\section{Intrinsic description of weak fluctuations}\label{sec:fluct}
Aside from the Bourgain--Spencer conjecture on ensemble averages, another natural question concerns the intrinsic description of fluctuations \mbox{$\nabla u_{\e,f}-\expec{\nabla u_{\e,f}}$} in the weak sense of $\Sc'(\Omega)$.
The proof of Theorem~\ref{th:main} ensures that such ``weak'' fluctuations are of order~$O(\e^{d-})$: for all $f\in C^\infty_c(\R^d)^d$, $X\in\Rc(\Omega)$, and~$\eta>0$,
\begin{equation*}
\big\|\big[\,\E[{X(\nabla u_{\e,f}-\expec{\nabla u_{\e,f}})}]\,\big]_{2;\e}\big\|_{\Ld^\infty(\R^d)}\,\lesssim_{X,f,\eta}\,\e^{d-\eta}.
\end{equation*}
This contrasts with the usual CLT scaling result, e.g.~\cite{DO1}, which states that fluctuations are of order $O(\e^{d/2})$ in the usual strong sense: for all $f,g\in C^\infty_c(\R^d)^d$ and $q<\infty$,
\[\expec{\Big|\int_{\R^d}g\cdot(\nabla u_{\e,f}-\expec{\nabla u_{\e,f}})\Big|^q}^\frac1q\,\lesssim_{f,g,q}\,\e^{\frac d2}.\]
Taking inspiration from our recent work with Gloria and Otto~\cite{DGO1,DO1} on ``strong'' fluctuations, we show that ``weak'' fluctuations can similarly be described intrinsically to relative order~$O(\e^{d/2-})$. This result only involves the standard strong correctors $\{\varphi^n\}_{1\le n\le\ell}$, and we do not know whether this could be improved.

\begin{theor}[Two-scale expansion for ``weak'' fluctuations]\label{th:fluct0}
For all $f\in C^\infty_c(\R^d)^d$,
there exists an $\ell$th-order differential operator $\Xi_\e^{\circ,\ell}[\nabla\cdot]$ with $\e$-rescaled stationary random coefficients expressed in terms of correctors $\{\varphi^n\}_{1\le n\le\ell}$, see~Definition~\ref{def:Xi0n} below, such that there holds
for all $X\in\Rc(\Omega)$, $g\in C^\infty_c(\R^d)^d$,
and $\eta>0$,
\begin{align*}\label{eq:decomp-nabu-Xin00}
&\bigg|\int_{\R^d}g\cdot\e^{-d}\,\expecm{X\big(\nabla u_{\e,f}-\expec{\nabla u_{\e,f}}\!\big)}-\int_{\R^d}(\nabla\bar\Oc_\e^{\ell})^*[g]\cdot\e^{-d}\,\expecm{X\Xi_\e^{\circ,\ell}[\nabla\bar\Oc_\e^\ell f]}\bigg|\\
+~&\bigg|\int_{\R^d}g\cdot\e^{-d}\,\expecm{X\big(\Aa(\tfrac\cdot\e)\nabla u_{\e,f}-\expec{\Aa(\tfrac\cdot\e)\nabla u_{\e,f}}\!\big)}\\
&\hspace{5.5cm}-\int_{\R^d}(\Id+\bar\Ac_\e^\ell\nabla\bar\Oc_\e^{\ell})^*[g]\cdot\e^{-d}\,\expecm{X\Xi_\e^{\circ,\ell}[\nabla\bar\Oc_\e^\ell f]}\bigg|\\
&\hspace{-0.7cm}\,\lesssim_{\eta}~\e^{\frac d2-\eta}\,\|[DX]_2\|_{\Ld^{1}(\R^d;\Ld^2(\Omega))}\|\mu_d\langle\nabla\rangle^{3\ell-2}g\|_{(\Ld^4\cap\Ld^{\infty})(\R^d)}\|\mu_d\langle\nabla\rangle^{3\ell-2}f\|_{(\Ld^4\cap\Ld^{\infty})(\R^d)},
\end{align*}
where the weight $\mu_d$ is given by
\[\mu_d(x):=\left\{\begin{array}{lll}
\langle x\rangle^\frac12&:&\text{$d$ odd},\\
\log(2+|x|)^\frac12&:&\text{$d$ even}.\end{array}\right.\qedhere\]
\end{theor}

\subsection{Intrinsic description of strong fluctuations}\label{sec:commut}
As first noticed by Gu and Mourrat~\cite{GuM}, random fluctuations of $\nabla u_{\e,f}$ are not described by the usual two-scale expansion:
the Gaussian limit of the centered rescaled observable $\e^{-d/2}\int_{\R^d}g\cdot(\nabla u_{\e,f}-\expec{\nabla u_{\e,f}})$ differs from the limit of $\e^{-d/2}\int_{\R^d}g\cdot\nabla\varphi_i^1(\tfrac\cdot\e)\nabla_i\bar u_f^1$.
In~\cite{DGO1}, with Gloria and Otto,
we establish the first \emph{pathwise} theory of fluctuations, providing an intrinsic description of fluctuations based on a suitably modified notion of two-scale expansion (see also the related heuristics in~\cite{GuM}). More precisely, fluctuations of~$\nabla u_{\e,f}$ coincide to first order with fluctuations of some deterministic Helmholtz projection of the corresponding {\it homogenization commutator}
\[\Xi_\e^1[\nabla u_{\e,f}]\,:=\,(\Aa(\tfrac\cdot\e)-\bar\Aa^1)\nabla u_{\e,f},\]
for which
the usual two-scale expansion is shown to be accurate: leading-order fluctuations are then governed by the so-called {\it standard} homogenization commutator,
\begin{equation*}
\Xi^{\circ,1}_\e[\nabla\bar u^1_f]\,:=\,(\Aa(\tfrac\cdot\e)-\bar\Aa^1)(\nabla\varphi_i^1(\tfrac\cdot\e)+\ee_i)\nabla_i\bar u_f^1.
\end{equation*}
We emphasize that the two-scale expansion of $\nabla u_{\e,f}$ via its commutator differs from the naive two-scale expansion~\eqref{eq:2sc}.

\medskip
As we showed in~\cite{DO1}, this theory is naturally extended to higher orders. For $0\le n\le\ell$ we define the following $n$th-order homogenization commutator, as inspired by the higher-order constitutive law~\eqref{eq:const-law-Aepsn},
\begin{equation}\label{eq:Xihomog}
\Xi^n_\e[\nabla u_{\e,f}]\,:=\,(\Aa(\tfrac\cdot\e)-\bar\Ac_\e^n)\nabla u_{\e,f}.
\end{equation}
We first recall that higher-order fluctuations of $\nabla u_{\e,f}$ coincide with fluctuations of some deterministic Helmholtz projection of this commutator; see~\cite[proof of Proposition~3.2]{DO1}.

\begin{lem}[Reduction to commutators; \cite{DO1}]\label{lem:homog-redE}
For all $1\le n\le\ell$ and $f,g\in C^\infty_c(\R^d)^d$, we have
\begin{multline}\label{eq:ident-nabu-XiE}
\int_{\R^d}g\cdot\big(\nabla u_{\e,f}-\expec{\nabla u_{\e,f}}\big)-\int_{\R^d}(\nabla\bar\Oc_\e^{n})^*[g]\cdot\big(\Xi_\e^n[\nabla u_{\e,f}]-\expec{\Xi_\e^n[\nabla u_{\e,f}]}\!\big)\\
=\,\int_{\R^d}\bigg(\sum_{k=2}^n\sum_{l=n+2-k}^n\e^{k+l-2}\bar\Aa_{i_1\ldots i_{k-1}}^{*,k}\nabla\nabla^{k-1}_{i_1\ldots i_{k-1}}\tilde v_g^l\bigg)\cdot\big(\nabla u_{\e,f}-\expec{\nabla u_{\e,f}}\big),
\end{multline}
which entails for all $q<\infty$,
\begin{align*}
&\E\bigg[\Big|\e^{-\frac d2}\int_{\R^d} g\cdot(\nabla u_{\e}-\expec{\nabla u_{\e}})\\
&\hspace{3cm}-\e^{-\frac d2}\int_{\R^d}(\nabla\bar\Oc_\e^{n})^*[g]\cdot\big(\Xi_\e^n[\nabla u_{\e,f}]-\expec{\Xi_\e^n[\nabla u_{\e,f}]}\!\big)\Big|^q\bigg]^\frac1q
\,\lesssim_{f,g,p}\,\e^n.
\qedhere
\end{align*}
\end{lem}

Next, we recall the definition of the $n$th-order standard commutator as the suitable two-scale expansion of the commutator $\Xi_\e^n[\nabla u_{\e,f}]$.
It is obtained by inserting the $n$th-order two-scale expansion $\Fc_\e^n[\cdot]$ into $\Xi_\e^n[\nabla\cdot]$ and by truncating the obtained differential operator to order $n$; see~\cite[Section~3]{DO1}.

\begin{defin}[Standard higher-order homogenization commutators]\label{def:Xi0n}
For \mbox{$0\le n\le\ell$}, the $n$th-order standard commutator $\Xi_\e^{\circ,n}[\nabla\cdot]$ is an $n$th-order differential operator with $\e$-rescaled stationary random coefficients in $H^{1-n}_\loc(\R^d;\Ld^2(\Omega))$, expressed in terms of correctors $\{\varphi^k\}_{1\le k\le n}$. Applied to $\bar w\in C^\infty_c(\R^d)$, it is defined by
\[\Xi_\e^{\circ,n}[\nabla\bar w](x)\,:=\,\Xi_\e^n\big[\nabla\Fc_\e^n[T_x^n\bar w]\big](x),\]
where $T_x^n\bar w$ is the $n$th-order Taylor polynomial of $\bar w$ at $x$.
Note that $\E[{\Xi_\e^{\circ,n}[\nabla\cdot]}]=0$.
\end{defin}

Our main result in~\cite[Corollary~1]{DO1} then shows that the two-scale expansion for commutators $\Xi_\e^{n}[\nabla u_{\e,f}]-\E[{\Xi_\e^{n}[\nabla u_{\e,f}]}]\approx \Xi_\e^{\circ,n}[\nabla\bar\Oc_\e^n[f]]$ is accurate to order $O(\e^n)$ in the fluctuation scaling. Due to the maximal number $\ell=\lceil\frac d2\rceil$ of stationary ``strong'' correctors, this saturates at order $O(\e^{d/2})$, which is expected to be optimal as fluctuations become non-Gaussian at that order~\cite{DO1}.

\begin{prop}[Two-scale expansion of commutators; \cite{DO1}]
For all $f,g\in C^\infty_c(\R^d)^d$ and~$q<\infty$, we have
\begin{align*}
&\expecM{\Big|\e^{-\frac d2}\int_{\R^d}g\cdot\big(\Xi_\e^\ell[\nabla u_{\e,f}]-\expecm{\Xi_\e^\ell[\nabla u_{\e,f}]}\big)-\e^{-\frac d2}\int_{\R^d}g\cdot\Xi_\e^{\circ,\ell}[\nabla\bar\Oc_\e^\ell[f]]\Big|^q}^\frac1q\\
&\hspace{7cm}\,\lesssim_{f,g,q}\,\e^\frac d2\times
\left\{\begin{array}{lll}
1&:&\text{$d$ odd},\\
\Log^\frac12&:&\text{$d$ even}.
\end{array}\right.\qedhere
\end{align*}
\end{prop}

\subsection{Proof of Theorem~\ref{th:fluct0}}

We split the proof into two steps.

\medskip
\step1 Reduction to commutators: for all $1\le n\le\ell$, $p>1$, and $\delta>0$,
\begin{multline}\label{eq:decomp-nabu-Xin0}
\bigg|\expecM{X\int_{\R^d}g\cdot\big(\nabla u_{\e,f}-\expec{\nabla u_{\e,f}}\!\big)}-\expecM{X\int_{\R^d}(\nabla\bar\Oc_\e^{n})^*[g]\cdot\big(\Xi_\e^n[\nabla u_{\e,f}]-\expec{\Xi_\e^n[\nabla u_{\e,f}]}\!\big)}\bigg|\\
\,\lesssim_{p,\delta}\,\e^{n+\frac d{p}}\|DX\|_{\Ld^{p}(\R^d;\Ld^{1+\delta}(\Omega))}\|\langle\nabla\rangle^{2n-1}g\|_{\Ld^{2p'}(\R^d)}\|f\|_{\Ld^{2p'}(\R^d)},
\end{multline}
and similarly, for fluctuations of the flux,
\begin{multline}\label{eq:decomp-nabu-Xin0-flux}
\bigg|\expecM{X\int_{\R^d}g\cdot\big(\Aa(\tfrac\cdot\e)\nabla u_{\e,f}-\expec{\Aa(\tfrac\cdot\e)\nabla u_{\e,f}}\,\big)}\\
-\expecM{X\int_{\R^d}(\Id+\bar\Ac_\e^n\nabla\bar\Oc_\e^{n})^*[g]\cdot\big(\Xi_\e^n[\nabla u_{\e,f}]-\expec{\Xi_\e^n[\nabla u_{\e,f}]}\!\big)}\bigg|\\
\,\lesssim_{p,q}\,\e^{n+\frac d{p}}\|DX\|_{\Ld^{p}(\R^d;\Ld^{1+\delta}(\Omega))}\|\langle\nabla\rangle^{3n-2}g\|_{\Ld^{2p'}(\R^d)}\|f\|_{\Ld^{2p'}(\R^d)}.
\end{multline}
Noting that $\Aa(\tfrac\cdot\e)\nabla u_{\e,f}=\Xi_\e^n[\nabla u_{\e,f}]+\bar\Ac_\e^n\nabla u_{\e,f}$, the claim~\eqref{eq:decomp-nabu-Xin0-flux} follows from~\eqref{eq:decomp-nabu-Xin0}, and it remains to prove the latter.
Given a test function $X\in\Sc(\Omega)$, starting from identity~\eqref{eq:ident-nabu-XiE} and appealing to the Helffer--Sjöstrand identity, see~Lemma~\ref{lem:Mall}, together with Hölder's inequality, we find for all $1\le p,q\le\infty$,
\begin{align}
\expecM{X\int_{\R^d}g\cdot\big(\nabla u_{\e,f}-\expec{\nabla u_{\e,f}}\!\big)}&-\expecM{X\int_{\R^d}(\nabla\bar\Oc_\e^{n})^*[g]\cdot\big(\Xi_\e^n[\nabla u_{\e,f}]-\expec{\Xi_\e^n[\nabla u_{\e,f}]}\!\big)}\nonumber\\
&\hspace{-2cm}=~\expecM{X\int_{\R^d}H_\e^n\cdot\big(\nabla u_{\e,f}-\expec{\nabla u_{\e,f}}\!\big)}\nonumber\\
&\hspace{-2cm}\le~\|DX\|_{\Ld^{p}(\R^d;\Ld^{q}(\Omega))}\Big\|\int_{\R^d}H_\e^n\cdot D\nabla u_{\e,f}\Big\|_{\Ld^{p'}(\R^d;\Ld^{q'}(\Omega))},\label{eq:decomp-nabu-Xin-flct}
\end{align}
where we have set for abbreviation
\[H_\e^n\,:=\,\sum_{k=2}^n\sum_{l=n+2-k}^n\e^{k+l-2}\bar\Aa_{i_1\ldots i_{k-1}}^{*,k}\nabla\nabla^{k-1}_{i_1\ldots i_{k-1}}\tilde v^l_g.\]
Consider the unique almost sure gradient solution $\nabla h_\e^n$ in $\Ld^2(\R^d)^d$ of the auxiliary equation
\[-\nabla\cdot\Aa^*(\tfrac\cdot\e)\nabla h_\e^n=\nabla\cdot H_\e^n,\qquad\text{in $\R^d$},\]
and note that taking the Malliavin derivative in the equation for $\nabla u_{\e,f}$ yields
\[-\nabla\cdot\Aa(\tfrac\cdot\e)\nabla D_z u_{\e,f}=\nabla\cdot (D_z\Aa(\tfrac\cdot\e))\nabla u_{\e,f}.\]
Then computing
\[\int_{\R^d}H_\e^n\cdot D_z\nabla u_{\e,f}\,=\,-\int_{\R^d}\nabla h_\e^n\cdot \Aa(\tfrac\cdot\e) D_z\nabla u_{\e,f}
\,=\,\int_{\R^d}\nabla h_\e^n\cdot (D_z\Aa(\tfrac\cdot\e))\nabla u_{\e,f},\]
and recalling that $|D_z\Aa(\frac x\e)|\lesssim|\calC_0(z-\frac x\e)|$, we obtain from Hölder's inequality and from the integrability condition~\eqref{eq:cov-L1},
\begin{eqnarray*}
\lefteqn{\Big\|\int_{\R^d}H_\e^n\cdot D\nabla u_{\e,f}\Big\|_{\Ld^{p'}(\R^d;\Ld^{q'}(\Omega))}}\\
&\lesssim&\bigg(\int_{\R^d}\Big|\int_{\R^d}[\calC_0]_\infty(\cdot-\tfrac x\e)\|[\nabla h_\e^n\nabla u_{\e,f}]_{1;\e}(x)\|_{\Ld^{q'}(\Omega)}\,dx\Big|^{p'}\bigg)^\frac1{p'}\\
&\lesssim&\e^{\frac d{p}}\|[\nabla h_\e^n\nabla u_{\e,f}]_{1;\e}\|_{\Ld^{p'}(\R^d;\Ld^{q'}(\Omega))}.
\end{eqnarray*}
Using again Hölder's inequality, together with the annealed $\Ld^p$ regularity of Theorem~\ref{th:CZ-ann}, this becomes for $p,q>1$,
\begin{multline*}
{\Big\|\int_{\R^d}H_\e^n\cdot D\nabla u_{\e,f}\Big\|_{\Ld^{p'}(\R^d;\Ld^{q'}(\Omega))}}
\,\lesssim\,\e^{\frac d{p}}\|[\nabla h_\e^n]_{2;\e}\|_{\Ld^{2p'}(\R^d;\Ld^{2q'}(\Omega))}\|[\nabla u_{\e,f}]_{2;\e}\|_{\Ld^{2p'}(\R^d;\Ld^{2q'}(\Omega))}\\
\,\lesssim_{p,q}\,\e^{\frac d{p}}\|H_\e^n\|_{\Ld^{2p'}(\R^d)}\|f\|_{\Ld^{2p'}(\R^d)}.
\end{multline*}
Inserting this into~\eqref{eq:decomp-nabu-Xin-flct}, and applying the standard constant-coefficient $\Ld^p$ regularity for higher-order homogenized equations, cf.~\eqref{eq:reg-high-hom}, the claim~\eqref{eq:decomp-nabu-Xin0} follows.

\medskip
\step2 Two-scale expansion of commutators: for all $1\le n\le\ell$, $1<p<\frac {2d}{2d-1}$, and~$\delta>0$,
\begin{align}\label{eq:bnd-DXiXi0ncl}
&\bigg|\expecM{X\int_{\R^d}g\cdot\Big(\big(\Xi_\e^n[\nabla u_{\e,f}]-\expec{\Xi_\e^n[\nabla u_{\e,f}]}\!\big)-\Xi_\e^{\circ,n}[\nabla\bar\Oc_\e^n[f]]\Big)}\bigg|\nonumber\\
&\hspace{0.5cm}\,\lesssim_{p,\delta}\,\e^{\frac dp+n}\mu_{d,n}(\tfrac1\e)\,\|DX\|_{\Ld^p(\R^d;\Ld^{1+\delta}(\Omega))}\nonumber\\
&\hspace{2cm}\times\|\mu_{d,n}\langle\nabla\rangle^{n}g\|_{(\Ld^4\cap\Ld^{2p'})(\R^d)}\|\mu_{d,n}\langle\nabla\rangle^{3n-2}f\|_{(\Ld^4\cap\Ld^{2p'})(\R^d)},
\end{align}
where the weight $\mu_{d,n}$ is given by $\mu_{d,n}=1$ if $n<\ell$, and by $\mu_{d,n}=\mu_d$ if $n=\ell$.
Combined with the result~\eqref{eq:decomp-nabu-Xin0}--\eqref{eq:decomp-nabu-Xin0-flux} of Step~1, and with the standard constant-coefficient $\Ld^p$ regularity for higher-order homogenized equations, cf.~\eqref{eq:reg-high-hom}, this yields the conclusion.

\medskip\noindent
Recalling that $\expec{\Xi_\e^{\circ,n}[\nabla\bar\Oc_\e^n[f]]}=0$, appealing to the Helffer--Sjöstrand identity, see Lemma~\ref{lem:Mall}(iii), together with Hölder's inequality, we find for all $1\le p,q\le\infty$,
\begin{multline*}
\bigg|\expecM{X\int_{\R^d}g\cdot\Big(\big(\Xi_\e^n[\nabla u_{\e,f}]-\expec{\Xi_\e^n[\nabla u_{\e,f}]}\!\big)-\Xi_\e^{\circ,n}[\nabla\bar\Oc_\e^n[f]]\Big)}\bigg|\\
\,\le\,\|DX\|_{\Ld^p(\R^d;\Ld^q(\Omega))}\Big\|\int_{\R^d}g\cdot D\big(\Xi^n_\e[\nabla u_{\e,f}]-\Xi^{\circ,n}_\e[\nabla\bar\Oc_\e^{n}[f]]\big)\Big\|_{\Ld^{p'}(\R^d;\Ld^{q'}(\Omega))}.
\end{multline*}
Repeating the analysis in~\cite[Theorem~1(ii)]{DO1} yields the following estimate for the last right-hand side term: for all $1\le n\le\ell$, $1<p<\frac d{n-1}$, and~$q>1$,
\begin{multline*}
\Big\|\int_{\R^d}g\cdot D\big(\Xi_\e^n[\nabla u_{\e,f}]-\Xi_\e^{\circ,n}[\nabla\bar\Oc_\e^nf]\big)\Big\|_{\Ld^{p'}(\R^d;\Ld^{q'}(\Omega))}\\
\,\lesssim_{p,q}\,\e^{\frac dp+n}\mu_{d,n}(\tfrac1\e)
\|\mu_{d,n}(|\cdot|)\langle\nabla\rangle^{n}g\|_{(\Ld^\frac{2dp'}{d+(n-1)p'}\cap\Ld^{2p'})(\R^d)}\\
\times\|\mu_{d,n}(|\cdot|)\langle\nabla\rangle^{2n-1}(f,\nabla\bar\Oc_\e^nf)\|_{(\Ld^\frac{2dp'}{d+(n-1)p'}\cap\Ld^{2p'})(\R^d)}.
\end{multline*}
Appealing to the standard (weighted) constant-coefficient $\Ld^p$ regularity for higher-order homogenized equations, cf.~\eqref{eq:reg-high-hom}, the claim~\eqref{eq:bnd-DXiXi0ncl} follows.
\qed

\section*{Acknowledgements}
The author thanks Antoine Gloria, Marius Lemm, François Pagano, and Felix Otto for many motivating and inspiring discussions around the Bourgain--Spencer conjecture,
and acknowledges financial support from the CNRS-Momentum program, from the F.R.S.-FNRS, as well as from the European Union (ERC, PASTIS, Grant Agreement n$^\circ$101075879).\footnote{{Views and opinions expressed are however those of the author only and do not necessarily reflect those of the European Union or the European Research Council Executive Agency. Neither the European Union nor the granting authority can be held responsible for them.}}

\bibliographystyle{plain}
\bibliography{biblio}

\end{document}